\documentclass{amsart}

\newcommand{\figex}[1]{#1}

\usepackage[english]{babel}
\usepackage{fullpage}
\usepackage{hyperref}

\makeatletter
\renewcommand\section{\@startsection
	{section}% nom du titre
{1}% niveau de titre
{0pt}% indentation
{-3.5ex plus -1ex minus -.2ex}%espace vertical avant 
{2.3ex plus.2ex}% espace vertical après 
{\centering\normalfont\Large\scshape}}

\renewcommand\subsection{\@startsection
	{subsection}% nom du titre
{2}% niveau de titre
{0pt}% indentation
{-3ex plus -1ex minus -.2ex}%espace vertical avant 
{1ex plus.2ex}% espace vertical après 
{\normalfont\large\bfseries}}

\renewcommand\subsubsection{\@startsection
	{subsubsection}% nom du titre
{3}% niveau de titre
{0pt}% indentation
{-1.5ex plus -1ex minus -.2ex}%espace vertical avant 
{0.8ex plus .2ex}% espace vertical après 
{\normalfont\bfseries}}

\renewcommand\paragraph{\@startsection
	{paragraph}% nom du titre
{4}% niveau de titre
{0em}% indentation
{-1.2ex plus -0.4ex minus -.2ex}%espace vertical avant 
{0ex}% espace vertical après 
{\normalfont\bfseries}}

\makeatother

\usepackage{latexsym,amsfonts,amsmath,amssymb,amsthm,amscd,stmaryrd}

\usepackage{bbm}
\usepackage{yhmath}

\usepackage[dvips]{epsfig}
\usepackage{graphicx,graphics,psfrag}
\usepackage{float}
\usepackage{color}
\usepackage[table]{xcolor}
\usepackage{subfigure}
\usepackage{graphics,psfrag}

\definecolor{fondrouge}{rgb}{1,0.3,0.3}
\definecolor{fondvert}{rgb}{0.3,1,0.5}

\usepackage{comment}

\newcommand{\card}{\textrm{card}}

\newcommand{\Dspace}{\textrm{Dspace}}
\newcommand{\Dtime}{\textrm{Dtime}}

\newcommand{\Base}{\mathbf{B}}

\newcommand{\G}{\mathbbm{G}}
\renewcommand{\H}{\mathbbm{H}}
\newcommand{\Z}{\mathbbm{Z}}

\newcommand{\N}{\mathbbm{N}}
\newcommand{\M}{\mathbbm{M}}
\newcommand{\U}{\mathbbm{U}}
\newcommand{\BB}{\mathbbm{B}}

\newcommand{\supp}{\mathrm{supp}}

\newcommand{\A}{\mathcal{A}}
\newcommand{\B}{\mathcal{B}}

\newcommand{\cF}{\mathcal{F}}

\newcommand{\T}{\mathbf{T}}
\newcommand{\TT}[1]{\mathbf{T}\left(#1\right)}
\newcommand{\Lang}{\mathcal{L}}
\newcommand{\Vect}[1]{\left\langle#1\right\rangle_{\Z}}
\newcommand{\folder}{\textrm{Folder}}

\newcommand{\s}{\sigma}

\newcommand{\sft}{\mathcal{SFT}}
\newcommand{\sofic}{\mathcal{S}ofic}

\newcommand{\TM}{\mathcal{M}}

\newcommand{\sa}[2]{\mathbf{SA}_{#1}\left(#2\right)}
\newcommand{\aprox}[2]{\mathbf{Aprox}_{#1}\left(#2\right)}
\newcommand{\ft}[2]{\mathbf{FT}_{#1}\left(#2\right)}

\newcommand{\id}{\mathrm{Id}}

\newcommand{\vect}[1]{\mathbf{#1}}
\newcommand{\trans}[1]{\textrm{Trans}_{#1}}
\newcommand{\proj}[1]{\textrm{Proj}_{#1}}

\theoremstyle{plain}
\newtheorem{theorem}{Theorem}[section]
\newtheorem{proposition}[theorem]{Proposition}

\newtheorem{corollary}[theorem]{Corollary}

\theoremstyle{definition}
\newtheorem*{definition}{Definition}
\newtheorem*{notation}{Notation}

\theoremstyle{remark}
\newtheorem*{remark}{Remark}
\newtheorem{example}{Example}[section]

\newcounter{claimcount}[theorem]
\newcommand{\THMfont}[1]{{\sl #1}}
\newcommand{\Claim}[1]{\refstepcounter{claimcount} \vspace{0.3em}              
\noindent {\sc Claim \theclaimcount: \ }\THMfont{ #1}}
\newcommand{\bprf}[1][Proof:]{\begin{list}{}    {\setlength{\leftmargin}{0.5em}
\setlength{\rightmargin}{0em}  \setlength{\listparindent}{1em}}   \item {\em
\hspace{-1em}  #1  }}
\newcommand{\eprf}{\end{list}}

\newcommand{\bclaimprf}{\bprf}

\newcommand{\eclaimprf}{ \hfill $\Diamond$~{\scriptsize {\tt
Claim~\theclaimcount}}\eprf} % %

\usepackage{marvosym,fancybox}

\usepackage{tikz}
\usepackage{tikz-3dplot}
\tikzset{math3d/.style= {x={(1cm,0cm)},y= {(0.8cm,0.6cm)}, z={(0cm,1cm)}}}
\tikzset{math3dfig/.style= {x={(1cm,0cm)},y= {(-0.706cm,-0.5cm)}, z={(0cm,1cm)}}}
\usepgflibrary{decorations.pathmorphing}
\usetikzlibrary{decorations.pathmorphing}
\usetikzlibrary{chains,fit,shapes}
\usetikzlibrary{patterns}
\usetikzlibrary{calc}
\usepackage{xlop}
\usepackage{ifthen}

\definecolor{red}{rgb}{1,0.3,0.3}
\definecolor{green}{rgb}{0,0.7,0.4}
\newcommand{\block}[3]{
\draw [dashed] (#1,#2) rectangle (#1+1,#2+1);
\draw (#1+0.5,#2+0.5) node{${#3}$};
}

\newcommand{\blockl}[2]{
\draw [dashed, fill=black] (#1,#2) rectangle (#1+1,#2+1);
}

\newcommand{\blocku}[3]{
\draw [dashed, fill=black!5] (#1,#2) rectangle (#1+1,#2+1);
\draw (#1+0.5,#2+0.5) node{${#3}$};
}

\newcommand{\blockd}[3]{
\draw [dashed, fill=black!25] (#1,#2) rectangle (#1+1,#2+1);
\draw (#1+0.5,#2+0.5) node{${#3}$};
}

\newcommand{\blockblue}[3]{
\fill[blue, opacity=0.3] (#1,#2) rectangle (#1+1,#2+1);
\draw [dashed] (#1,#2) rectangle (#1+1,#2+1);
\draw (#1+0.5,#2+0.5) node{${#3}$};
}

\newcommand{\blockred}[3]{
\fill[red, opacity=0.2] (#1,#2) rectangle (#1+1,#2+1);
\draw [dashed] (#1,#2) rectangle (#1+1,#2+1);
\draw (#1+0.5,#2+0.5) node{${#3}$};
}

\newcommand{\blockgreen}[3]{
\draw [dashed, fill=green!25] (#1,#2) rectangle (#1+1,#2+1);
\draw (#1+0.5,#2+0.5) node{${#3}$};
}

\newcommand{\blocktm}[3]{
\draw [dashed, fill=black!5] (#1,#2) rectangle (#1+1,#2+1);
\draw (#1+0.5,#2+0.5) node{${#3}$};
}

\newcommand{\blocktmhead}[4]{
\draw [dashed, fill=black!5] (#1,#2) rectangle (#1+1,#2+1);
\draw (#1,#2) -- (#1+1,#2+1);
\draw (#1+0.25,#2+0.75) node{\tiny{$#3$}};
\draw (#1+0.75,#2+0.25) node{\tiny{$#4$}};
}

\newcommand{\inside}[1]{\vbox to 9pt{\hbox{
\begin{tikzpicture}[scale=0.4]
#1
\end{tikzpicture}
}}}

\newcommand{\insidebis}[1]{\vbox to 15pt{\hbox{
\begin{tikzpicture}[scale=0.4]
#1
\end{tikzpicture}
}}}

\newcommand{\hilbertb}[4]{
\begin{scope}[xshift=#1cm,yshift=#2cm,rotate=#3,xscale=#4]
\draw (-0.5,-0.5) rectangle (0.5,0.5); 
\draw[gray!30] (0,-0.5)-- (0,0.5);
\draw[gray!30] (0.5,0)-- (-0.5,0);
\draw[blue,>=latex,->] (-0.25,-0.5)--(-0.25,0.25)--(0.25,0.25)--(0.25,-0.25)--(0.5,-0.25);
\end{scope}}

\newcommand{\hilbertl}[4]{
\begin{scope}[xshift=#1cm,yshift=#2cm,rotate=#3,xscale=#4]
\draw (-0.5,-0.5) rectangle (0.5,0.5); 
\draw[gray!30] (0,-0.5)-- (0,0.5);
\draw[gray!30] (0.5,0)-- (-0.5,0);
\draw[blue,>=latex,->] (-0.5,-0.25)--(-0.25,-0.25)--(-0.25,0.25)--(0.25,0.25)--(0.25,-0.5);
\end{scope}}

\newcommand{\hilbertlr}[4]{
\begin{scope}[xshift=#1cm,yshift=#2cm,rotate=#3,xscale=#4]
\draw (-0.5,-0.5) rectangle (0.5,0.5); 
\draw[gray!30] (0,-0.5)-- (0,0.5);
\draw[gray!30] (0.5,0)-- (-0.5,0);
\draw[blue,>=latex,->] (-0.5,-0.25)--(-0.25,-0.25)--(-0.25,0.25)--(0.25,0.25)--(0.25,-0.25)--(0.5,-0.25);
\end{scope}}

\newcommand{\hilbertbtwo}[4]{
\begin{scope}[xshift=#1cm,yshift=#2cm,rotate=#3,xscale=#4]
\hilbertlr{-0.5}{-0.5}{-90}{-1}
\hilbertb{-0.5}{0.5}{0}{1}
\hilbertl{0.5}{0.5}{0}{1}
\hilbertl{0.5}{-0.5}{90}{-1}
\end{scope}}

\newcommand{\hilbertltwo}[4]{
\begin{scope}[xshift=#1cm,yshift=#2cm,rotate=#3,xscale=#4]
\hilbertb{-0.5}{-0.5}{-90}{-1}
\hilbertb{-0.5}{0.5}{0}{1}
\hilbertl{0.5}{0.5}{0}{1}
\hilbertlr{0.5}{-0.5}{90}{-1}
\end{scope}}

\newcommand{\hilbertlrtwo}[4]{
\begin{scope}[xshift=#1cm,yshift=#2cm,rotate=#3,xscale=#4]
\hilbertb{-0.5}{-0.5}{-90}{-1}
\hilbertb{-0.5}{0.5}{0}{1}
\hilbertl{0.5}{0.5}{0}{1}
\hilbertl{0.5}{-0.5}{90}{-1}
\end{scope}}

\newcommand{\hilbertbthree}[4]{
\begin{scope}[xshift=#1cm,yshift=#2cm,rotate=#3,xscale=#4]
\hilbertlrtwo{-1}{-1}{-90}{-1}
\hilbertbtwo{-1}{1}{0}{1}
\hilbertltwo{1}{1}{0}{1}
\hilbertltwo{1}{-1}{90}{-1}
\end{scope}}

\newcommand{\hilbertlthree}[4]{
\begin{scope}[xshift=#1cm,yshift=#2cm,rotate=#3,xscale=#4]
\hilbertbtwo{-1}{-1}{-90}{-1}
\hilbertbtwo{-1}{1}{0}{1}
\hilbertltwo{1}{1}{0}{1}
\hilbertlrtwo{1}{-1}{90}{-1}
\end{scope}}

\newcommand{\hilbertlrthree}[4]{
\begin{scope}[xshift=#1cm,yshift=#2cm,rotate=#3,xscale=#4]
\hilbertbtwo{-1}{-1}{-90}{-1}
\hilbertbtwo{-1}{1}{0}{1}
\hilbertltwo{1}{1}{0}{1}
\hilbertltwo{1}{-1}{90}{-1}
\end{scope}}

\title[Speed of convergence for the realization of a subshift]{Speed of convergence for the realization of an effective subshift by a multidimensional SFT or Sofic}
\date{}
\author{Anne Crumi\`ere}
\address{ }
\email{crumierepro@gmail.com}
\urladdr{}

\author{Mathieu Sablik}
\address{LATP, Aix-Marseille Universit\'e, 39, rue F. Joliot Curie,13453 Marseille Cedex 13, France }
\email{sablik@latp.univ-mrs.fr}
\urladdr{http://www.latp.univ-mrs.fr/$\sim$sablik/}

\author{Michael Schraudner}
\address{CMI, Universidad de Chile }
\email{mschraudner@dim.uchile.cl}


\newcommand{\define}{\emph}

\begin{document}

%\end{document}
\maketitle
\begin{abstract}
Realization of $d$-dimensional effective subshifts as projective sub-actions of $d+d'$-dimensional sofic subshifts for $d'\geq 1$ is now well know~\cite{Hochman-2009,Durand-Romashchenko-Shen-2010,Aubrun-Sablik-2010}. In this paper we are interested in the speed of convergence of this realization. That is to say given an effective subshift $\Sigma$ realized as projective sub-action of a sofic $\T$, we study the function which on input an integer $k$ returns the smallest width of the strip which verify the local rules of $\T$ necessary to obtain exclusively the language of size $k$ of $\Sigma$ in the central row of the strip. We study this topological conjugacy invariant for effective subshifts in order to exhibit algorithmic properties of these subshifts.   
\end{abstract}

\section*{Introduction}
A subshift of dimension $d$ is a closed and shift-invariant subset of $\A^{\Z^d}$ where $\A$ is a finite alphabet. A subshift can be characterized by a set of forbidden patterns. With this last point of view, the simplest class is the set of subshifts of finite type, which are subshifts that can be characterized by a finite set of forbidden patterns. Applying a continuous shift-invariant function on a subshift of finite type, one obtains a sofic shift. This class of symbolic systems is the smallest class stable under factors which contains the class of subshift of finite type. In dimension one, this class can be  characterized by a set of forbidden pattern accepted by a finite automaton~\cite{Weiss-1973}.

When we consider multidimensional subshifts, we can consider their stability according to another dynamical operation: tacking a sub-action. To obtain a subshift we consider projective subaction which consists of restricting the configurations of a subshift of $\A^{\Z^d}$ to a sublattice of $\Z^d$, we obtain a subshift of $\A^{\Z^{d'}}$ where $d'<d$ is the dimension of the sublattice. The smallest class stable under this operation which contains the class of sofic is the set of effective subshifts which are subshifts that can be characterized by a set of forbidden patterns enumerated by a Turing machine. Clearly this class is stable by projective subaction and contains all sofic subshifts. Moreover a consequence of the main result of~\cite{Hochman-2009} states that every $d$-dimensional effective subshift can be obtained via projective subaction of a $d+2$-dimensional sofic. This result was improved independently in~\cite{Durand-Romashchenko-Shen-2010,Aubrun-Sablik-2010} to hold for projective subactions of $d+1$-dimensional sofics.

The three classes evoked below are stable by conjugacy and are characterized by computability property on their set of forbidden patterns. In this article, we introduce new conjugacy invariant classes based on the speed of convergence of the realization via projective subaction. This emphasizes algorithmic complexity of the set of forbidden patterns. 

In~\cite{Pavlov-Schraudner-2010}, the authors characterize which one-dimensional sofic subshifts can be obtained by a projective subaction of subshift of finite type. This classification shows a difference between certain type of sofic subshifts, according to wether their realization can be stable or unstable that is to say wether
a bounded strip around the central one is necessary to obtain the desired sofic subshift or wether there is no bounds which guarantee to the central row to be in the subshift. This approach is inspired by the notion of limit sets of cellular automata where a such that notion of stable and unstable limit-set appears~\cite{Maass-1993,Maass-1995}.

In this article, we would like to go beyond the dichotomy stable vs unstable realization and try to quantify this notion. Thus we introduce the notion of speed of convergence of the realization of an effective subshift $\Sigma$ by projective subaction of a subshift of finite type or a sofic. This is defined as the function which, for a given integer $k$, returns the width of the strip necessary to obtain the language of  the effective subshift up to a word length of size $k$ in the central rows.

Modulo an equivalence relation this quantity is invariant under conjugacy both for subshifts of finite type and sofic subshifts (see Section~\ref{section:InvariantOfConjugacy}) when we look at the speed of convergence for the projective subaction according to a given direction. Using this invariant we are able to show that a necessary condition to have two conjugate substitutive subshifts is that the substitution rules of both system must have the same size.

In another way, given an effective subshift, one can study the set of speed of convergence which realizes it as projective subaction of subshift of finite type or sofic. Modulo an equivalence relation this set is invariant under conjugacy (Sections~\ref{section:InvariantOfConjugacyBis}). In Section~\ref{section:RealizationSFTvsSofic} we exhibit a condition on effective subshift which forces sofic realization to be equivalent to subshift of finite type realization. 

In Section~\ref{section:GeneralConstruction} we compare the different constructions of realization of an effective subshift by projective subaction of a sofic and we propose a quicker construction if the effective subshift has a periodic point. When the dimension of the sofic increase the convergence is quicker. These results give upper bounds for realization by sofic, but is also possible to obtain lower bounds (see Section~\ref{section:LowerBound}). These results bring out some examples of different classes which exhibit the optimality of the different previous results (see Section~\ref{section:example}).

\section{Definitions}

\subsection{Subshifts}

Let $\A$ be a finite alphabet and let $d$ be a positive integer. A \define{configuration} $x$ is an element of $\A^{\Z^d}$. Let $\U$ be a finite subset of $\Z^d$, denote $x_{\U}$ the \define{restriction} of $x$ to $\U$. A \define{$d$-dimensional pattern} is an element $p\in\A^{\U}$ where $\U\subset\Z^d$ is finite, $\U$ is the \define{support} of $p$, which is denoted by $\supp(p)$. Denote by $\A^{\ast}$ the set of $d$-dimensional patterns. A pattern $p$ of support $\U\subset\Z^d$ \define{appears} in a configuration $x$ if there exists $\vect{i}\in\Z^d$ such that $p=x_{\vect{i}+\U}$, this is denoted by $p\sqsubset x$.

Endowing $\A$ with the discrete topology, one  considers the product topology on $\A^{\Z^d}$. For this topology, $\A^{\Z^d}$ is a compact metric space on which $\Z^d$ acts by translation via the shift map $\s$ defined for all $i\in\Z^d$ by:
$$\begin{array}{ccccc}
\s^i:&\A^{\Z^d} & \longrightarrow &\A^{\Z^d}&\\
&x&\longmapsto & \s^i(x)& \textrm{ such that } \s^i(x)_u=x_{i+u} \ \forall u\in\Z^d.
\end{array}$$

The $\Z^d$-dynamical system $(\A^{\Z^d},\s)$ is called the \define{fullshift}.  A \define{$\Z^d$-subshift} is a $\s$-invariant closed subset of $\A^{\Z^d}$.

Let $\T\subset\A^{\Z^d}$ be a subshift, $\U\subset\Z^d$ be a finite support and $n\in\N$, define:
\begin{itemize}
\item $\Lang_{\U}(\T)=\left\{p\in\A^{\U}:\textrm{ there exists }x\in\T\textrm{ such that } p\sqsubset x\right\}$ the \define{language of $\T$ of support $\U$}; 
\item $\Lang_n(\T)=\Lang_{[0,n-1]^d}(\T)$ the \define{square language of $\T$ of size $n$};
\item $\Lang_{\textrm{rect}}(\T)=\cup_{n\in\N}\Lang_n(\T)$ the \define{square language of $\T$};
\item $\Lang(\T)=\{p\in\A^{\ast}:\exists x\in\T\textrm{ such that }p\sqsubset x\}$ the \define{language of $\T$}.
\end{itemize}
We define the entropy of a $d$-dimensional subshift $\T$ as $$h(\T)=\lim_{n\to\infty}\frac{\log(\card(\Lang_n(\T)))}{n^d}=\inf_n\frac{\log(\card(\Lang_n(\T)))}{n^d}.$$

\subsection{Classes of subshifts}

Let $F$ be a set of patterns, we define the \define{subshift of forbidden patterns} $F$ by:
\[\T_F=\left\{x\in\A^{\Z^d}: \forall p\in F,\, p\not\sqsubset x \right\}.\]

Every subshift can be defined in this way and this allows to define classes of subshifts according to the complexity of $F$. Let $\T$ be a subshift, 
\begin{itemize}
\item if there exists a finite set of forbidden patterns such that $\T=\T_F$ then $\T$ is called a \define{subshift of finite type};
\item if there exists a recursively enumerable set of forbidden patterns such that $\T=\T_F$ then $\T$ is called an \define{effective subshift} (we recall that a set of patterns is recursively enumerable if there exists a Turing machine which enumerates all these patterns).
\end{itemize}

Let  $\T\subset\A^{\Z^d}$ and  $\T'\subset\B^{\Z^d}$ be two subshifts, a \define{morphism} between them is a continuous function $\pi:\T\longrightarrow\T'$ such that $\pi\circ\s^{\vect{i}}=\s^{\vect{i}}\circ \pi$ for all $\vect{i}\in\Z^d$. By Hedlund's Theorem~\cite{Hedlund-1969}, there exists a local function $\overline{\pi}:\Lang_{\U}(\T)\longrightarrow\B$ where $\U\subset\Z^d$ is a finite set called \define{neighborhood} such that $\pi(x)_{\vect{i}}=\overline{\pi}(x_{\vect{i}+\U})$ for all $x\in\A^{\Z^d}$ and $\vect{i}\in\Z^d$. For $\mathbb{V}\subset\Z^d$ and $u\in\A^{\mathbb{V}}$, by extension we denote $\overline{\pi}(u)_{\vect{i}}=\overline{\pi}(u_{\vect{i}+\U})$ for all $\vect{i}\in\Z^d$ such that $\vect{i}+\U\subset\mathbb{V}$. If there is no ambiguity one just denotes $\pi$ for $\overline{\pi}$. The \define{radius} of $\pi$ is the minimal $r\in\N$ such that $\U\subset[-r,r]^d$. If $\U=\{0\}$, the morphism is said \define{letter to letter}

Let $\T\subset\A^{\Z^d}$ be a subshift and let $\pi:\A^{\Z^d}\to\B^{\Z^d}$ be a morphism, then $\pi(\T)\subset\B^{\Z^d}$ is a subshift called the \define{factor subshift} of $\T$ by $\pi$ which is called the \define{factor map}. A subshift $\T$ is called \define{sofic} if there exists a subshift of finite type $\T_{F}$ with forbidden patterns $F$ and a factor map $\pi$ such that $\T=\pi(\T_{F})$. The factor map $\pi$ can be considered letter to letter, that is to say that the neighborhood $\U=\{\vect{0}\}$,. 

Two subshifts $\T$ and $\T'$ are \define{conjugate} if there exists a bijective morphism $\psi:\T\longrightarrow\T'$. The different classes of subshifts defined here (finite type, sofic and effective subshifts) are stable under conjugacy.

\subsection{Operations on subshifts}

\subsubsection{Finite type condition}

Let $\T\subset\A^{\Z^d}$ be a subshift and $F$ be a finite set of $d$-dimensional patterns, we consider the subshift $$\ft{F}{\T}=\left\{x\in\T : \forall p\in F, p \textrm{ does not appear in } x \right\}
.$$

Let $F$ be a finite set of $d$-dimensional patterns. It is clear that if $\T$ is a subshift of finite type (respectively a sofic subshift, an effective subshift) then $\ft{F}{\T}$ is a subshift of finite type (respectively a sofic subshift, an effective subshift).

\subsubsection{Projective subactions}

Let $\G$ be a sublattice of $\Z^d$ finitely generated by $\Base_{\G}=(\vect{u_1},\dots,\vect{u_{d'}})\in\Z^d$ ($d'\leq d)$. We can denote $\G=\Vect{\vect{u_1},\dots,\vect{u_{d'}}}=\{\lambda_1\vect{u_1}+\dots+\lambda_{d'}\vect{u_{d'}}:(\lambda_1,\dots,\lambda_{d'})\in\Z^{d'}\}$. Let $\T\subseteq\A^{\Z^d}$ be a subshift, the \define{projective subdynamic} of $\T$ with respect to $\G$ is the  subshift of dimension $d'$ defined by $\sa{\G}{\T}$ where
\[\begin{array}{rcccl}
\mathbf{SA}_{\G}:&\A^{\Z^d}&\longrightarrow&\A^{\G}&\\
&x&\longmapsto&y &\textrm{such that }  y_{\vect{i}}=x_{\vect{i}}\textrm{ for all } \vect{i}\in\G,.
\end{array}\]

Identifying $\G$ with $\Z^{d'}$ according to the basis $\Base_{\G}=(\vect{u_1},\dots,\vect{u_{d'}})$, it is possible to consider $\sa{\G}{\T}$ as a subshift of $\A^{\Z^{d'}}$ denoted $\sa{\Base_{\G}}{\T}$ where
\[\begin{array}{rcccl}
\mathbf{SA}_{\Base_{\G}}:&\A^{\Z^d}&\longrightarrow&\A^{\Z^{d'}}&\\
&x&\longmapsto&y &\textrm{such that } y_{(i_1,\dots,i_{d'})}=x_{i_1\vect{u_1}+\dots+i_{d'}\vect{u_{d'}}}\textrm{ for all } (i_1,\dots,i_{d'})\in\Z^{d'}.
\end{array}\]

Any effective subshift of dimension $d$ can be obtained with factor and projective subaction operations from a subshift of finite type of dimension $d+1$. More precisely one has the following theorem.

\begin{theorem}\cite{Hochman-2009,Aubrun-Sablik-2010,Durand-Romashchenko-Shen-2010}\label{Th:ProjectiveSubactionSofic}
Let $\Sigma\subset \A^{\Z^d}$ be an effective subshift, then $\widetilde{\Sigma}=\{x\in\A^{\Z^{d+1}}: \exists y\in\Sigma \textrm{ such that }x_{\Z^d\times\{i\}}=y \textrm{ for all }i\in\Z\}$ is sofic.

In particular there exists a subshift of finite type $\T\subset\B^{\Z^{d+1}}$ and a factor map $\pi:\B\to\A$, which can be considered letter to lefter, such that 
\[\sa{(\vect{e_1},\dots,\vect{e_{d}})}{\pi(\T)}=\Sigma.\]
\end{theorem}

%%%%%%%%%%%%%%%%%%%
%
%%%%%%%%%%%%%%%%%%%

\section{Speed of convergence: conjugacy invariant of  for multidimensional SFTs and sofic subshifts}\label{section:InvariantOfConjugacy}

\subsection{Definitions}

\paragraph{Approximation row}

\ In this section we consider $\G$ and $\H$ two sublattices of $\Z^d$ such that $\G\oplus\H=\Z^d$ and consider $\Base_{\G}=(\vect{u_1},\dots,\vect{u_{d'}})$ (respectively $\Base_{\H}=(\vect{v_1},\dots,\vect{v_{d''}})$) a base of $\G$ (respectively of $\H$) of course $d=d'+d''$. For all $n\in\N$, define $\BB_n^{\Base_{\H}}=\left\{k_1\vect{v_1}+\dots+ k_{d''}\vect{v_{d''}}:(k_1,\dots,k_{d''})\in [-n,n]^{d''}\right\}$. Denote $\proj{\U}:\A^{\BB_n^{\Base_{\H}}}\longrightarrow\A^{\U}$ the projection according to the coordinates $\U\subset \BB_n^{\Base_{\H}}$. For $\vect{i}\in\BB_n^{\Base_{\H}}$, $\proj{\{\vect{i}\}}$ is just denoted $\proj{\vect{i}}$.

Let $\T_F\subset\A^{\Z^d}$ be a subshift of finite type defined by the finite set of forbidden patterns $F$. We are interested in the subshift on the sublattice $\G$ which gives a row of size $n$ and which verifies the local condition given by $F$; that is to say we have a subshift indexed by $\G\times\BB_n^{\Base_{\H}}$ where no pattern of $F$ appear. Formally, one defines the \define{$n$-approximation row following $\G$  with the grow along to $\Base_{\H}$} of $\T_F$ by:
$$\T_F^{n,\G,\Base_{\H}}=\left\{x\in\left(\A^{\BB_n^{\Base_{\H}}}\right)^{\G}: \forall p\in F,  p\not\sqsubset (\proj{\vect{j}}(x_{\vect{i}}))_{(\vect{i},\vect{j})\in \G\times \BB_n^{\Base_{\H}}}\right\}.$$

Let $\pi:\A^{\Z^d}\to\B^{\Z^d}$ be a factor map. For $n$ sufficiently large, the function $\mathbf{SA}_{\Base_{\G}}:\A^{\Z^d}\longrightarrow\A^{\Z^{d'}}$ can be mistakenly used as 
\[\begin{array}{rcccl}
\mathbf{SA}_{\Base_{\G}}\circ\pi:&\T_F^{n,\G,\Base_{\H}}&\longrightarrow&\B^{\Z^{d'}}&\\
&x&\longmapsto&y &\textrm{such that } y_{(i_1,\dots,i_{d'})}=\pi(x)_{i_1\vect{u_1}+\dots+i_{d'}\vect{u_{d'}}}\textrm{ for all } (i_1,\dots,i_{d'})\in\Z^{d'}.
\end{array}\]

We verify that $$\sa{\Base_{\G}}{\pi(\T_F)}=\bigcap_{n\in\N}\sa{\Base_{\G}}{\pi\left(\T_F^{n,\G,\Base_{\H}}\right)}.$$

Let $d'<d$ and consider $(\vect{e_1},\dots,\vect{e_d})$ the canonical basis of $\Z^d$. If $\Base_{\G}=(\vect{e_1},\dots,\vect{e_{d'}})$ and $\Base_{\H}=(\vect{e_{d'+1}},\dots,\vect{e_d})$ then we denote $\T_F^{n,d\to d'}$ instead of $\T_F^{n,\G,\Base_{\H}}$.

\paragraph{Speed of convergence} \ 
By definition of $\T_F^{n,\G,\Base_{\H}}$, if $u\in \Lang\left(\sa{\Base_{\G}}{\pi(\T_F)}\right)$, then $u\in\Lang\left(\sa{\Base_{\G}}{\pi\left(\T_F^{n,\G,\Base_{\H}}\right)}\right)$. We want to quantify the reciprocal, that is to say given a $k$, find the smallest $n$ such that $u\notin \Lang_k(\sa{\Base_{\G}}{\pi(\T_F)})\Longrightarrow u\notin\Lang_k\left(\sa{\Base_{\G}}{\pi(\T_F^{n,\G,\Base_{\H}})}\right)$. This allows to quantify when a word is forbidden by the local rules $F$ in the approximation row.

The \define{speed of convergence as sofic} of the cover $\T_F$ with the factor $\pi$ following $\Base_{\G}$ with the growth along to the base $\Base_{\H}$ is the following function:
\[\begin{array}{cccl}
\varphi_{F,\pi,\Base_{\G},\Base_{\H}}:&\N&\longrightarrow &\N \\
&k&\longmapsto &\min\left\{ n\in\N : \forall u\in\A'^{[0,k-1]^{d'}} \textrm{ one has }u\notin \Lang(\sa{\Base_{\G}}{\pi(\T_F)})\Longrightarrow u\notin\Lang\left(\sa{\Base_{\G}}{\pi(\T_F^{n,\G,\Base_{\H}})}\right)\right\}.
\end{array}\]

When $\pi=\id$, we call this function the \define{speed of convergence as SFT} of $\T_F$ following $\Base_{\G}$ with the growth according to the base $\Base_{\H}$.

Clearly $\varphi_{F,\pi,\Base_{\G},\Base_{\H}}$ is non-decreasing since we have the following inclusion
\[ \sa{\Base_{\G}}{\pi(\T_F)}\subset \sa{\Base_{\G}}{\pi(\T_F^{n+1,\G,\Base_{\H}})}\subset\sa{\Base_{\G}}{\pi(\T_F^{n,\G,\Base_{\H}})}.\] 

Denote $\cF$ the set of non-decreasing functions from $\N$ to $\N$. If $\Base_{\G}=(\vect{e_1},\dots,\vect{e_{d'}})$ and $\Base_{\H}=(\vect{e_{d'+1}},\dots,\vect{e_d})$ then we denote $\varphi_{F,\pi,d\to d'}$ instead of $\varphi_{F,\pi,\Base_{\G},\Base_{\H}}$.

\begin{example}

Consider the following set of forbidden patterns
\[F=\left\{\begin{array}{l}
\insidebis{\block{0}{0}{\beta}\blockblue{1}{0}{a}\block{1}{1}{\alpha}},\insidebis{\blockblue{0}{0}{a}\blockblue{1}{0}{a}\block{1}{1}{\beta}},\inside{\blockblue{0}{0}{a}\block{1}{0}{\$}},
\\
\vspace{0.cm}
\\
\insidebis{\blockred{0}{0}{b}\block{1}{0}{\gamma}\block{0}{1}{\alpha}},
\insidebis{\blockred{0}{0}{b}\blockred{1}{0}{b}\block{0}{1}{\gamma}},
\inside{\block{0}{0}{\$}\blockred{1}{0}{b}}
\end{array}
\textrm{ such that } \begin{array}{l}
\inside{\block{0}{1}{\alpha}}\in\left\{\inside{\blockblue{0}{1}{a}},\inside{\blockred{0}{1}{b}}\right\},\\
\inside{\block{0}{1}{\beta}}\in\left\{\inside{\block{0}{1}{\$}},\inside{\blockred{0}{1}{b}}\right\},\\
\inside{\block{0}{1}{\gamma}}\in\left\{\inside{\block{0}{1}{\$}},\inside{\blockblue{0}{1}{a}}\right\}.
\end{array}
\right\}.
\]

If we consider the subshift of finite type $\T_F$, it is easy to realize that $\sa{\vect{e_1}}{\T_F}$ is the subshift where the forbidden patterns are $\left\{\$a^nb^m\$,\$a^nb^ma,ba^nb^m\$,ba^nb^ma:m\ne n\right\}$. The idea is that in a configuration of $\T_F$, if a line contains a forbidden pattern with the word $a^nb^m$ with $n\ne m$ in the center, the next line in the direction $\vect{e_2}$ contains the forbidden patterns $\$a^{n-1}b^{m-1}\$ $ and recursively. Thus it appears the pattern $a\$$ or $\$b$ and the configuration considered is excluded (see Figure~\ref{figure.exSA}).

If we consider the $n$-approximation row following $\vect{e_1}$ with the grow along to $\vect{e_2}$. In the subshift of finite type $\T^{n,\vect{e_1},\vect{e_2}}_{F}=\T^{n,2\to1}_{F}$, there is only $n$ lines to detect a forbidden pattern so $\sa{\vect{e_1}}{\T^{n,2\to1}_{F}}$ is the subshift where the forbidden patterns are $\left\{\$a^pb^m\$,\$a^pb^ma,ba^pb^m\$,ba^pb^ma:m\ne p \textrm{ and }\max(p,m)\leq n\right\}$ (Figure~\ref{figure.exSAapprox}). We deduce that $\varphi_{F,\id,2\to1}(n)=n$. In Example~\ref{ex.log} we will see that it is possible to obtain $\sa{\vect{e_1}}{\T_F}$ thanks to another SFT but with a better speed. A natural question is to determine the optimal speed. 

\figex{
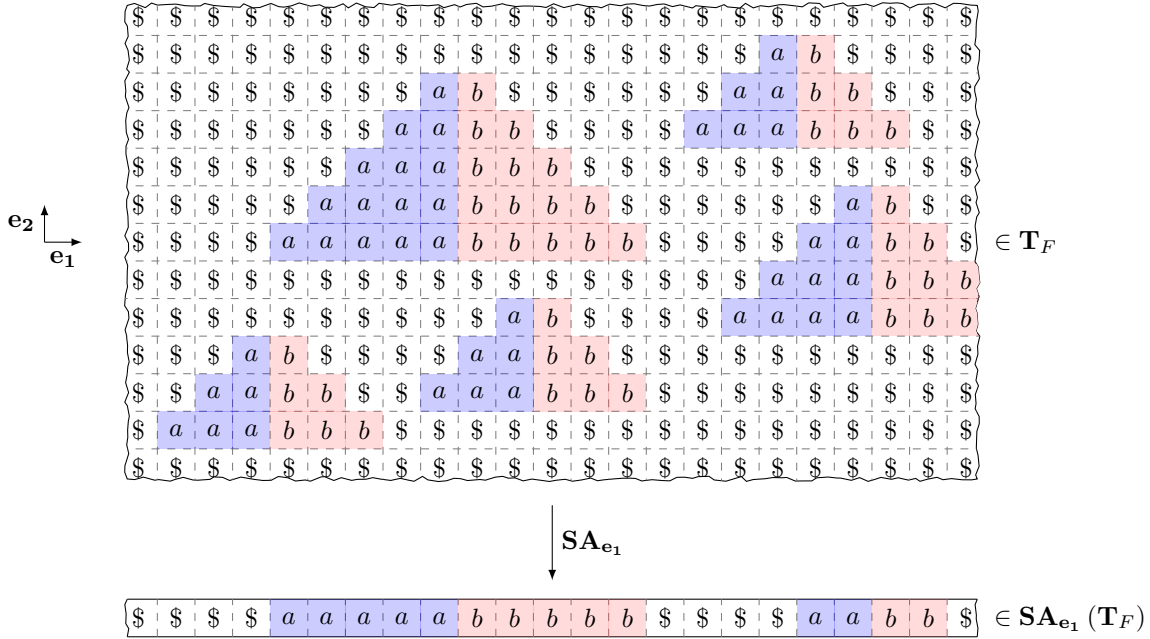
\begin{figure}[h!]
\begin{center}
\begin{tikzpicture}[scale=0.5]
\begin{scope}
\clip[draw,decorate,decoration={random steps, segment length=3pt, amplitude=1pt}] (0.2,0.2) rectangle (22.8,12.8);

\foreach \x in {0,1,...,22}{
\draw[dashed,black!50] (\x,0)--(\x,13);
\foreach \y in {0,1,...,13}{
\draw (\x+0.5,\y+0.5) node {$\$$};
}}
\foreach \y in {0,1,...,13}{
\draw[dashed,black!50] (0,\y)--(23,\y);}

\foreach \x\y in {1/1,2/1,3/1,2/2,3/2,3/3,4/6,5/6,6/6,7/6,8/6,5/7,6/7,7/7,8/7,6/8,7/8,8/8,7/9,8/9,8/10,16/4,17/4,18/4,19/4,17/5,18/5,19/5,18/6,19/6,19/7,8/2,9/2,10/2,9/3,10/3,10/4,15/9,16/9,17/9,16/10,17/10,17/11}{
\fill [white] (\x+0.1,\y+0.1) rectangle (\x+0.9,\y+0.9);
\fill [opacity=0.2,blue] (\x,\y) rectangle (\x+1,\y+1);
\draw (\x+.5,\y+.5) node{$a$};
}
\foreach \x\y in {4/1,5/1,6/1,4/2,5/2,4/3,9/6,10/6,11/6,12/6,13/6,9/7,10/7,11/7,12/7,9/8,10/8,11/8,9/9,10/9,9/10,20/4,21/4,22/4,20/5,21/5,22/5,20/6,21/6,20/7,11/2,12/2,13/2,11/3,12/3,11/4,18/9,19/9,20/9,18/10,19/10,18/11}{
\fill [white] (\x+0.1,\y+0.1) rectangle (\x+0.9,\y+0.9);
\fill [opacity=0.2,red] (\x,\y) rectangle (\x+1,\y+1);
\draw (\x+.5,\y+.5) node{$b$};
}
\end{scope}

\begin{scope}

\draw[->,>=latex] (-2,6.5) -- (-1,6.5);
\draw[->,>=latex] (-2,6.5) -- (-2,7.5);
\draw(-1.5,6.5) node[below]{$\vect{e_1}$};
\draw(-2,7) node[left]{$\vect{e_2}$};
\draw(23,6.5) node[right]{$\in\T_F$};
\draw(23,-3.5) node[right]{$\in\sa{\vect{e_1}}{\T_F}$};
\draw[->,>=latex] (11.5,-0.5) -- (11.5,-2.5);
\draw(11.5,-1.5) node[right]{$\mathbf{SA}_{\vect{e_1}}$};
\end{scope}

\begin{scope}[yshift=-10cm]

\draw[draw,decorate,decoration={random steps, segment length=3pt, amplitude=1pt}] (0.2,6)--(0.2,7);
\draw[draw,decorate,decoration={random steps, segment length=3pt, amplitude=1pt}] (22.8,6)--(22.8,7);
\draw (0.2,6)--(22.8,6);
\draw (0.2,7)--(22.8,7);

\foreach \x in {1,...,22}{
\draw[dashed,black!50] (\x,6)--(\x,7);
\draw (\x+0.5,6.5) node {$\$$};
}
\draw (0.5,6.5) node {$\$$};

\foreach \x\y in {4/6,5/6,6/6,7/6,8/6,18/6,19/6}{
\fill [white] (\x+0.1,\y+0.1) rectangle (\x+0.9,\y+0.9);
\fill [opacity=0.2,blue] (\x,\y) rectangle (\x+1,\y+1);
\draw (\x+.5,\y+.5) node{$a$};
}
\foreach \x\y in {9/6,10/6,11/6,12/6,13/6,20/6,21/6}{
\fill [white] (\x+0.1,\y+0.1) rectangle (\x+0.9,\y+0.9);
\fill [opacity=0.2,red] (\x,\y) rectangle (\x+1,\y+1);
\draw (\x+.5,\y+.5) node{$b$};
}
\end{scope}

\end{tikzpicture}
\end{center}
\caption{Application of $\mathbf{SA}_{\vect{e_1}}$ under a configuration of $\T_{F}$.}\label{figure.exSA}
\end{figure}
%%%%
%%%%
%%%%
\begin{figure}[h!]
\begin{center}
\begin{tikzpicture}[scale=0.5]
\begin{scope}
%\draw[draw,decorate,decoration={random steps, segment length=3pt, amplitude=1pt}] (0.2,3) -- (0.2,10);
%\draw[draw,decorate,decoration={random steps, segment length=3pt, amplitude=1pt}] (22.8,3) -- (22.8,10);
%\draw (0.2,3)--(22.8,3) ;
%\draw (0.2,10)--(22.8,10) ;

\clip[draw,decoration={random steps, segment length=3pt, amplitude=1pt}]decorate{(0.2,3) -- (0.2,10)}--(22.8,10)decorate{--(22.8,3)}--cycle;

\foreach \x in {0,1,...,22}{
\draw[dashed,black!50] (\x,3)--(\x,10);
\foreach \y in {3,4,...,9}{
\draw (\x+0.5,\y+0.5) node {$\$$};
}}
\foreach \y in {4,5,...,9}{
\draw[dashed,black!50] (0,\y)--(23,\y);}

\foreach \x\y in {3/3,2/6,3/6,4/6,5/6,6/6,7/6,8/6,3/7,4/7,5/7,6/7,7/7,8/7,4/8,5/8,6/8,7/8,8/8,5/9,6/9,7/9,8/9,16/4,17/4,18/4,19/4,17/5,18/5,19/5,18/6,19/6,19/7,9/3,10/3,10/4,15/9,16/9,17/9}{
\fill [white] (\x+0.1,\y+0.1) rectangle (\x+0.9,\y+0.9);
\fill [opacity=0.2,blue] (\x,\y) rectangle (\x+1,\y+1);
\draw (\x+.5,\y+.5) node{$a$};
}
\foreach \x\y in {4/3,9/6,10/6,11/6,12/6,13/6,9/7,10/7,11/7,12/7,9/8,10/8,11/8,9/9,10/9,20/4,21/4,22/4,20/5,21/5,22/5,20/6,21/6,20/7,11/3,12/3,11/4,18/9,19/9,20/9}{
\fill [white] (\x+0.1,\y+0.1) rectangle (\x+0.9,\y+0.9);
\fill [opacity=0.2,red] (\x,\y) rectangle (\x+1,\y+1);
\draw (\x+.5,\y+.5) node{$b$};
}
\end{scope}

\begin{scope}

\draw[->,>=latex] (-2,6.5) -- (-1,6.5);
\draw[->,>=latex] (-2,6.5) -- (-2,7.5);
\draw(-1.5,6.5) node[below]{$\vect{e_1}$};
\draw(-2,7) node[left]{$\vect{e_2}$};
\draw(23,6.5) node[right]{$\in\T^{3,\vect{e_1},\vect{e_2}}_{F}=\T^{3,2\to1}_{F}$};
\draw(23,-0.5) node[right]{$\in\sa{\vect{e_1}}{\T^{3,2\to1}_{F}}$};
\draw[->,>=latex] (11.5,2.5) -- (11.5,0.5);
\draw(11.5,1.5) node[right]{$\mathbf{SA}_{\vect{e_1}}$};
\end{scope}

\begin{scope}[yshift=-7cm]

\draw[draw,decorate,decoration={random steps, segment length=3pt, amplitude=1pt}] (0.2,6)--(0.2,7);
\draw[draw,decorate,decoration={random steps, segment length=3pt, amplitude=1pt}] (22.8,6)--(22.8,7);
\draw (0.2,6)--(22.8,6);
\draw (0.2,7)--(22.8,7);

\foreach \x in {1,...,22}{
\draw[dashed,black!50] (\x,6)--(\x,7);
\draw (\x+0.5,6.5) node {$\$$};
}
\draw (0.5,6.5) node {$\$$};

\foreach \x\y in {2/6,3/6,4/6,5/6,6/6,7/6,8/6,18/6,19/6}{
\fill [white] (\x+0.1,\y+0.1) rectangle (\x+0.9,\y+0.9);
\fill [opacity=0.2,blue] (\x,\y) rectangle (\x+1,\y+1);
\draw (\x+.5,\y+.5) node{$a$};
}
\foreach \x\y in {9/6,10/6,11/6,12/6,13/6,20/6,21/6}{
\fill [white] (\x+0.1,\y+0.1) rectangle (\x+0.9,\y+0.9);
\fill [opacity=0.2,red] (\x,\y) rectangle (\x+1,\y+1);
\draw (\x+.5,\y+.5) node{$b$};
}
\end{scope}
\end{tikzpicture}
\end{center}
\caption{Application of $\mathbf{SA}_{\vect{e_1}}$ under a configuration of $\T^{3,\vect{e_1},\vect{e_2}}_{F}=\T^{3,2\to1}_{F}$.}\label{figure.exSAapprox}
\end{figure}
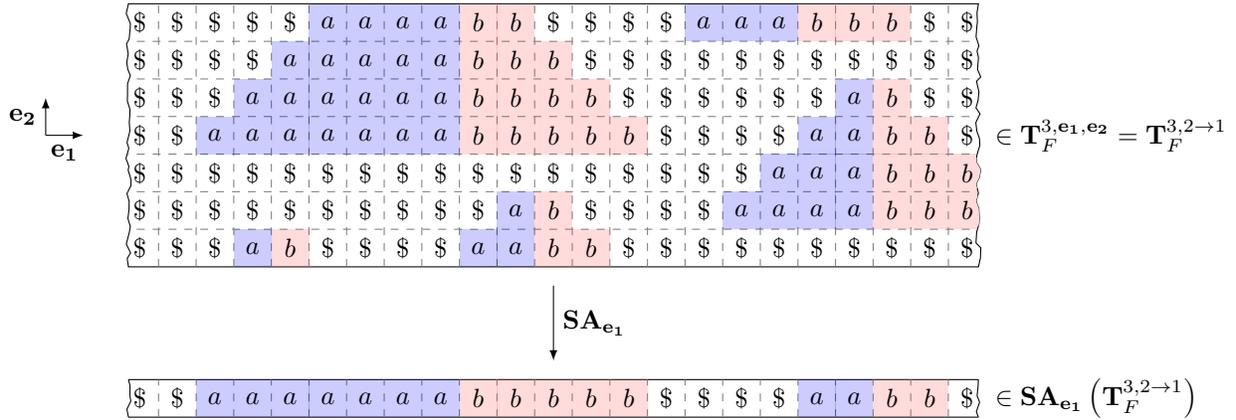}
\end{example}

\paragraph{Speed of convergence with larger row}\  Let $r\in\N$, recall that for the basis $\Base_{\H}=(\vect{v_1},\dots,\vect{v_{d''}})$ one denotes $\BB_r^{\Base_{\H}}=\left\{k_1\vect{v_1}+\dots+ k_{d''}\vect{v_{d''}}:(k_1,\dots,k_{d''})\in [-r,r]^{d''}\right\}$. Define 
\[\begin{array}{lccl}
\theta_r:&\A^{\Z^d}&\longrightarrow&\left(\A^{\BB_r^{\Base_{\H}}}\right)^{\Z^d}\\
&x&\longmapsto	&\left(x_{\vect{i}+\BB_r^{\Base_{\H}}}\right)_{\vect{i}\in\Z^d}.
\end{array}\]
Clearly $\theta_r$ is a conjugacy. Given a morphism $\pi:\A^{\Z^d}\to\B^{\Z^d}$ one just denotes $\pi=\theta_r\circ\pi\circ\theta_r^{-1}:\left(\A^{\BB_r^{\Base_{\H}}}\right)^{\Z^d}\to\left(\B^{\BB_r^{\Base_{\H}}}\right)^{\Z^d}$. Moreover, given a finite set of patterns $F$ one considers $\theta_r(F)$ is the set of patterns obtained as pre-image of $F$ by $\theta_r^{-1}$ defined by
\[\theta_r(F)=\left\{p\in\left(\A^{\BB_r^{\Base_{\H}}}\right)^{\ast}: \exists p'\in\A^\ast,\exists p''\in F \textrm{ such that }p'_\supp(p'')=p''\textrm{ and }\theta_r(p')=p\right\}.\] 
Study speed of convergence of $T_{\theta_r(F)}$ is the same that define speed of convergence of $\T_F$ for a row of size $r$ along to $\Base_{\H}$. There is no reason for the speeds of convergence of $T_F$ and $T_{\theta_r(F)}$ being comparable.

\subsection{Some Invariances} 

Note that for a given subshift of finite type, its speed of convergence depends on the finite set of forbidden patterns chosen, this variation can nevertheless be bounded as specified in the next proposition. 
\begin{proposition}\label{prop:SpeedEqualSFT}
Let $F$ and $F'$ be two finite set of forbidden patterns such that $\T_F=\T_{F'}\subset\A^{\Z^d}$. There exists $M\in\N$ such that $\left|\varphi_{F,\id,\Base_{\G},\Base_{\H}}(k)-\varphi_{F',\id,\Base_{\G},\Base_{\H}}(k)\right|\leq M$ for all $k\in\N$.
\end{proposition}
\begin{proof}
Let $\G$ and $\H$ be two sublattices such that $\Z^d=\G\oplus\H$ of basis $\Base_{\G}=(\vect{u_1},\dots,\vect{u_{d'}})$ and $\Base_{\H}=(\vect{u_{d'+1}},\dots,\vect{u_{d}})$. Denote $\BB_k=\{n_1\vect{u_1}+\dots+n_d\vect{u_d}:(n_1,\dots,n_d)\in[-k,k]^d\}$.

Assume that the supports of patterns of $F$ are included in $\BB_k$. There exists $M$ such that if a pattern $p\in\A^{\BB_M}$ verifies $p'\not\sqsubset p$ for all $p'\in F'$, then no pattern of $F$ appears in $p_{\BB_k}$. If not, by compactness, there exists a configuration of $\T_{F'}$ which contains a pattern of $F$. The same property holds if we reverse the role of $F$ and $F'$, moreover it is possible to choose the same $M$. 

Thus, one has $\sa{\Base_{\G}}{\T_{F'}^{n+M,\G,\Base_{\H}}}\subset\sa{\Base_{\G}}{\T_{F}^{n,\G,\Base_{\H}}}\subset\sa{\Base_{\G}}{\T_{F'}^{n-M,\G,\Base_{\H}}}$ for $n\in\N$ sufficiently large. The result follows.
\end{proof}

We are now study the behavior of speed of convergence under factor and conjugacy. 

\begin{proposition}\label{prop:FactorAndSpeed}
Let $F$ be a set of forbidden patterns such that $\T_F\subset\A^{\Z^d}$. Consider $\pi:\A^{\Z^d}\to\B^{\Z^d}$ and $\psi:\B^{\Z^d}\to\B'^{\Z^d}$ two morphisms. There exists $r$ and $r'$ such that $\varphi_{F,\psi\circ\pi,\Base_{\G},\Base_{\H}}(k)\leq\varphi_{\theta_{r'}(F),\pi,\Base_{\G},\Base_{\H}}(k+r)$. 

Moreover if the neighborhood of $\psi$ is included in $\G$, one has $\varphi_{F,\psi\circ\pi,\Base_{\G},\Base_{\H}}(k)\leq\varphi_{F,\pi,\Base_{\G},\Base_{\H}}(k+r)$.
\end{proposition}
\begin{proof}
Let $\U$ be the neighborhood of $\psi$ and assume that $\U\subset\BB^{\Base_{\G}}_r+\BB^{\Base_{\H}}_{r'}$. Denote $d'$ and $d''$ respectively the dimension of $\G$ and $\H$. Consider $u\in\B'^{[0,k-1]^{d'}}$ such that $u\notin\Lang(\sa{\Base_{\G}}{\psi\circ\pi(\T_{F})})$. One has 
\[\theta_{r'}(\psi^{-1}(u))=\left\{v\in\left(\A^{\BB_r^{\Base_{\H}}}\right)^{[-r,k+r-1]^{d'}}: \psi(\theta_{r'}^{-1}(v))=u  \right\}\not\subset\Lang(\sa{\Base_{\G}}{\theta_{r'}(\pi(\T_{F}))}).\] 
Thus $\theta_{r'}(\psi^{-1}(u))\not\subset\Lang\left(\sa{\Base_{\G}}{\pi(\T^n_{\theta_{r'}(F)})}\right)$ for all $n\geq\varphi_{\theta_{r'}(F),\pi,\Base_{\G},\Base_{\H}}(k+2r)$. One deduces that $u\notin\Lang(\sa{\Base_{\G}}{\psi\circ\pi(\T^n_{F})})$, that is to say $n\geq\varphi_{F,\psi\circ\pi,\Base_{\G},\Base_{\H}}(k)$.
\end{proof}

Using Proposition~\ref{prop:FactorAndSpeed}, we deduce the following corollary.
\begin{corollary}\label{cor:ConjugaisonSofic}
Let $\T\subset\B^{\Z^d}$ and $\T'\subset\B'^{\Z^d}$ be two sofic subshifts. If $\T$ and $\T'$ are conjugate, there exist $\T_F\subset\A^{\Z^d}$ a subshift, $\pi:\T_F\to\T$ and $\pi':\T_F\to\T'$ two factor maps and $r,r'\in\N$ such that  
\[\varphi_{F,\pi,\Base_{\G},\Base_{\H}}(k)\leq\varphi_{\theta_{r'}(F),\pi',\Base_{\G},\Base_{\H}}(k+2r)\leq\varphi_{\theta_{2r'}(F),\pi,\Base_{\G},\Base_{\H}}(k+4r)\textrm{ for all }k\in\N.\]
One can choose $r'=0$ if the neighborhood of the conjugacy is included in $\G$.
\end{corollary}

\begin{corollary}\label{cor:ConjugaisonSFT}
Let $\T_F\subset\A^{\Z^d}$ and $\T_{F'}\subset\A'^{\Z^d}$ be two subshifts of finite type of forbidden patterns $F$ and $F'$. If $\T_F$ and $\T_{F'}$ are conjugate then there exists $r,r',M\in\N$ such that 
\[\varphi_{F,\id,\Base_{\G},\Base_{\H}}(k)\leq\varphi_{\theta_{r'}(F'),\id,\Base_{\G},\Base_{\H}}(k+2r)+M\leq\varphi_{\theta_{2r'}(F),\id,\Base_{\G},\Base_{\H}}(k+4r)+2M\textrm{ for all }k\in\N.\]

One can chose $r'=0$ if the neighborhood of the conjugacy  is included in $\G$.
\end{corollary}
\begin{proof}
Let $\psi:\T_{F}\longrightarrow\T_{F'}$ be the conjugacy map of neighborhood $\U\subset\BB^{\Base_{\G}}_r+\BB^{\Base_{\H}}_{r'}$. Denote \[\psi^{-1}(F')=\left\{p \in\A^{\ast} :  \exists p'\in F' \textrm{ such that } \supp(p)=\supp(p')+\U \textrm{ and }\psi(p)=p'\right\}.\]
Let $u\in\A'^{[0,k-1]^{d'}}$ and $n\geq\varphi_{\theta_{r'}(\psi^{-1}(F')),\id,\Base_{\G},\Base_{\H}}(k+2r)$, one has:
\[\begin{array}{rcll}
u\notin\sa{\Base_{\G}}{\T_{F'}}&\Longrightarrow&\theta_{r'}(\psi^{-1}(u))\cap\Lang\left(\sa{\Base_{\G}}{\theta_{r'}(\T_{\psi^{-1}(F')})}\right)=\emptyset&\\
&\Longrightarrow&\theta_{r'}(\psi^{-1}(u))\cap\Lang\left(\sa{\Base_{\G}}{\theta_{r'}(\T_{\psi^{-1}(F')}^{n,\G,\Base_{\H}})}\right)=\emptyset&\textrm{by definition of speed of convergence}\\
&\Longrightarrow&u \notin \Lang\left(\sa{\Base_{\G}}{\T_{F'}^{n,\G,\Base_{\H}}}\right)
\end{array}\]
One has $\T_{\theta_{r'}(F)}=\T_{\theta_{r'}(\psi^{-1}(F'))}$, from Proposition~\ref{prop:SpeedEqualSFT} we deduce that there exists $M\in\N$ such that
\[\varphi_{F',\id,\Base_{\G},\Base_{\H}}(k)\leq\varphi_{\theta_{r'}(\psi^{-1}(F')),\id,\Base_{\G},\Base_{\H}}(k+2r)\leq\varphi_{\theta_{r'}(F),\id,\Base_{\G},\Base_{\H}}(k+2r)+M.\]
\end{proof}

\begin{example}[Substitutive tilings]
A two-dimensional \define{substitution} is a function $s:\A\to\A^{[0,a-1]\times[0,b-1]}$ which can be naturally iterated to obtain a function $s^n:\A\to\A^{[0,a^n-1]\times[0,b^n-1]}$. The substitutive subshift is defined by \[\T_s=\{x\in\mathcal{A}^{\Z^2}: \textrm{ for all }u\sqsubset x, \textrm{ there exists $n\in\N$ and $a\in\mathcal{A}$ which verifies }u\sqsubset s^n(a)\}.\]
According to Mozes' result~\cite{Mozes-1989}, $\T_s$ is sofic, that is to say there exist a subshift of finite type $\T_{F_s}\subset\mathcal{A}_s^{\Z^2}$ and a factor map $\pi_s:\mathcal{A}_s\to\mathcal{A}$ such that $\pi_s(\T_{F_s})=\T_{s}$. In the construction of Mozes it is possible to consider that  $\pi_s$ is a conjugacy almost everywhere and there exists a substitution $\widetilde{s}$ on $\mathcal{A}_s$ which factorizes on $s$ and verifies $\T_{F_s}=\T_{\widetilde{s}}$. 

Let $\Base_{\G}=(\vect{e_1})$ and $\Base_{\H}=(\vect{e_2})$. One has $\varphi_{F_s,\id,\Base_{\G},\Base_{\H}}=\varphi_{F_s,\pi,\Base_{\G},\Base_{\H}}$ since $\pi$ is a conjugacy almost everywhere and letter to letter.

Since it is possible to apply $s$ on $\T_{F_s}^{n,\G,\Base_{\H}}$, one deduces that for all $k\in\N$ one has $\varphi_{F_s,\id,\Base_{\G},\Base_{\H}}(ak)=b\varphi_{F_s,\id,\Base_{\G},\Base_{\H}}(k)$ so there exists $M_1,M_2$ such that
\[ M_1 k^{\frac{\log(b)}{\log(a)}}\leq \varphi_{F_s,\id,\Base_{\G},\Base_{\H}}(k)\leq M_2 k^{\frac{\log(b)}{\log(a)}},\]
and in the same way
\[ M_1 k^{\frac{\log(b)}{\log(a)}}\leq \varphi_{\theta_r(F_s),\id,\Base_{\G},\Base_{\H}}(k)\leq M_2 k^{\frac{\log(b)}{\log(a)}},\]
We deduce that if two substitutions do not have the same size of support, then the substitutive subshifts associated cannot be conjugate. 
\end{example}

\section{Subshift $\varphi,d$-realizable by SFT or by sofic}\label{section:InvariantOfConjugacyBis}

A subshift $\Sigma\subset\A^{\Z^{d'}}$ is \define{$d$-realizable by subshift of finite type} (respectively \define{$d$-realizable by sofic subshift}) if there exists $\T$ a subshift of finite type (respectively a sofic subshift) and a basis $\Base_{\G}=(\vect{e_1},\dots,\vect{e_{d'}})$ such that $\sa{\Base_{\G}}{\T}=\Sigma$. It follows from Theorem 1.1 that every d'-dimensional effective subshift is d-realizable by sofic subshift for d>d', nevertheless the same result does not hold for realization by SFT (see Section~8 of~\cite{Pavlov-Schraudner-2010} for a counter-example). In the sequel, we are interested in the speed of convergence and we are going to see that it is an invariant of conjugacy for effective subshift.

\subsection{Canonical representation}

Consider $\Base=(\vect{e_1},\dots,\vect{e_{d}})$ the canonical base of $\Z^d$. Let $d',d''\in\N$ such that $d'+d''=d$, define $\Base_{\G_{d'}}=(\vect{e_1},\dots,\vect{e_{d'}})$ and $\Base_{\H_{d''}}=(\vect{e_{d'+1}},\dots,\vect{e_{d}})$. In fact when we study the speed of convergence to realize a subshift by a $d$-dimensional subshift of finite type, it is sufficient to study the realization of $\Sigma$ following $\G_{d'}$ with the growth according to the base $\Base_{\H_{d''}}$.  More precisely, one has the following proposition.

\begin{proposition}
Let $\G$ and $\H$ be two sublattices such that $\G\oplus\H=\Z^d$ of basis respectively $\Base_{\G}$ and $\Base_{\H}$.
Let $\Sigma\subset\B^{\Z^{d'}}$ be a $d'$-dimensional subshift. Assume that there exists a set of $d$-dimensional patterns $F$ and a factor map $\pi:\A^{\Z^{d}}\to\B^{\Z^{d}}$ such that $\sa{\Base_{\G}}{\pi(\T_{F})}=\Sigma$. Then there exists a set of $d$-dimensional patterns $F'$ and a factor map $\pi':\A^{\Z^{d}}\to\B^{\Z^{d}}$ such that $\sa{\Base_{\G_{d'}}}{\pi'(\T_{F'})}=\Sigma$ and $\varphi_{F,\pi,\Base_{\G},\Base_{\H}}= \varphi_{F',\pi',\Base_{\G_{d'}},\Base_{\H_{d-d'}}}$ where $\Base_{\G_{d'}}=(\vect{e_1},\dots,\vect{e_{d'}})$ and $\Base_{\H_{d-d'}}=(\vect{e_{d'+1}},\dots,\vect{e_{d}})$. 

If $\pi$ is letter to letter, one can choose $\pi$ instead $\pi'$; in particular, if $\pi=\id$ then $\varphi_{F,\id,\Base_{\G},\Base_{\H}}= \varphi_{F',\id,\Base_{\G_{d'}},\Base_{\H_{d-d'}}}$.
\end{proposition}
\begin{proof}
Assume that $\Base_{\G}=(\vect{u_1},\dots,\vect{u_{d'}})$ and $\Base_{\H}=(\vect{u_{d'+1}},\dots,\vect{u_{d}})$. Denote $\U$ the neighborhood of $\pi$. Consider the bijection
\[\begin{array}{rlll}
\trans{}:&\Z^d&\longrightarrow&\Z^d\\
&\vect{n}&\longmapsto& n_1\vect{e_1}+\dots+n_d\vect{e_{d}}\\
&&&\textrm{where }\vect{n}=n_1\vect{u_1}+\dots+n_d\vect{u_{d}}\textrm{ with }(n_1,\dots,n_d)\in\Z^d.
\end{array}\]
Consider the set of forbidden patterns $F'$ 
\[F'=\left\{p': \exists p\in F \textrm{ such that } p'_{\trans{}(\vect{n})}=p_{\vect{n}} \textrm{ for all }\vect{n}\in\supp(p)\right\}\]
and the factor map $\pi'$ of neighborhood $\trans{}(\U)$ defined by $\pi'((u_{\trans{}(\vect{i})})_{\vect{i}\in\U})=\pi(u)$. Clearly $\sa{\Base_{\G}}{\pi(\T_{F})}=\sa{\Base_{\G_{d'}}}{\pi'(\T_{F'})}$ and $\varphi_{F,\pi,\Base_{\G},\Base_{\H}}(k)= \varphi_{F',\pi',\Base_{\G_d},\Base_{\H_{d-d'}}}(k)$ for all $k\in\N$. If $\pi$ is letter to letter then $\pi'=\pi$. 
\end{proof}

\begin{notation}
In the following we just denote $\T_F^{n,d\to d'}$ for $\T_F^{n,\G_{d'},\Base_{\H_{d-d'}}}$, $\sa{d'}{\cdot}$ for $\sa{\G_{d'}}{\cdot}$ and $\varphi_{F,\pi,d\to d'}$ for $\varphi_{F,\pi,\Base_{\G_{d'}},\Base_{\H_{d-d'}}}$.  
\end{notation}

\subsection{How compare different speed of convergence} If a subshift is realized by a speed $\varphi$, it is possible to find another realization where the speed is slower. This is formalized by the following proposition.

\begin{proposition}
Let $\Sigma\subset\B^{\Z^{d'}}$ be a $d'$-dimensional subshift. Assume  such that $\sa{d'}{\pi(\T_{F})}=\Sigma$ for a morphism $\pi:\A\to\B$ and a finite set of forbidden patterns $F$. Let $M\in\N$ one has:
\begin{itemize}
\item there exists a finite set of $d$-dimensional patterns $F'$ and a factor $\pi'$ such that $\sa{d'}{\pi'(T_{F'})}=\Sigma$ and $\varphi_{F',\pi',d\to d'}(k)=M\varphi_{F,\pi,d\to d'}(k)$ for all $k\in\N$ (if $\pi=\id$, one can choose $\pi'=\id$).
\item there exists a finite set of $d$-dimensional patterns $F'$ and a factor $\pi'$ such that $\sa{d'}{\pi'(T_{F'})}=\Sigma$ and $\varphi_{F',\pi',d\to d'}(k)=\frac{1}{M}\varphi_{F,\pi,d\to d'}(k)$ for all $k\in\N$.
\end{itemize}

\end{proposition}
\begin{proof} We are going to prove the first point, the proof of the second point is in the same spirit. Let $F$ be a finite set of forbidden patterns and $\pi$ a morphism of neighborhood $\U$ such that $\sa{d'}{\pi(\T_{F})}=\Sigma$. For $l\in[0,\dots,M-1]$, define
\[\begin{array}{rlll}
\trans{l}:&\Z^d&\longrightarrow&\Z^d\\
&\vect{n}&\longmapsto& \vect{n'}+((M-1)\times n+l)\vect{e_{d'+1}}+\vect{n''}\\
&&&\textrm{where }\vect{n}=\vect{n'}+n\vect{e_{d'+1}}+\vect{n''}\textrm{ with }\vect{n'}\in\Vect{\vect{e_{1}},\dots,\vect{e_{d'}}},n\in\Z \textrm{ and }\vect{n''}\in\Vect{\vect{e_{d'+2}},\dots,\vect{e_{d}}}.
\end{array}\]
Define the set of forbidden patterns 
\[F'=\left\{p': \exists p\in F \textrm{ such that } p'_{\trans{l}(\vect{n})}=p_{\vect{n}} \textrm{ for all }\vect{n}\in\supp(p) \textrm{ and } l\in[0,\dots,M-1]\right\},\]
and $\pi'$ of neighborhood $\cup_l\trans{l}(\U)$ which apply $\pi$ on each slice $l\in[0,\dots,M-1]$.
Clearly $\sa{d'}{\pi'(\T_{F'})}=\Sigma$ and $ \varphi_{F',\pi',d\to d'}(k)=M\varphi_{F,\pi,d\to d'}(k)$ for all $k\in\N$.
\end{proof}

To speak of class of speed convergence of a $d$-dimensional subshift of finite type (or sofic subshift) towards a $d'$-dimensional subshift, we need to introduce an order relation on $\cF$ the set of non-decreasing functions from $\N$ to $\N$. 

\begin{definition}
We say that $\varphi\prec\varphi'$ if there exists $r,M\in\N$ such that  $\varphi(k)\leq M\varphi'(k+r)$ for all $k\in\N$. We say that  $\varphi\sim\varphi'$ if  $\varphi\prec\varphi'$ and  $\varphi'\prec\varphi$. 
\end{definition}

\begin{definition}
A $d'$-dimensional subshift $\Sigma$ is  \define{$\varphi,d$-realizable by projective subaction of a SFT} if there exists a set of $d$-dimensional forbidden pattern $F$ such that $\sa{d'}{\T_F}=\Sigma$ and  $\varphi_{F,\id,d\to d'}\sim\varphi$. 

A $d'$-dimensional subshift $\Sigma$ is  \define{$\varphi,d$-realizable by projective subaction of a sofic} if there exists a set of $d$-dimensional forbidden pattern $F$ and a factor $\pi$ such that $\sa{d'}{\pi(\T_F)}=\Sigma$ and  $\varphi_{F,\pi,d\to d'}\prec\varphi$. 

Denote 
\[
\begin{array}{rcl}
\cF^{\sft}_{\Sigma,d\to d'}&=&\left\{\varphi \in\cF: \exists F\underset{\textrm{fini}}{\subset}\A^{\ast} \textrm{ such that }\varphi=\varphi_{F,d\to d'} \textrm{ and }\sa{\Z^{d'}}{\T_{F}}=\Sigma\right\}\\
\textrm{ and }\cF^{\sofic}_{\Sigma,d\to d'}&=&\left\{\varphi\in\cF:\exists F\underset{\textrm{fini}}{\subset}\B^{\ast} \textrm{ and } \pi:\B\to\A \textrm{ such that }\varphi=\varphi_{F,d\to d'} \textrm{ and }\sa{\Z^{d'}}{\pi(\T_{F})}=\Sigma\right\}.
\end{array}
\]
Of course $\cF^{\sft}_{\Sigma,d\to d'}\subset\cF^{\sofic}_{\Sigma,d\to d'}$, $\cF^{\sft}_{\Sigma,d\to d'}\subset\cF^{\sft}_{\Sigma,d+1\to d'}$ and $\cF^{\sofic}_{\Sigma,d\to d'}\subset\cF^{\sofic}_{\Sigma,d+1\to d'}$.

A $d'$-dimensional subshift $\Sigma$ is  \define{sharp $\varphi,d$-realizable by projective subaction of a SFT} (respectively \define{sharp $\varphi,d$-realizable by projective subaction of a sofic}) if $\varphi\in\cF^{\sft}_{\Sigma,d\to d'}$ (respectively $\varphi\in\cF^{\sofic}_{\Sigma,d\to d'}$) and $\varphi\prec\varphi'$ for all $\varphi'\in\cF^{\sft}_{\Sigma,d\to d'}$ (respectively $\varphi'\in\cF^{\sofic}_{\Sigma,d\to d'}$). 
\end{definition}

\begin{example}
Consider the sofic subshift $\Sigma\subset\{0,1,2\}^{\Z}$ where the forbidden patterns are given by $\{10^{2n+1}2; 20^n1 : n\in\N\}$. Of course $\Sigma$ is constant realizable by projective subaction of a sofic. However, like it is unstable, it is not constant realizable by projective subaction of a SFT (see~\cite{Pavlov-Schraudner-2010}). 
\end{example}

\subsection{Invariance by conjugacy of subshift $\varphi,d$-realizable}

\begin{proposition}
Let $\Sigma$ and $\Sigma'$  be two conjugated $d'$-dimensional subshfits .  

The subshift $\Sigma$ is  $\varphi,d$-realizable by projective subaction of a subshift of finite type if and only if it is the same for $\Sigma'$.

The subshift $\Sigma$ is  $\varphi,d$-realizable by projective subaction of a sofic if and only if it is the same for $\Sigma'$.
\end{proposition}
\begin{proof}
Let $\psi:\Sigma\longrightarrow\Sigma'$ be the conjugaison between $\Sigma\subset\A^{\Z^{d'}}$ and $\Sigma'\subset\B^{\Z^{d'}}$. The local function associated of neighborhood $\U\subset\G_{d'}$ can be extended in a function $\psi:\A^{\Z^{d}}\longrightarrow\B^{\Z^{d}}$ of same neighborhood. 

Let $\T_F$ be a subshift of finite type and $\pi$ be a morphism such that $\sa{d'}{\pi(\T_F)}=\Sigma$. Since $\psi$ is injective on $\Sigma$ and that its neighborhood is included in $\G_{d'}$, one deduces that $\psi$ is injective on $\pi(\T_F)$, so $\pi(\T_F)$ and $\psi\circ\pi(\T_F)$ are conjugate. From Corollary~\ref{cor:ConjugaisonSofic}, one deduces that $\Sigma$ and $\Sigma'$ are $\varphi_{F,\pi,d\to d'},d$-realizable. 

The SFT case comes with similar ideas and Corollary~\ref{cor:ConjugaisonSFT}.
\end{proof}

Thus it is interesting to study the sets $\cF^{\sft}_{\Sigma,d}$ or $\cF^{\sofic}_{\Sigma,d}$ modulo the equivalence relation induced by $\prec$. The next proposition shows that two elements have always a supremum.

\begin{proposition}
If $\Sigma$ is $\varphi'$ and $\varphi''$ realized by projective subaction of $d$-dimensional subshift of finite type, then there exists  $\varphi$ which realizes $\Sigma$ such that $\varphi'\prec\varphi$ and $\varphi''\prec\varphi$ and for all $\varphi'''$ such that $\varphi'\prec\varphi'''$ and $\varphi''\prec\varphi'''$ one has $\varphi\prec\varphi'''$.

The same result holds if one consider realization by projective subaction of sofic. 
\end{proposition}
\begin{proof}
Let $F'$ and $F''$ be finite set of $d$-dimensional patterns such that $\Sigma=\sa{d'}{\T_{F'}}=\sa{d'}{\T_{F''}}$, $\varphi'\sim\varphi_{F',\id,d'\to d}$ and $\varphi''\sim\varphi_{F'',\id,d'\to d}$. For $l\in\{0,1\}$, consider the application
\[\begin{array}{rlll}
\trans{l}:&\Z^d&\longrightarrow&\Z^d\\
&\vect{n}&\longmapsto& \vect{n'}+(2n+l)\vect{e_{d'}}+\vect{n''}\\
&&&\textrm{where }\vect{n'}\in\Vect{\vect{e_1},\dots,\vect{e_{d'}}}, \vect{n'}\in\Vect{\vect{e_{d'+2}},\dots,\vect{e_{d}}}\textrm{ and }n\in\Z\\
&&& \textrm{such that }\vect{n}=\vect{n'}+n\vect{e_{d'+1}}+\vect{n''}.
\end{array}\]
Consider the set of forbidden patterns 
\[F=\left\{p:\begin{array}{ll}& \exists p'\in F' \textrm{ such that } p_{\trans{0}(\vect{n})}=p'_{\vect{n}} \textrm{ for all }\vect{n}\in\supp(p)\\
\textrm{ or }& \exists p''\in F'' \textrm{ such that } p_{\trans{1}(\vect{n})}=p''_{\vect{n}} \textrm{ for all }\vect{n}\in\supp(p)
\end{array}\right\}\]
Clearly $\sa{\G}{\T_{F}}=\Sigma$ and $\varphi_{F,\id,d'\to d}(k)=2\max(\varphi_{F',\id,d'\to d}(k),\varphi_{F'',\id,d'\to d}(k))$ for all $k\in\N$. The result follows.

The same idea holds for realization by projective subaction of sofic.
\end{proof}

In the case of realization by projective subaction of $\Sigma$ by $d$-dimensional sofic, the next proposition shows that two elements od  $\cF^{\sofic}_{\Sigma,d}$ admit also an infimum for the order $\prec$, so $(\cF^{\sofic}_{\Sigma,d},\prec)$ is a lattice order which is preserved by conjgacy.

\begin{proposition}
If $\Sigma$ is $\varphi'$ and $\varphi''$ realized by projective subaction of $d$-dimensional sofic, then there exists  $\varphi$ which realizes $\Sigma$ such that $\varphi\prec\varphi'$ and $\varphi\prec\varphi''$ and for all $\varphi'''$ such that $\varphi'''\prec\varphi'$ and $\varphi'''\prec\varphi'$ one has $\varphi'''\prec\varphi$.
\end{proposition}
\begin{proof}
Let $F'$ and $F''$ be finite set of $d$-dimensional patterns and $\pi',\pi''$ two morphisms ,such that $\Sigma=\sa{d'}{\pi'(\T_{F'})}=\sa{d'}{\pi''(\T_{F''})}$, $\varphi'\sim\varphi_{F',\pi',d'\to d}$ and $\varphi''\sim\varphi_{F'',\pi'',d'\to d}$. 

Consider the finite set of forbidden patterns $F$ such that $\T_F=\left\{(x,y)\in\T_{F'}\times\T_{F''}:\pi'(x)=\pi''(y)\right\}$ and the morphism $\pi$ which apply $\pi'$ on the first coordinate of $\T_F$ (or $\pi''$ on the second coordinate). One has $\Sigma=\sa{d'}{\pi(\T_{F})}$ and $\varphi_{F,\pi,d'\to d}(k)=\min(\varphi_{F',\pi',d'\to d}(k),\varphi_{F'',\pi'',d'\to d}(k))$.
\end{proof}

\begin{remark}
We do not know if this result holds for realization by projective subaction of subshift of finite type. 
\end{remark}
\begin{remark}It is also easy to check that $\cF^{\sofic}_{\Sigma,d}$ is stable by addition and multiplication. 
\end{remark}

%\end{document}
%%%%%%%%%%%%%%%%%%%
%
%%%%%%%%%%%%%%%%%%%

\section{Realization of subshift by projective subaction of SFT}\label{section:RealizationSFTvsSofic}

For an effective subshift $\Sigma$, generally $\cF^{\sft}_{\Sigma,d}$ and $\cF^{\sofic}_{\Sigma,d}$ are different. For example in~\cite{Pavlov-Schraudner-2010} the authors exhibit a large class of effective Sturmian subshifts which can be realized by projective subaction of sofic subshift but which cannot be realized by projective subaction of subshift of finite type. In this section we give a sufficient condition on an effective subshift $\Sigma$ to have $\cF^{\sft}_{\Sigma,d}$ and $\cF^{\sofic}_{\Sigma,d}$ equal modulo the equivalence relation $\sim$.

\begin{definition}
A $d$-dimensional subshift $\Sigma\subset\A^{\Z^{d}}$ verifies the $\mathcal{H}$ condition if there exist $N,N'\in\N$ and $\Sigma'\subset\Sigma$ an effective subshift  such that for all $z^{\star}\in\Sigma'$ one can find $\overline{z^{\star}}\in\Sigma'$,  and $\M\subset\Z^{d}$ such that:
\begin{itemize}
\item for all $\vect{i}\in\Z^d$ there exists $\vect{i'}\in\M$ such that $\vect{i}\in\vect{i}'+[0,N]^{d}$;
\item for all $\vect{i}\in\M$ one has $z^{\star}_{\vect{i}}\ne \overline{z^{\star}}_{\vect{i}}$ and $z^{\star}_{\vect{i}-\vect{e_l}} =\overline{z^{\star}}_{\vect{i}-\vect{e_l}}$ for all $l\in[1,d]$;
\item for all $\vect{i}\notin\M+[0,N']^d$ one has $z^{\star}_{\vect{i}}= \overline{z^{\star}}_{\vect{i}}$;
\item moreover, \[\textrm{for all $z\in\A^{\Z^{d}}$ which verifies }
\begin{cases}
\textrm{ $z_{\vect{i}}= z^{\star}_{\vect{i}}$} &\textrm{ for all $\vect{i}\notin\M+[0,N']^d$, }\\
\textrm{ $z_{\vect{i}+[0,N']^{d}}=z^{\star}_{\vect{i}+[0,N']^{d}}$ or $\overline{z^{\star}}_{\vect{i}+[0,N']^{d}}$}&\textrm{ for all $\vect{i}\in\M$,}
\end{cases}
\textrm{ one has $z\in\Sigma'$.}\]
\end{itemize}

One says that $\Sigma$ verifies the conditions $\widetilde{\mathcal{H}}$ if moreover there exists a morphism $\psi:\Sigma'\to\Sigma'$ such that for all $z^{\ast}\in\Sigma'$ one can take $\overline{z^{\ast}}=\psi(z^{\ast})$.
\end{definition}

A subshift with the conditions $\mathcal{H}$ or $\widetilde{\mathcal{H}}$ has positive entropy and the conditions $\mathcal{H}$ or $\widetilde{\mathcal{H}}$ are stable by conjugacy. The positions of the differences between $z^{\star}$ and $\overline{z^{\star}}$ are illustrated in Figure~\ref{figure.HProperty}.  

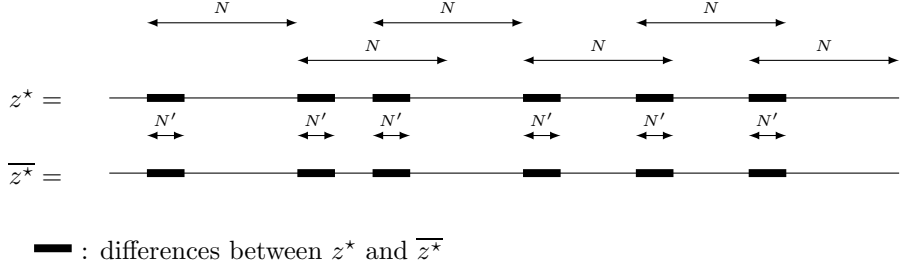
\begin{figure}[h!]
 \begin{tikzpicture}
 \draw (0.5,0)--(11,0);
 \draw (0,0) node[left] {$z^{\star}=$};
 \draw (0.5,-1)--(11,-1);
 \draw (0,-1) node[left] {$\overline{z^{\star}}=$};
 \foreach \x in {1,3,4,6,7.5,9}{
 \draw[line width=3pt] (\x,0)--(\x+0.5,0);
 \draw[line width=3pt] (\x,-1)--(\x+0.5,-1);
 \draw[<->,>=latex] (\x,-0.5)--(\x+0.5,-0.5);
 \draw (\x+0.25,-0.5) node[above]{\tiny{$N'$}};
 }
  \foreach \x in {1,4,7.5}{
 \draw [<->,>=latex] (\x,1)--(\x+2,1);
 \draw (\x+1,1) node[above]{\tiny{$N$}};

 }
  \foreach \x in {3,6,9}{
 \draw [<->,>=latex] (\x,0.5)--(\x+2,0.5);
  \draw (\x+1,0.5) node[above]{\tiny{$N$}};
 }
 
 \draw[line width=3pt] (-0.5,-2)--(0,-2);
 \draw (0,-2) node[right]{: differences between $z^{\star}$ and $\overline{z^{\star}}$};
 \end{tikzpicture}

 \caption{Illustration of the positions of the differences between $z^{\star}$ and $\overline{z^{\star}}$.}\label{figure.HProperty}
\end{figure}

\begin{example}
If $\Sigma\subset\A^{\Z}$ contains a sofic subshift of positive entropy then it verifies the condition $\widetilde{\mathcal{H}}$. 
\end{example}
\begin{example}
Let $\alpha\in\mathbb{R}$, define $S_{\alpha}\subset\{0,1\}^{\Z}$ be the subshift such that every pattern of a configuration is a pattern of the sturmian sequence of slope $\alpha$. 

Let $\pi':\{0,0',1\}^{\Z}\to\{0,1\}^{\Z}$ be a letter to letter morphism such that $\pi'(0)=0=\pi'(0')$ and $\pi'(1)=1$. Define $S'_{\alpha}=\{x\in\{0,0',1\}^{\Z}:\pi'(x)\in S_{\alpha}\}$. If $\alpha$ is computable then $S'_{\alpha}$ has the property $\widetilde{\mathcal{H}}$ whereas $S_{\alpha}$ has not.

Let $\pi'':\{0,0_1,0_2,1\}^{\Z}\to\{0,1\}^{\Z}$ be a letter to letter morphism such that $\pi''(0)=\pi''(0_1)=\pi''(0_2)=0$ and $\pi''(1)=1$. For $x\in\{0,0_1,0_2,1\}^{\Z}$, define $\psi(x)\in\{0,0_1,0_2\}^{\Z}$ obtained by removing letters $1$ in $x$. Define $S''_{\alpha,\beta}$ the set of element $x\in\{0,0_1,0_2,1\}^{\Z}$ such that $\pi''(x)\in S'_{\alpha}$ and there exists an element of $S_{\beta}$ obtained from $\psi(x)$ when we replace $0_1$ by $0, 0_2$ by $1$ and $0$ by $0$ or $1$. If $\alpha$ and $\beta$ are computable then $S''_{\alpha,\beta}$ has the property $\mathcal{H}$ but not the property $\widetilde{\mathcal{H}}$.

The same type of exemples can be constructed with uniform recurrent sequences. 
\end{example}

\begin{theorem}\label{theorem-realization}
Let $\Sigma\subset\A^{\Z^{d'}}$ be a subshift which verifies the conditon $\mathcal{H}$. Then there exists a finite set $F_{\texttt{Final}}$ of $d$-dimensional patterns such that $\sa{d'}{\T_{F_{\texttt{Final}}}}=\Sigma$, that is to say $\cF^{\sft}_{\Sigma,d}$. 

More precisely, if $\Sigma'\subset\Sigma$ is an effective subshift given by the $\mathcal{H}$ property. Assume there exists two $d$-dimensional SFT $\T_F\subset\B^{\Z^d}$ and $\T_{F'}\subset\B'^{\Z^d}$ and two factor maps $\pi:\B^{\Z^d}\to\A^{\Z^d}$ and $\pi':\B'^{\Z^d}\to\A^{\Z^d}$ such that $\Sigma=\sa{d'}{\pi(\T_F)}$ and $\Sigma'=\sa{d'}{\pi'(\T_{F'})}$. 
\begin{itemize}
\item If $x=\sigma^{e_l}(x)$ for all $x\in\pi'(\T_{F'})$ and $l\in[d'+1,d]$, then we can assume that $\varphi_{F_{\texttt{Final}},\id,d\to d'}\sim\max(\varphi_{F,\pi,d\to d'},\varphi_{F',\pi',d\to d'})$.
\item If $\Sigma'\subset\Sigma$ verifies the condition $\widetilde{\mathcal{H}}$, then we can assume that $\varphi_{F_{\texttt{Final}},\id,d\to d'}\sim\max(\varphi_{F,\pi,d\to d'},\varphi_{F',\pi',d\to d'})$.
\end{itemize}

\end{theorem}
\begin{remark}
Since $\Sigma'$ is an effective subshift, by Theorem\ref{Th:ProjectiveSubactionSofic}, it is always possible to find a multidimensional SFT which verifies the first point.
\end{remark}
\begin{proof}
Denote $d''=d-d'$, $\G_{d'}=\Vect{\vect{e_1},\dots,\vect{e_{d'}}}$ and $\G_{d''}=\Vect{\vect{e_{d'+1}},\dots,\vect{e_{d''}}}$. 

\Claim{It is possible to assume without change the speed of convergence that for all $x\in\T_{F'}$ there exists $\overline{x}\in\T_{F'}$ such that for $\vect{j}\in\G_{d''}$ one has \[\M_{\vect{j}}=\left\{\vect{i}\in\G_{d'}:\pi'(\sigma^{\vect{j}}(x))_{\vect{i}}\ne \pi'(\sigma^{\vect{j}}(\overline{x}))_{\vect{i}}\textrm{ and } :\pi'(\sigma^{\vect{j}}(x))_{\vect{i}-\vect{e_l}}\ne \pi'(\sigma^{\vect{j}}(\overline{x}))_{\vect{i}-\vect{e_l}} ,\textrm{ for all } l\in[1,d']\right\}\]
which verifies $\M_{\vect{j}}+[0,N]^{d'}=\G_{d'}$.}\label{Claim.syncro}
\bclaimprf
This claim allows to produces two synchronized configurations of $\T_{F'}$ which allow to code an alphabet. 

If we are in the first point of the Theorem, the result is trivial since $x_{\G_{d'}}=x_{\vect{i}+\G_{d'}}$ for all $\vect{i}\in\G_{d''}$ and using the definition of $\Sigma'$.

In the case of the second point, it is sufficient to consider the subshift $T_{F'}\times\{0,1\}^{\Z^d}$ and the projection $\pi':T_{F'}\times\{0,1\}^{\Z^d}\to\pi(T_{F'})$ defined by $\pi'(x,y)_i=\pi'(x)_i$ if $y_i=0$ and $\pi'(x,y)_i=\psi\circ\pi'(x)_i$ if $y_i=1$ for all $(x,y)\in T_{F'}\times\{0,1\}^{\Z^d}$.
\eclaimprf

\vspace{0.4cm}

Has left to increase the alphabet $\B$, one can assume that $T_F$ and $T_{F'}$ are defined on the same alphabet $\B$. Let $k=\lceil\log_2(\card(\B^{N+1})^{d'})\rceil$ and consider the surjection $\phi:\{0,1\}^{k}\longrightarrow\B^{[0,N]^{d'}}$. For $j,l\in\N$ define:
\[\begin{array}{rlll}
\trans{j,l}:&\Z^d&\longrightarrow&\Z^d\\
&\vect{n}&\longmapsto& \vect{n'}+(j+19\times k\times n+l)\vect{e_{d'+1}}+\vect{n''}\\
&&&\textrm{where }\vect{n}=\vect{n'}+n\vect{e_{d'+1}}+\vect{n''}\textrm{ with }\vect{n'}\in\G_{d'},n\in\Z \textrm{ and }\vect{n''}\in\Vect{\vect{e_{d'+2}},\dots,\vect{e_{d}}}
\end{array}\]

We are going to construct a finite set of $d$-dimensional forbidden patterns on the alphabet $\A$ denoted $F_{\texttt{Final}}=F_{\texttt{Syncro}}\cup F_{\texttt{Code}}\cup F_{\texttt{Compat}}\cup F_{\texttt{Simul}}\cup F_{\texttt{Fact}}$ such that $\sa{d'}{\T_{F_{\texttt{Final}}}}=\Sigma$.

\textbf{Definition of $F_{\texttt{Syncro}}$:} Consider the finite set of patterns $F_{\texttt{Syncro}}$ of $d$-dimensional forbidden patterns on the alphabet $\A$ such that  for all $x\in\T_{F_{\texttt{Syncro}}}$, for all $\vect{n}=\vect{n'}+\vect{n''}\in\Z^d$ with $\vect{n'}\in\G_{d'}$ and $\vect{n''}\in\G_{d''}$ there exists $j\in[0,19k-1]$ such that:
\begin{enumerate}
\item $x_{\trans{j,0}(\vect{n})}=x_{\trans{j,l}(\vect{n})}$ for all $l\in[0,k-1]$;
\item $x_{\trans{j,k}(\vect{n})}=x_{\trans{j,l}(\vect{n})}=x_{\trans{j,l}(\vect{n}+\vect{e_{\ell'}})}$ for all $l\in[k,6k-1]\cup[11k,16k-1]$ and $\ell'\in[d'+1,d]$;
\item $x_{\trans{j,6k}(\vect{n})}=x_{\trans{j,l}(\vect{n})}=x_{\trans{j,l}(\vect{n}+\vect{e_{\ell'}})}$ for all $l\in[6k,11k-1]$ and $\ell'\in[d'+1,d]$;
\item there exists $\vect{i}\in\G_{d'}$ such that $\vect{n'}\in\vect{i}+[0,N]^{d'}$, $x_{\trans{j,k}(\vect{n})}\ne x_{\trans{j,6k}(\vect{n})}$ and $x_{\trans{j,k}(\vect{n}-\vect{e_l})}= x_{\trans{j,6k}(\vect{n}-\vect{e_l})}$ for all $l\in[1,d']$.
\end{enumerate}

The next claim shows that the choose of $j$ does not depend of the position $\vect{n}$ of the pattern considered but only of the configuration $x\in\T_{F_{\texttt{Syncro}}}$.

\Claim{For all $x\in\T_{F_{\texttt{Syncro}}}$, there exists an unique $j\in[0,19k-1]$ which verifies the previous properties independently of the choose of $\vect{n}$, we call it the \define{synchronization index of the configuration $x$}.}\label{Syncro}
\bclaimprf
Let $x\in\T_{F_{\texttt{Syncro}}}$. Consider $j_1\in[0,19k-1]$ (respectively $j_2\in[0,19k-1]$) which verifies the previous properties for $\vect{n_1}\in\Z^d$ (respectively $\vect{n_2}\in\Z^d$). There exists $\vect{n}\in\Z^d$ and $l\in[0,19k-1]$ such that $\trans{j_2,0}(\vect{n_2})=\trans{j_1,0}(\vect{n_1}+\vect{n})+l\vect{e_{d+1}}$.

By periodicity conditions given by (2) and (3), $x_{\trans{j,l'}(\vect{n_1}+\vect{n})+\G_{d'}}$ is the same pattern for $l'\in [k,6k-1]\cup[11k,16k-1]$ and  is different for $l'\in[6k,11k-1]$ and we cannot say for $l'\in [0,k-1]\cup[16k,19k-1]$. The same properties hold for $x_{\trans{j,l'}(\vect{n_2})+\G_{d'}}$. One deduces that $l=0$ since if not there is a superposition which does not hold. That is to say $j_1=j_2$ 
\eclaimprf

Thus the local rules $F_{\texttt{Syncro}}$ divide the space in $19k$ slices according the direction $\vect{e_{d'+1}}$. Figure~\ref{figure.SlicesSFT} represent the different slices. For $x\in\T_{F_{\texttt{Syncro}}}$, referring to the position index associated $j\in[0,19k-1]$, $F_{\texttt{Syncro}}$ fixes the following things for $\vect{n}\in\Z^d$:
\begin{itemize}
\item condition $(1)$ fixes the same configuration in the slices $\trans{j,l}(\vect{n})+\G_{d'}$ with $l\in[0,k-1]$, in these slices will appear a configuration of $\pi(\T_F)$;
\item condition $(2)$ fixes the same configuration in the slices $\trans{j,l}(\vect{n})+\G_{d'}$ for all $l\in[k,6k-1]\cup[11k,16k-1]$ and condition $(3)$ fixes the same configuration in the slices $\trans{j,l}(\vect{n})+\G_{d'}$ for all $l\in[6k,11k-1]$, in these slices will appear the configurations $z^{\ast}$ and $\overline{z^{\ast}}$ and the differences are imposed by condition $(4)$;
\item the repetition of the different lines allows to synchronize the configuration in view to obtain the unicity of   $j\in[0,19k-1]$ as it is shown in Claim~\ref{Syncro}. 
\end{itemize}
Let $x\in\T_{F_{\texttt{Syncro}}}$, $j\in[0,19k-1]$ the synchronization index associated and $\vect{n''}\in\G_{d''}$. One defines:
\begin{align*}
z^{\star}(x,\vect{n''})&= x_{\G_{d'}+\trans{j,k}(\vect{n''})}\in\A^{\Z^{d'}};\\
\overline{z^{\star}}(x,\vect{n''})&=x_{\G_{d'}+\trans{j,6k}(\vect{n''})}\in\A^{\Z^{d'}};\\
 \M(x,\vect{n''})&=\left\{\vect{i}\in\G_{d'} : x_{\trans{j,k}(\vect{i}+\vect{n''})}\ne x_{\trans{j,6k}(\vect{i}+\vect{n''})}\textrm{ and } x_{\trans{j,k}(\vect{i}-\vect{e_l}+\vect{n''})}= x_{\trans{j,6k}(\vect{i}-\vect{e_l}+\vect{n''})},\ \forall l\in[1,d']\right\}.
\end{align*}

The configurations $z^{\star}_{\vect{n''}}(x)$ and $\overline{z^{\star}_{\vect{n''}}}(x)$ give the base for coding the alphabet $\B$ and the set $\M_{\vect{n''}}(x)$ gives the position where the coding is done. 

\textbf{Definition of $F_{\texttt{Code}}$:} We are going to define the condition $F_{\texttt{Code}}$ which allows to code the alphabet $\B$ and the the conditions verified by the subshifts $\T_F$ and $\T_{F'}$. For $x\in\T_{F_{\texttt{Syncro}}}$ with synchronization index $j\in[0,19k-1]$ and $\vect{n}=\vect{n'}+\vect{n''}\in\Z^d$ with $\vect{n'}\in\G_{d'}$ and $\vect{n''}\in\G_{d''}$, the forbidden patterns $F_{\texttt{Code}}$ verifies
\begin{enumerate}
\item if $\vect{n'}\in\M(x,\vect{n''})$ then for all $l\in[16,19k-1]$ one has $x_{\trans{j,l}(\vect{n'})+[0,N']^{d'}}=x_{\trans{k,l}(\vect{n'})+[0,N']^{d'}}=z^{\ast}(x,\vect{n''})_{\vect{n'}+[0,N']^{d'}}$ or $x_{\trans{6k,l}(\vect{n'})+[0,N']^{d'}}=\overline{z^{\ast}}(x,\vect{n''})_{\vect{n'}+[0,N']^{d'}}$;
\item if $\vect{n'}\notin\M(x,\vect{n''})+[0,N']^{d'}$ then $x_{\trans{j,l}(\vect{n})}=x_{\trans{j,k}(\vect{n})}$ for all $l\in[k,19k-1]$.
\end{enumerate}

The local rule $F_{\texttt{Code}}$ allows to use the difference between $z^{\ast}(x,\vect{n''})$ and $\overline{z^{\ast}}(x,\vect{n''})$ to code binary $\B^{[0,N]}$ in slices $x_{\trans{j,l}(\vect{n''})+\G_{d'}}$ with $l\in[16,19k-1]$. More precisely, 
\begin{itemize}
\item the slices $x_{\trans{j,l}(\vect{n''})+\G_{d'}}$ with $l\in[16,17k-1]$ codes the configuration contained in  $x_{\trans{j,0}(\vect{n''})+\G_{d'}}=x_{\trans{j,l'}(\vect{n''})+\G_{d'}}$ for all $l'\in[0,k-1]$;
\item the slices $x_{\trans{j,l}(\vect{n''})+\G_{d'}}$ with $l\in[17,18k-1]$ codes the configuration contained in  $x_{\trans{k,l}(\vect{n''})+\G_{d'}}=x_{\trans{j,l'}(\vect{n''})+\G_{d'}}$ for all $l\in[k,6k-1]\cup[11k,16k-1]$;
\item the slices $x_{\trans{j,l}(\vect{n''})+\G_{d'}}$ with $l\in[18,19k-1]$ codes the configuration contained in  $x_{\trans{6k,l}(\vect{n''})+\G_{d'}}=x_{\trans{j,l'}(\vect{n''})+\G_{d'}}$ for all $l\in[6k,11k-1]$.
\end{itemize}

Let $x\in\T_{F_{\texttt{Syncro}}\cup F_{\texttt{Code}}}$ and $j$ be the synchronization index associated. For $l\in[0,2]$, $\vect{i''}\in\G_{d''}$ and $\vect{i'}\in\M(x,\vect{i''})$ define 
\[
\overline{\phi_l}(x)_{\vect{i'}+\vect{i''}}=\phi(\alpha_0,...,\alpha_{k-1})\in\B^{[0,N]^{d'}} \textrm{ where } \alpha_p=\begin{cases}1 &\textrm{ if } x_{\trans{j,(16+l)k+p}(\vect{i'}+\vect{i''})}=z^{\star}(x,\vect{i''})_{\vect{i'}}\\
0 &\textrm{ if } x_{\trans{j,(16+l)k+p}(\vect{i'}+\vect{i''})}=\overline{z^{\star}}(x,\vect{i''})_{\vect{i'}}
\end{cases}\textrm{ for all }p\in[0,k-1].
\]

\textbf{Definition of $F_{\texttt{Compat}}$:} Consider the finite set of patterns $F_{\texttt{Compat}}$ such that for all $x\in\T_{F_\texttt{Syncro}\cup F_{\texttt{Code}} \cup F_\texttt{Compat}}$ and $\vect{i''}\in\G_{d''}$, if $\vect{i_1}+\vect{j_1}=\vect{i_2}+\vect{j_2}$ where $\vect{i_1},\vect{i_2}\in\M(x,\vect{i''})$ and $\vect{j_1},\vect{j_2}\in[0,N]^{d'}$ one has
\[ \left(\overline{\phi_l}(x)_{\vect{i_1}+\vect{i''}}\right)_\vect{j_1}=\left(\overline{\phi_l}(x)_{\vect{i_2}+\vect{i''}}\right)_\vect{j_2} \textrm{ for all }l\in[0,2].\]

Thus, for $x\in\T_{F_\texttt{Syncro}\cup F_{\texttt{Code}}\cup F_\texttt{Compat}}$ and $l\in[0,2]$ one can define the function 
\[
\begin{array}{rccll}
\widetilde{\phi_l}:&\T_{F_\texttt{Syncro}\cup F_{\texttt{Code}}\cup F_\texttt{Compat}}&\longrightarrow&\B^{\Z^d}&\\
&x&\longmapsto&\widetilde{\phi_l}(x)&\textrm{ where }\widetilde{\phi_l}(x)_{\vect{i}+\vect{j}+\vect{i''}}=\left(\overline{\phi_l}(x)_{\vect{i}+\vect{i''}}\right)_{\vect{j}} \textrm{ with }\vect{i''}\in\G_{d''},\vect{i}\in\M(x,\vect{i''}) \textrm{ and } \vect{j}\in[0,N]^{d'}.
\end{array}
\]

\textbf{Definition of $F_{\texttt{Simul}}$:} Then consider the finite set of patterns $F_{\texttt{Simul}}$ such that for all $x\in\T_{F_\texttt{Syncro}\cup F_\texttt{Code}\cup F_\texttt{Compat}\cup F_{\texttt{Simul}}}$ one has $\widetilde{\phi_0}(x)\in\T_{F}$ and $\widetilde{\phi_1}(x),\widetilde{\phi_2}(x)\in\T_{F'}$. 

\textbf{Definition of $F_{\texttt{Fact}}$:} To finish, define the finite set of patterns $F_{\texttt{Fact}}$ such that for all $x\in\T_{F_\texttt{Syncro} \cup F_\texttt{Code}\cup F_\texttt{Compat}\cup F_{\texttt{Simul}}\cup F_{\texttt{Fact}}}$, one has $\pi(\widetilde{\phi_0}(x))_\vect{n}=x_{\trans{j,0}(\vect{n})}$, $\pi'(\widetilde{\phi_1}(x))_\vect{n}=x_{\trans{j,k}(\vect{n})}$ and $\pi'(\widetilde{\phi_2}(x))_\vect{n}=x_{\trans{j,6k}(\vect{n})}$

\textbf{In brief:} The previous forbidden patterns involve the following properties on $x\in\T_{\widetilde{F}}$ and $\vect{n''}\in\G_{d''}$:
\begin{itemize}
\item $F_\texttt{Syncro}$ fixes the index synchronization $j$, fixes the same configuration $z^{\star}(x,\vect{n''})$ and $\overline{z^{\star}}(x,\vect{n''})$ which give a reference to code $\B^{[0,N]}$;
\item $F_\texttt{Code}$ verifies the correspondence between $z^{\star}(x,\vect{n''})$ and $\overline{z^{\star}}(x,\vect{n''})$ in the strips coding;
\item $F_\texttt{Compat}$ allows to verify the compatibility with redundant coding of the alphabet $\B$;
\item  $F_{\texttt{Simul}}$ verifies that the condition $F$ and $F'$ are checked to produce one element of $\T_F$ and two elements of $\T_{F'}$;
\item  $F_{\texttt{Fact}}$ verifies that the subshifts given by  $F_{\texttt{Simul}}$ produces the desired element after projection on the corresponding slice.
\end{itemize}

%\figex{
\begin{figure}[h!]
\begin{center}
\begin{tiny}
 \begin{tikzpicture}
\begin{scope}[xshift=-3cm,yshift=1cm,scale=0.6]
\draw[->, >=latex] (0,0)--(1,0);
\draw (0.5,0) node[below] {$\vect{e_1}$};
\draw[->, >=latex] (0,0)--(0,1);
\draw (0,0.5) node[left] {$\vect{e_{d'+1}}$};
\end{scope}

\begin{scope}[scale=0.45]

\foreach \i in {-1,0,1}{
\begin{scope}[yshift=9.5*\i cm]

\fill[blue!20] (0,0.5)rectangle(27,3);
\draw (14,1.75) node {$z^{\star}(x)$};
\fill[blue!20] (0,5.5)rectangle(27,8);
\draw (14,6.75) node {$z^{\star}(x)\in\Sigma'$};
\fill[yellow!20] (0,3)rectangle(27,5.5);
\draw (14,4.25) node {$\overline{z^{\star}}(x)\in\Sigma'$};

\fill[red!40] (0,0)rectangle(27,0.5);
\draw (14,0.25) node {$y_{\i}\in\Sigma$};

\fill[red!80] (0,8)rectangle(27,8.5);
\draw (14,8.25) node {coding of $y_{\i}$};
\fill[blue!80] (0,8.5)rectangle(27,9);
\draw (14,8.75) node {coding of $z^{\star}(x,\i)$};
\fill[yellow!80] (0,9)rectangle(27,9.5);
\draw (14,9.25) node {coding of $\overline{z^{\star}}(x,\i)$};
\end{scope}
}
\foreach \x in {9.5,0,-9.5}{
\begin{scope}[yshift=\x cm]
\draw[thick,dashed] (0,0)--(29,0);
\draw[<->] (27.5,0)--(27.5,0.5);
\draw[<->] (27.5,0.5)--(27.5,3);
\draw[<->] (27.5,3)--(27.5,5.5);
\draw[<->] (27.5,5.5)--(27.5,8);
\draw[<->] (27.5,8)--(27.5,8.5);
\draw[<->] (27.5,8.5)--(27.5,9);
\draw[<->] (27.5,9)--(27.5,9.5);

\draw (27.5,0.25) node[right] {$k$};
\draw (27.5,1.75) node[right] {$5k$};
\draw (27.5,4.25) node[right] {$5k$};
\draw (27.5,6.75) node[right] {$5k$};
\draw (27.5,8.25) node[right] {$k$};
\draw (27.5,8.75) node[right] {$k$};
\draw (27.5,9.25) node[right] {$k$};
\end{scope}}

\draw (0,0) node[left] {$j$};
\draw (0,9.5) node[left] {$j+19k$};
\draw (0,-9.5) node[left] {$j-19k$};

\foreach \x in {1,3,4,6,8,10,11,13,16,17,20,21,22,25,26}{
\draw[dotted] (\x,9.5)--(\x,19);
}
\foreach \x in {2,5,7,8,10,13,15,16,18,21,23,25}{
\draw[dotted] (\x,0)--(\x,9.5);
}
\foreach \x in {1,3,5,6,7,9,10,12,15,16,18,21,23,24,26}{
\draw[dotted] (\x,-9.5)--(\x,0);
}
\draw (1,-9.5) node[below] {$i_{-3}$};
\draw (3,-9.5) node[below] {$i_{-2}$};
\draw (5,-9.5) node[below] {$i_{-1}$};
\draw (6,-9.5) node[below] {$i_{0}$};
\draw (7,-9.5) node[below] {$i_{1}$};
\draw (9,-9.5) node[below] {$i_{2}$};
\draw (10,-9.5) node[below] {$i_{3}$};
\draw (12,-9.5) node[below] {$i_{4}$};
\draw (15,-9.5) node[below] {$i_{5}$};
\draw (16,-9.5) node[below] {$i_{6}$};
\draw (18,-9.5) node[below] {$i_{7}$};
\draw (21,-9.5) node[below] {$i_{8}$};
\draw (23,-9.5) node[below] {$i_{9}$};
\draw (24,-9.5) node[below] {$i_{10}$};
\draw (26,-9.5) node[below] {$i_{11}$};

\fill[pattern = north west lines] (1,18.5)rectangle(1.5,19);
\fill[pattern = north west lines] (1,12.5)rectangle(4,15);
\path[->] (1,18.75) edge[bend right=100] node[left] {code} (1,13.75);

\fill[pattern = north east lines] (1,18)rectangle(1.5,18.5);
\fill[pattern = north east lines] (1,10)rectangle(4,12.5);
\fill[pattern = north east lines] (1,15)rectangle(4,17.5);
\path[->] (1,18.25) edge[bend right=10] node[left] {} (1,11.25);
\path[->] (1,18.25) edge[bend right=10] node[left] {} (1,16.25);

\fill[pattern = crosshatch] (1,17.5)rectangle(1.5,18);
\fill[pattern = crosshatch] (1,9.5)rectangle(4,10);
\path[->] (1,17.75) edge[bend right] node[left] {code} (1,9.75);

\end{scope}

\end{tikzpicture}
\end{tiny}
\end{center}
\caption{Part of $x\in\T_{\widetilde{F}}$ corresponding to the different slices which appear in $[i_{-4},i_{12}]\vect{e_1}\times[j,j+21k]\vect{e_{d'+1}}$. The dotted vertical lines represent differences which allow to code.  }\label{figure.SlicesSFT}
\end{figure}
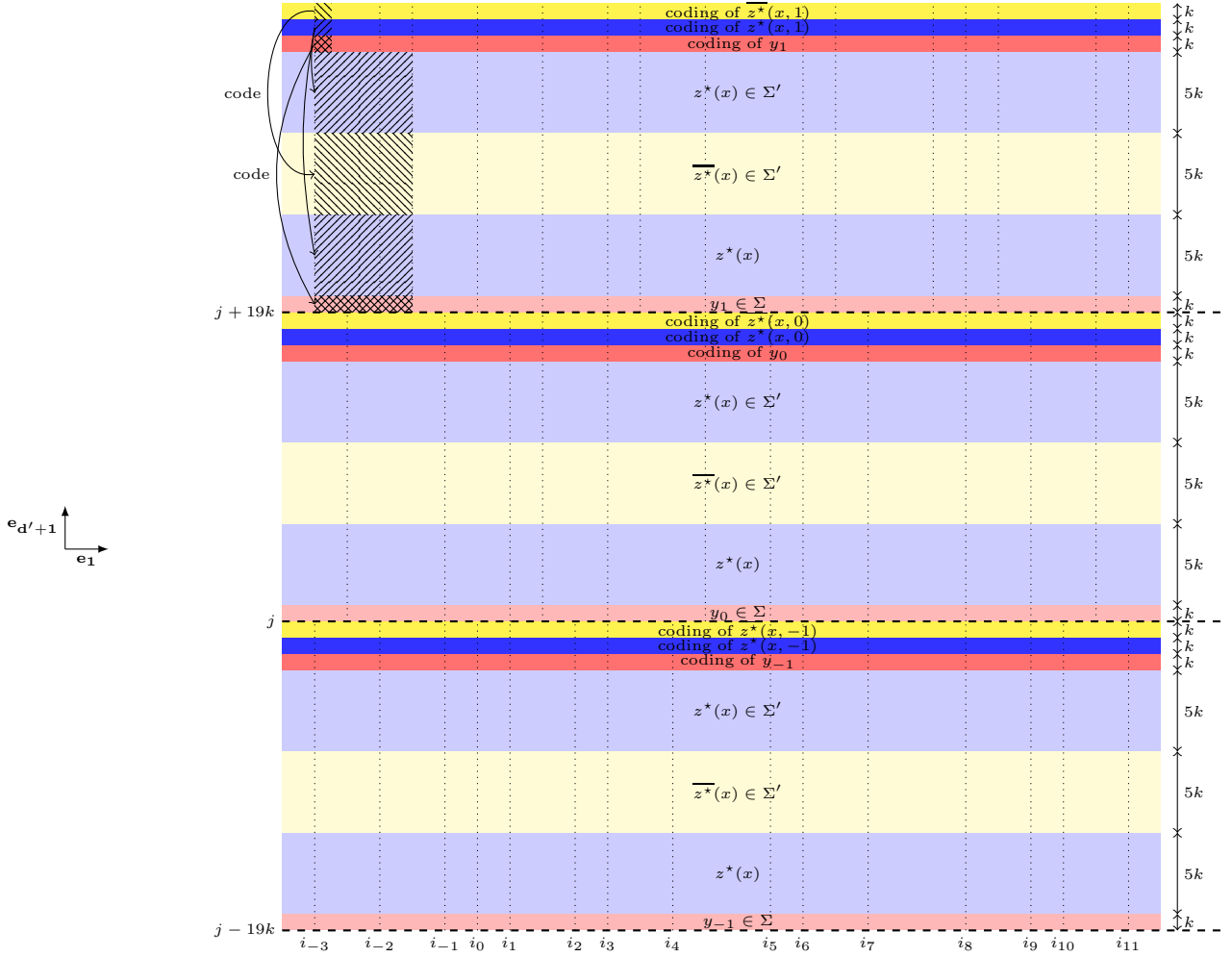
%}

\Claim{If $x\in\T_{F_{\texttt{Final}}}$, then $\sa{d'}{x}\in\Sigma$.}
\bclaimprf
Let $x\in\T_{F_{\texttt{Final}}}$ and $j$ the synchronization index associated. One has:
\begin{itemize}
\item since $\widetilde{\phi_0}(x)\in\T_{F}$ and $\pi(\widetilde{\phi_0}(x))_{\vect{n''}+\G_{d'}}=x_{\trans{j,0}(\vect{n''})+\G_{d'}}$ for all $\vect{n''}\in\G_{d''}$, one deduces that $x_{\trans{j,0}(\vect{n''})+\G_{d'}}\in\Sigma$, thus  by  $F_\texttt{Syncro}$ one has $x_{\trans{j,l}(\vect{n''})+\G_{d'}}\in\Sigma$  for all $l\in[0,k-1]$;
\item since $\widetilde{\phi_1}(x)\in\T_{F'}$ and $\pi'(\widetilde{\phi_1}(x))_{\vect{n''}+\G_{d'}}=x_{\trans{j,k}(\vect{n''})+\G_{d'}}$ for all $\vect{n''}\in\G_{d''}$, one deduces that $x_{\trans{j,k}(\vect{n''})+\G_{d'}}=z^{\star}(x,\vect{n''})\in\Sigma'$, thus by $F_\texttt{Syncro}$ one has $x_{\trans{j,l}(\vect{n''})+\G_{d'}}\in\Sigma'$  for all $l\in[k,6k-1]\cup[11k,16k-1]$;

\item in the same way, $\widetilde{\phi_2}(x)\in\T_{F'}$ and $\pi'(\widetilde{\phi_2}(x))_{\vect{n''}+\G_{d'}}=x_{\trans{j,6k}(\vect{n''})+\G_{d'}}$ for all $\vect{n''}\in\G_{d''}$, one deduces that $x_{\trans{j,1}(\vect{n''})+\G_{d'}}=\overline{z^{\star}}(x,\vect{n''})\in\Sigma'$, thus by $F_\texttt{Syncro}$ one has $x_{\trans{j,l}(\vect{n})+\G_{d'}}\in\Sigma'$  for all $l\in[6k,11k-1]$;

\item by  $F_\texttt{Code}$, for all $l\in[16k,19k-1]$ and for all $\vect{n''}\in\G_{d''}$, if $\vect{i}\in\M(x,\vect{n''})$ one has $x_{\vect{i}+\trans{j,l}(\vect{n''})+[0,N']^{d'}}=x_{\vect{i}+\trans{j,k}(\vect{n''})+[0,N']^{d'}}$ or $x_{\vect{i}+\trans{j,6k}(\vect{n''})+[0,N']^{d'}}$ and if $\vect{i}\in\G_{d'}\setminus(\M(x,\vect{n''})+[0,N']^{d'})$ one has $x_{\vect{i}+\trans{j,l}(\vect{n''})}=x_{\vect{i}+\trans{j,k}(\vect{n''})}$; by the definition of $\Sigma'$ given by the properties $\mathcal{H}$, one deduces that $x_{\trans{j,l}(\vect{n''})+\G_{d'}}\in\Sigma'\subset\Sigma$ for all $l\in[16k,19k-1]$. 
\end{itemize}
Thus $x_{\vect{n''}+\G_{d'}}\in\Sigma$ for all $\vect{n''}\in\G_{d''}$.
\eclaimprf

\Claim{For all $y\in\Sigma\subset\A^{\Z^{d'}}$, there exists $x\in\T_{F_{\texttt{Final}}}\subset\A^{\Z^d}$ such that $y=\sa{d'}{x}\in\Sigma$.}
\bclaimprf
Consider $x^1\in\T_F\subset\B^{\Z^d}$ such that $\sa{d'}{\pi(x^1)}=y$
By the property $\mathcal{H}$ and Claim~1, there exists $x^2,\overline{x^2}\in\T_{F'}$ such that for all $\vect{j}\in\G_{d''}$ one has \[\M_{\vect{j}}=\left\{\vect{i}\in\G_{d'}:\pi'(\sigma^{\vect{j}}(x^2))_{\vect{i}}\ne \pi'(\sigma^{\vect{j}}(\overline{x^2}))_{\vect{i}}\textrm{ and } :\pi'(\sigma^{\vect{j}}(x^2))_{\vect{i}-\vect{e_l}}\ne \pi'(\sigma^{\vect{j}}(\overline{x^2}))_{\vect{i}-\vect{e_l}} ,\textrm{ for all } l\in[1,d']\right\}\]
which verifies $\M_{\vect{j}}+[0,N]^{d'}=\G_{d'}$.

Define $x\in\A^{\Z^d}$ such that:
\begin{itemize}
\item $x_{\G_{d'}+\trans{0,l}(\vect{n''})}=\pi(x^1)_{\vect{n''}+\G_{d'}}$ for all $\vect{n''}\in\G_{d''}$ and $l\in[0,k-1]$;
\item $x_{\G_{d'}+\trans{0,l}(\vect{n''})}=x^2_{\vect{n''}+\G_{d'}}$ for all $\vect{n''}\in\G_{d''}$ and $l\in[k,6k-1]\cup[11k,16k-1]$;
\item $x_{\G_{d'}+\trans{0,l}(\vect{n''})}=\overline{x^2}_{\vect{n''}+\G_{d'}}$ for all $\vect{n''}\in\G_{d''}$ and $l\in[6k,11k-1]$;
\item if $\phi(\alpha_0,\dots,\alpha_{k-1})=x^1_{\vect{i}+\vect{n''}+[0,N]^{d''}}$ with $\vect{i}\in\M_{\vect{n''}}$ and $\vect{n''}\in\G_{d''}$ then 
for $l\in[0,k-1]$ one has $x_{\vect{i}+\trans{0,16k+l}(\vect{n''})+[0,N]^{d'}}=x^2_{\vect{i}+[0,N]^{d'}}$ if $\alpha_l=1$ and $x_{\vect{i}+\trans{0,15k+l}(\vect{n''})+[0,N]^{d'}}=\overline{x^{2}}_{\vect{i}+[0,N]^{d'}}$ if not. 
\item if $\phi(\alpha_0,\dots,\alpha_{k-1})=x^2_{\vect{i}+\vect{n''}+[0,N]^{d''}}$ with $\vect{i}\in\M_{\vect{n''}}$ and $\vect{n''}\in\G_{d''}$ then 
for $l\in[0,k-1]$ one has $x_{\vect{i}+\trans{0,17k+l}(\vect{n''})+[0,N]^{d'}}=x^{2}_{\vect{i}+[0,N]^{d'}}$ if $\alpha_l=1$ and $x_{\vect{i}+\trans{0,15k+l}(\vect{n''})+[0,N]^{d'}}=\overline{x^2}_{\vect{i}+[0,N]^{d'}}$ if not. 
\item if $\phi(\alpha_0,\dots,\alpha_{k-1})=\overline{x^2}_{\vect{i}+\vect{n''}+[0,N]^{d''}}$ with $\vect{i}\in\M_{\vect{n''}}$ and $\vect{n''}\in\G_{d''}$ then 
for $l\in[0,k-1]$ one has $x_{\vect{i}+\trans{0,18k+l}(\vect{n''})+[0,N]^{d'}}=x^2_{\vect{i}+[0,N]^{d'}}$ if $\alpha_l=1$ and $x_{\vect{i}+\trans{0,15k+l}(\vect{n''})+[0,N]^{d'}}=\overline{x^{2}}_{\vect{i}+[0,N]^{d'}}$ if not. 
\end{itemize}

It is easy to verify that $x\in\T_{F_{\texttt{Final}}}$ and $y=\sa{d'}{x}\in\Sigma$, moreover the speed of convergence is the worst between $\varphi_{F,\pi,d\to d'}$ and $\varphi_{F',\pi',d\to d'}$.

\eclaimprf
\end{proof}

In the majority of the examples of Section~\ref{section:example}, $\Sigma'$ is a transitive SFT so it could be realized with constant speed of convergence. In a larger sense one has the following corollary. 
\begin{corollary}
For all $d'$-multidimensional subshift $\Sigma$ which verifies the condition $\widetilde{\mathcal{H}}$ where the sharp speed of convergence of $\Sigma'$ is better than $\Sigma$ then $\cF^{\sft}_{\Sigma,d\to d'}\sim\cF^{\sofic}_{\Sigma,d\to d'}$ for all $d\geq d'+1$.
\end{corollary}

\begin{example}
Consider the sequence of patterns such that $u_0=1$ and $u_{n+1}=u_n^{\alpha_n}0^{n+1}u_n^{\alpha_n}$ with $\alpha_n=10^{2^n}$ for all $n\in\N$. Consider the subshift $\Sigma=\left\{x\in\{0,1\}^{\Z} : \forall u\sqsubset x, \exists n\in\N \textrm{ such that }u\sqsubset u_n\right\}$ and the one-to-one block map $\pi:\{0,1,1'\}^{\Z}\to\{0,1\}^{\Z}$ such that $\pi(1)=1=\pi(1')$ and $\pi(0)=0$. Define $\Sigma'= \left\{x\in\{0,1,1'\}^{\Z} : \exists y\in\Sigma \textrm{ such that }\pi(x)=y\right\}$ and consider $\Sigma''$ an effective subshift which cannot be realized by projective subaction with minimal projective subaction (such subshifts exists~\cite{Pavlov-Schraudner-2010}).
Clearly $\Sigma'$ is an effective subshift which does not verify the condition $\mathcal{H}$. We are going to see that $\Sigma'\times\Sigma''$ has positive entropy but cannot be realized by projective subaction of a SFT.

\Claim{$\Sigma'$ has positive entropy.}
\bclaimprf
Denote $|u|_1$ the number of $1$ in a word $u\in\{0,1\}^{\ast}$. One has $|u_{n}|_1=2\alpha_n|u_{n-1}|_1=2^n\prod_{i=1}^n\alpha_i$ and $|u_{n}|=2\alpha_nu_{n-1}+n\leq2(\alpha_n+1)u_{n-1}=2^n\prod_{i=1}^n(\alpha_i+1)$.
syndetic
One has $\card(\Lang_{|u_n|}(\Sigma'))\geq 2^{|u_n|_1}$ thus $$h(\Sigma')\geq\lim_{n\to\infty} \frac{|u_n|_1}{|u_n|}\geq\lim_{n\to\infty} \frac{\prod_{i=1}^n\alpha_i}{\prod_{i=1}^n(\alpha_i+1)}>\frac{1}{2}>0.$$
\eclaimprf

\Claim{Let $\T$ be a $2$-dimensional subshift such that $\sa{1}{\T}=\Sigma'$, then for all $x\in\T$, for all $r\in\N$ and for $n<m$, the set $\M^r_{[n,m]}(x)=\left\{i\in\Z: x_{[i,i+r]\times [n,m]}=0^{[0,r]\times[n,m]}\right\}$ is $\alpha'_{r,m-n}$-syndetic where $\alpha'_{r,0}=\alpha_{2r}$ and $\alpha'_{r,n+1}=\alpha'_{r,n}+\alpha_{2(\alpha'_{r,n}+r)}+r$. We recall that a $\alpha$-syndetic set is a subset of integer such that the sizes of the gaps in the sequence are bounded by $\alpha$.}
\bclaimprf
Let $x\in\T$ and $r\in\N$, the set $\M^r_{\{0\}}(x)$ is $\alpha_{2r}$-syndetic since every sub-pattern of $\pi(x)$ is a sub-pattern of $u_n$ for some $n\in\N$ and by induction, for $n\geq r$ two consecutive sequences of $0^r$ in $u_n$ is separated at most by $2\alpha_r|u_{r-1}|\leq|u_r|\leq2^n\prod_{i=1}^n(\alpha_i+1)\leq \alpha_{2r}$ symbols. 

Assume that the set $\M^r_{[0,n]}(x)$ is $\alpha'_{r,n}$-syndetic. If there exists $i\in\Z$ such that $x_{[i,i+\alpha'_{r,n}+r-1]\times\{n+1\}}=0^{\alpha'_{r,n}+r}$ then there exists $j\in[i,i+\alpha'_{r,n}]$ such that $x_{[j,j+r]\times[0,n+1]}=0^{[0,r]\times[0,n+1]}$. Since $\M^{\alpha'_{r,n}+r}_{\{n+1\}}(x)$ is $\alpha_{2(\alpha'_{r,n}+r)}$-syndetic, one deduces that  $\M^r_{[0,n+1]}(x)$ is $\alpha_{2(\alpha'_{r,n}+r)}+\alpha'_{r,n}+r$-syndetic. The Claim results by induction.

\eclaimprf

\Claim{$\Sigma'\times\Sigma''$ cannot be realized by projective subaction of a SFT.}
\bclaimprf
Assume that $\Sigma'\times\Sigma''$ is realized by projective subaction of a $2$-dimensionnal SFT denoted $\T$ and consider $\T''$ the projection according the second coordinate. Since $\Sigma'$ has arbitrary large zone with only 0 by the previous Claim, one deduces that the local rules which define $\T$ can be used to define $\{\textbf{0}\}\times\T''\subset\T$ where $\textbf{0}$ is the configuration with only $0$. The projective subaction of  $\{\textbf{0}\}\times\T''$ is include in $\{\textbf{0}\}\times\Sigma''$ so by minimality there is equality. One deduces that $\Sigma''$ can be realized by projective subaction of a SFT which is a contradiction.

\eclaimprf

\end{example}

\section{Speed of convergence in general constructions}\label{section:GeneralConstruction}

\subsection{Notion of  Turing Machine}

In this section we recall some properties on Turing machine and algorithmic complexity.

\begin{definition}
Formally, a Turing machine with $k$ tapes $\TM=(k,Q,\Gamma,\#,q_0,\delta,Q_F)$ is defined by:
\begin{itemize}
\item $\Gamma$ a finite alphabet, with a blank symbol $\#\in\Gamma$. Initially, $k$ infinite memory tapes represented as an element of $(\Gamma^k)^{\Z}$, are filled with $\#$, except for a finite prefix on the first tape (the input), and a computing head is located on the first letter of the tape;
\item $Q$ the finite set of states of the head and  $q_0\in Q$ is the initial state;
\item $\delta : Q\times\Gamma^k\to Q\times\Gamma^k\times\{\leftarrow,\cdot\,
,\rightarrow\}^k$ the transition function. Given the state of the head and the letter associated, it reads on the tape, depending on its position, the head can change state, replace the letter and move by one cell at most.
\item $Q_F\subset Q$ the set of final states, when a final state is
reached, the computation stops and the output is the value currently written on the tape. 
\end{itemize}
\end{definition}

Turing machines are a very robust model of computation. We presented above only one of the possible definitions, but there exist several variants in the literature that are all equivalent from a decidability point of view (i.e. these variants all define the same recursive sets). Nevertheless these modifications on the definition are not without effects on the time and space complexities (time unit is one application of the transition function, space unit is one cell of the tape). 

To detect forbidden patterns in the projective subaction, one of the fundamental construction is the use of SFT to encode Turing machine computations. In this article we choose to use the basic version of TM because it provides simpler constructions when encoding computations inside an SFT. But the reader should have in mind that it is possible to improve time and space complexities, using by instance the non-exhaustive following acceleration techniques:

\begin{itemize}
 \item[\textbf{Compare-Copy:}] It is possible to compare or copy instantaneously a word between two markers from a tape to another one.
  \item[\textbf{Transfert head:}] It is possible to transfer instantaneously  the head to another cell of the tape marked by a special symbol. 
  \item[\textbf{Fill:}] It is possible to fill instantaneously a part of a tape with a periodic pattern.  
\end{itemize}

A set of patterns $F\subset\A^{\ast}$ is \define{recursive} if there exists a Turing machine such that on the input $u\in\A^{\ast}$ gives back $1$ if $u\in F$ and $0$ otherwise. A set of patterns $F\subset\A^{\ast}$ is \define{recursively enumerable} if there exists a Turing machine that, on the input $u\in\A^{\ast}$, gives back $1$ if $u\in F$ and does not halt otherwise.

Let $F$ be a recursively enumerable set of forbidden patterns.Then the complementary rectangular language of $\T_F\subset\A^{\Z^d}$, denoted $\Lang_{\textrm{rect}}(\T_F)^c=\Lang_{\textrm{rect}}(\A^{\Z^d})\smallsetminus\Lang_{\textrm{rect}}(\T_F)$, is also recursively enumerable. Let $\TM_{\Lang_{\textrm{rect}}(\T_F)^c}$ be a Turing machine associated to $\Lang_{\textrm{rect}}(\T_F)^c$, denote
\begin{itemize}
 \item $\Dtime_{\TM_{\Lang_{\textrm{rect}}(\T_F)^c}}(k)$ the maximal time needed by the Turing machine $\TM_{\Lang_{\textrm{rect}}(\T_F)^c}$ to know if a pattern of size $k$ is not in the language of $\T_F$;
 \item $\Dspace_{\TM_{\Lang_{\textrm{rect}}(\T_F)^c}}(k)$ the maximal space needed by the Turing machine $\TM_{\Lang_{\textrm{rect}}(\T_F)^c}$ to know if a pattern of size $k$ is not in the language of $\T_F$ (we just take in consideration space necessary for the computation, thus the input is considered to be in an auxiliary tape).
%\item $\Ntime_{\Sigma}(k)$ the time necessary to a non-deterministic Turing machine to know if a pattern of size $k$ is in the language of $\Sigma$; 
\end{itemize}
Of course $\Dtime_{\TM_{\Lang_{\textrm{rect}}(\T_F)^c}}(k)$ and $\Dspace_{\TM_{\Lang_{\textrm{rect}}(\T_F)^c}}(k)$ are not computable if $\Lang_{\textrm{rect}}(\T_F)^c$ is not recursive. 

If $F$ is a recursively enumerable set of patterns, by definition there exists a Turing machine $\TM_F=(k,Q,\Gamma,\#,q_0,\delta,Q_F)$ with the following behavior: it starts on the empty tape and successively writes the patterns of $F$ on its tape. Each time a word is entirely written, the machine enters a special state before starting again the enumeration -- the machine may stay a given time in this special state. This machine is called an \define{enumerative Turing machine of $F$}. Every recursively enumerable language admits an enumerative Turing machines.

A set of finite patterns $F$ \define{forbids the pattern $w$} if $w\notin\Lang(\T_F)$, equivalently if there exists $n\in\N$ such that, for all $u\in\A^{[0;n-1]^d}$ such that $w\sqsubset u$, there exists $u'\in F$ which verifies $u'\sqsubset u$. 

Let $\TM_F$ be an enumerative Turing machine of $F$, denote $\Dtime^{\textrm{enu}}_{\TM_F}(k)$ (respectively $\Dspace^{\textrm{enu}}_{\TM_F}(k)$)  the smallest time (respectively the smallest space) taken by the Turing machine $\TM_F$ such that the subset of $F$ generated at this time forbid all the words of size $n$ of $\Lang_{n}(\T_F)^c=\Lang_{n}(\A^{\Z^d})\smallsetminus\Lang_{n}(\T_F)$.

\subsection{Speed of convergence for previous constructions}

In this section, we just give some elements to determine the speed of convergence given by the construction of M. Hochman~\cite{Hochman-2009} and N. Aubrun and M. Sablik~\cite{Aubrun-Sablik-2010}.

The idea of the two proofs is to construct a d-dimensional subshift of finite type, denoted $\T_{\texttt{Final}}$, which realizes a given effective subshift $\Sigma\subset \A^{\Z}$ in one direction after a projection where $d=3$ in~\cite{Hochman-2009} and $d=2$ in~\cite{Aubrun-Sablik-2010}. In the two constructions the subshift $\T_{\texttt{Final}}$ is constituted by three layers:
\begin{itemize}
\item the first one is $\A^{\Z^d}$ and contains different copies of the same configuration $y\in\A^{\Z}$ superposed on additional directions, the additional finite type conditions check if $y\in\Sigma$; 
\item the second is $\T_{\texttt{Grid}}\subset\A_{\texttt{Grid}}^{\Z^d}$ and constructs a grid which allows to implement well initialized Turing machine in all the configuration with different size of time and space for the computation;
\item the third is $\T_{\TM}\subset\A_{\TM}^{\Z^d}$ and cheeks if no forbidden pattern appears: since $\Sigma$ is an effective subshift, its forbidden patterns can be enumerated by a Turing machine, thus the purpose is to implement a Turing machine $\TM_F$ which enumerates these forbidden patterns and an additional procedure $\TM_\texttt{Search}$ which cheeks if the patterns produced appear in the configuration of the first layer (if it is the case, the Turing machine enters in a special state which is forbidden by the subshift of finite type $\T_{\texttt{Final}}$). 
\end{itemize}
Thus for $x\in\T_{\texttt{Final}}\subset\A^{\Z^d}\times\T_{\texttt{Grid}}\times\T_{\TM}$, one has $\pi(x)_{\Z\vect{e_1}}=\pi(x)_{\vect{i}+\Z\vect{e_1}}\in\Sigma$ for all $\vect{i}\in\Vect{\vect{e_2},\dots,\vect{e_d}}$ where $\pi$ is the factor on the first layer which deletes computation states. In particular $\Sigma=\sa{1}{\pi(\T_{\texttt{Final}})}$ but moreover $\Sigma$ is conjugate to a sub-action of $\pi(\T_{\texttt{Final}})$. This result is stronger that only realization by projective subaction and allows to construct local rules for exotic tilings~\cite{Aubrun-Sablik-2012,Fernique-Sablik-2012}.

Now the purpose is to determine $\varphi_{F_{\texttt{Final}},\pi,d\to1}$. By definition, $\Dtime^{\textrm{enu}}_{\TM_F}$ and $\Dspace^{\textrm{enu}}_{\TM_F}$ are respectively the time and the space necessary to enumerate all forbidden patterns of $F$ which forbid all patterns of size $k$ of $\T_F$. Thus to determine $\varphi_{F_{\texttt{Final}},\pi,d\to1}$ we must analyze the width of the row necessary in $\T_{\texttt{Final}}$ to give $\Dtime^{\textrm{enu}}_{\TM_F}(k)$-time steep and $\Dspace^{\textrm{enu}}_{\TM_F}(k)$ space at $\TM_F$ and the larger of the row necessary to detect the forbidden pattern by the additional procedure $\TM_\texttt{Search}$. 

\paragraph{Speed of convergence in the construction of~\cite{Hochman-2009}}\ As it is described in Section 4 of~\cite{Hochman-2009}, $\T_{\texttt{Grid}}$ gives a rectangular partition of $\Z^3$ generated by $\widehat{W}_3\times\widehat{W}_5$ where $\widehat{W}_3$ and $\widehat{W}_5$ are obtained thanks a substitution and generate on each line topelitz sequences. Thus for all $s,t\in\N$ there exists $\M\subset\Z$ such that for all $i\in\M$, the slice $\{i\}\times\Z^2$ is a partition of rectangle of size $3^s\times 5^t$ which delimitates computation zones. Moreover $\M$ has not gap bigger than $3^s5^t$. According to the processes of copy of the initial configuration in the first layer, we need an approximation row of width $O(3^s5^t)$ to detect a forbidden word enumerated in space less than $3^s$ and in time less than $5^t$. One deduces that 
 
\[\left(k\longmapsto\varphi_{F_{\texttt{Final}},\pi,3\to1}\right)\sim\left(k\longmapsto\Dspace^{\textrm{enu}}_{\TM_F}(k)\Dtime^{\textrm{enu}}_{\TM_F}(k)\right).\]

\paragraph{Speed of convergence in the construction of~\cite{Aubrun-Sablik-2010}}\   
As it is described in Section 2 of~\cite{Aubrun-Sablik-2010} (sum up in Fact 2.4), $\T_{\texttt{Grid}}$ defines fractured zone of computation to implement the Turing machine of size $2^n\times 2^{2^n}$, the first coordinate according to $\vect{e_1}$ corresponds to the space and the second according to $\vect{e_2}$ corresponds to the time. By the substitution rules and the clock rules, this ractured zone of computation is included in a pattern of $\T_{\texttt{Grid}}$ of size $4^n\times(2^n2^{2^n})$ and every row $\T_{\texttt{Final}}^{m,2\to1}$ of size $m=2^{n+2^n}$ contains such computation zone every $4^n$ cells. Since the time to cheek if a forbidden pattern of size $k$ appears in the responsibility zone ($n^2 2^n$ steep in direction $\vect{e_2}$ by Fact 3.4 of~\cite{Aubrun-Sablik-2010}) is negligible according to the time given at the Turing machine to compute forbidden patterns ($2^{n+2^n}$ steep in direction $\vect{e_2}$), one deduces that $\left(k\mapsto\varphi_{F_{\texttt{Final}},\pi,2\to1}(k)\right)\sim\left(k\mapsto 2^{n(k)+2^{n(k)}}\right)$ where $n(k)=\min\{n: \Dspace^{\textrm{enu}}_{\TM_F}(k)<2^n\}$. Thus
\[\left(k\longmapsto\varphi_{F_{\texttt{Final}},\pi,2\to1}\right)\sim\left(k\longmapsto\Dspace^{\textrm{enu}}_{\TM_F}(k)2^{\Dspace^{\textrm{enu}}_{\TM_F}(k)}\right).\]

\subsection{A more efficient construction}

In the particular case where $\Sigma$ is an effective subshift that contains a periodic configuration, the previous constructions can be highly simplified, and this simplification also improves the speeds of convergence. In a few words, the same type of construction with different layers is built. However the computation cheeks if no forbidden patterns appear only in one line, the other lines are mapped into the periodic configuration by the factor map. Thus the computation zones do not need to be fractionated, one uses a simplified layer $\T_{\texttt{Grid}}$ which allows a computation in real time.

\begin{theorem}\label{theorem.QuickerConstruction}
Let $\Sigma\subset\A^{\Z^d}$ be an effective subshift of dimension $d$ with a periodic point defined by a recursively enumerable set $F$ of forbidden patterns enumerated by a Turing machine $\TM_F$. 

Then there exists a subshift of finite type $\T$ of dimension $d+1$ and a factor map $\pi$ such that $\sa{d}{\pi(\T)}=\Sigma$ and \[\varphi_{F,\pi,d+1\to d}\sim\Dtime^{\textrm{enu}}_{\TM_F}.\]
\end{theorem}
\begin{proof}
The proof is given for $d=1$ but may be easy generalizes to any dimension. 

Consider an effective subshift $\Sigma\subset\A^{\Z}$ of dimension $1$ with a periodic point $^{\infty}w^{\infty}\in\Sigma$. Suppose that $\Sigma$ is defined by a recursively enumerable set $F$ of forbidden patterns, enumerated by the enumerative Turing machine $\TM=(k,Q,\Gamma,\#,q_0,\delta,Q_F)$ (we suppose that $\TM$ enumerates all forbidden patterns of $F$ on the first tape separated by the symbol $\$$). Finally we assume that the tapes of $\TM$ are onside, that is to say indexed by $\N$, if not it is sufficient to double the alphabet and consider that each cell stocks the positive and negative indexed cell. 

We now construct a subshift of finite type $\T_{\texttt{Final}}\subset\A_{\texttt{Final}}^{\Z^2}$ and a factor map $\pi_{\texttt{Final}}:\A_{\texttt{Final}}\to\A$  such that $\sa{1}{\pi_{\texttt{Final}}(\T_{\texttt{Final}})}=\Sigma$. The subshift $\T_{\texttt{Final}}$ is constituted by four layers:
\begin{enumerate}
\item[\textbf{Layer 1:}] The first layer is $\T_{\texttt{Line}}\subset \left((\A\times\{\inside{\blocku{0}{0}{}},\inside{\blockd{0}{0}{}}\})\cup \{\inside{\blockl{0}{0}}\}\right)^{\Z^2}$ the subshift of finite type such that the forbidden patterns are 
$$F_{\texttt{Line}}=\left\{
\vbox to 1.7cm{\hbox{\begin{tikzpicture}[scale=0.5]
\blockl{0}{1}
\blockl{0}{0}

\blockd{2}{1}{a}
\blockl{2}{0}

\blockl{4}{1}
\blocku{4}{0}{a}

\draw (6,1) node[right] {where $a\in\A$};

\draw [dotted, fill=black!5] (0,-2) rectangle (5,-1);
\draw (2.5,-1.5) node{$\cdots$};
\blocku{0}{-2}{a_1}
\blocku{4}{-2}{a_{|w|}}
\draw (6,-1.5) node[right] {where $a_1\dots a_{|w|}\not\sqsubset ^{\infty}w^{\infty}$};

\draw [dotted, fill=black!25] (0,-4) rectangle (5,-3);
\draw (2.5,-3.5) node{$\cdots$};
\blockd{0}{-4}{a_1}
\blockd{4}{-4}{a_{|w|}}
\draw (6,-3.5) node[right] {where $a_1\dots a_{|w|}\not\sqsubset ^{\infty}w^{\infty}$};
\end{tikzpicture}}}
\right\}$$

Concretely, for a configuration $x\in\T_{\texttt{Line}}$ there is at most one $i\in\Z$ such that $\sa{\vect{e_1}\Z}{\s^{i\vect{e_2}}(x)}=^{\infty}\inside{\blockl{0}{0}}^{\infty}$ and for all $j\ne i$ one has $\sa{\vect{e_1}\Z}{\s^{j\vect{e_2}}(x)}\in\{\s^k(^{\infty}w^{\infty}):k\in\Z\}\times\{\inside{\blocku{0}{0}{}},\inside{\blockd{0}{0}{}}\}^{\Z}$.
\item[\textbf{Layer 2:}] The second layer is the subshift of finite type $\T_{\texttt{Config}}=\{x\in\A^{\Z^2}: \s^{\vect{e_1}-\vect{e_2}}(x)=x\}$.
\item[\textbf{Layer 3:}] The third layer is $\T_{\texttt{Grid}}\subset \left\{\inside{\blockred{0}{0}{}},\inside{\blockred{0}{0}{\ast}}\inside{\blockblue{0}{0}{}},\inside{\blockblue{0}{0}{\ast}}\right\}^{\Z^2}$ the subshift of finite type such that the forbidden patterns are $F_{\texttt{Grid}}=F_{\texttt{Grid}}^1\cup F_{\texttt{Grid}}^2\cup F_{\texttt{Grid}}^3\cup F_{\texttt{Grid}}^4\cup F_{\texttt{Grid}}^5$ where: 
\begin{itemize}
\item $F_{\texttt{Grid}}^1=\left\{\inside{\blockred{0}{0}{\ast}\blockred{1}{0}{}},\inside{\blockred{0}{0}{\ast}\blockblue{1}{0}{}},\inside{\blockblue{0}{0}{\ast}\blockred{1}{0}{}},\inside{\blockblue{0}{0}{\ast}\blockblue{1}{0}{}},\inside{\blockred{0}{0}{}\blockred{1}{0}{\ast}},\inside{\blockred{0}{0}{}\blockblue{1}{0}{\ast}},\inside{\blockblue{0}{0}{}\blockred{1}{0}{\ast}},\inside{\blockblue{0}{0}{}\blockblue{1}{0}{\ast}}\right\}$, this set implies that if the symbol $\ast$ appears in a cell of one line then all cells of this line contain the symbol $\ast$;

\item $F_{\texttt{Grid}}^2=\left\{\insidebis{\blockred{0}{0}{}\blockblue{0}{1}{}},\insidebis{\blockblue{0}{0}{}\blockred{0}{1}{}},\insidebis{\blockblue{0}{0}{}\blockred{0}{1}{\ast}},\insidebis{\blockred{0}{0}{}\blockblue{0}{1}{\ast}}\right\}$, this set implies that in a column the color changes only if it crosses the symbol $\ast$ from bottom to top;

\item $F_{\texttt{Grid}}^3=\left\{\insidebis{\blockred{0}{0}{\ast}\blockred{1}{0}{\ast}\blockred{0}{1}{}\blockblue{1}{1}{}},\insidebis{\blockred{0}{0}{\ast}\blockred{1}{0}{\ast}\blockblue{0}{1}{}\blockred{1}{1}{}},\insidebis{\blockblue{0}{0}{\ast}\blockblue{1}{0}{\ast}\blockred{0}{1}{}\blockblue{1}{1}{}},\insidebis{\blockblue{0}{0}{\ast}\blockblue{1}{0}{\ast}\blockblue{0}{1}{}\blockred{1}{1}{}},\insidebis{\blockred{0}{0}{\ast}\blockred{1}{0}{\ast}\blockred{0}{1}{\ast}\blockblue{1}{1}{\ast}},\insidebis{\blockred{0}{0}{\ast}\blockred{1}{0}{\ast}\blockblue{0}{1}{\ast}\blockred{1}{1}{\ast}},\insidebis{\blockblue{0}{0}{\ast}\blockblue{1}{0}{\ast}\blockred{0}{1}{\ast}\blockblue{1}{1}{\ast}},\insidebis{\blockblue{0}{0}{\ast}\blockblue{1}{0}{\ast}\blockblue{0}{1}{\ast}\blockred{1}{1}{\ast}}\right\}$, this set implies that when it crosses a line with the symbol $\ast$, a monochromatic sequence of symbols remains monochromatic;

 \item $F_{\texttt{Grid}}^4=\left\{\insidebis{\blockred{0}{0}{\ast}\blockblue{1}{0}{\ast}\blockred{0}{1}{}\blockblue{1}{1}{}},\insidebis{\blockred{0}{0}{\ast}\blockblue{1}{0}{\ast}\blockblue{0}{1}{}\blockred{1}{1}{}},\insidebis{\blockred{0}{0}{\ast}\blockblue{1}{0}{\ast}\blockred{0}{1}{\ast}\blockblue{1}{1}{\ast}},\insidebis{\blockred{0}{0}{\ast}\blockblue{1}{0}{\ast}\blockblue{0}{1}{\ast}\blockred{1}{1}{\ast}}\right\}$, this set implies that when it crosses a line with the symbol $\ast$, a transition $\inside{\blockred{0}{0}{}\blockblue{1}{0}{}}$ becomes monochromatic;

 \item $F_{\texttt{Grid}}^5=\left\{\insidebis{\blockblue{0}{0}{\ast}\blockred{1}{0}{\ast}\blockred{0}{1}{}\blockred{1}{1}{}},\insidebis{\blockblue{0}{0}{\ast}\blockred{1}{0}{\ast}\blockblue{0}{1}{}\blockblue{1}{1}{}},\insidebis{\blockblue{0}{0}{\ast}\blockred{1}{0}{\ast}\blockred{0}{1}{\ast}\blockred{1}{1}{\ast}},\insidebis{\blockblue{0}{0}{\ast}\blockred{1}{0}{\ast}\blockblue{0}{1}{\ast}\blockblue{1}{1}{\ast}}\right\}$, this set implies that when it crosses a line with the symbol $\ast$, a transition $\inside{\blockblue{0}{0}{}\blockred{1}{0}{}}$ stays a transition of colors in view to force an alternation of the colors. 
\end{itemize}

Consider a configuration $x\in\T_{\texttt{Grid}}=\ft{F_{\texttt{Grid}}}{\left\{\inside{\blockred{0}{0}{}},\inside{\blockred{0}{0}{\ast}}\inside{\blockblue{0}{0}{}},\inside{\blockblue{0}{0}{\ast}}\right\}^{\Z^2}}$. On each line of $x$, the two colors alternates and this alternation is repeated above until it crosses a line which contains the symbol $\ast$. In this case the transitions red/blue become monochromatic and the transitions blue/red force the alternation. Thus the sequences of monochromatic colors become larger. We remark that if a line contains the periodic configuration $^{\infty}\left( \inside{\blockred{0}{0}{}\blockblue{1}{0}{}}\right)^{\infty}$, then all lines below contain this periodic configuration and upstair, if we have cross $n$ time a line with the symbol $\ast$, we obtain a line with the periodic configuration $^{\infty}\left(\inside{\blockred{0}{0}{}}^{2^n}\inside{\blockblue{1}{0}{}}^{2^n}\right)^{\infty}$ (see Figure~\ref{figure.Grid}). 

\begin{figure}[h!]
\begin{center}
\begin{tikzpicture}[scale=0.4]
\begin{scope}
\draw (-1.2,6.5) node[left] {$x=$};
\clip[draw,decorate,decoration={random steps, segment length=3pt, amplitude=1pt}] (0.2,0.2) rectangle (12.8,12.8);
\foreach \x in {0, ...,13}{\draw[dotted] (\x,0) -- (\x,13);}
\foreach \y in {0, ...,13}{\draw[dotted] (0,\y) -- (13,\y);}

\foreach \x in {0, ...,3}{
  \foreach \y in {0, ..., 2}{
     \blockblue{\x}{\y}{}}}
\foreach \x in {4, ...,5}{
  \foreach \y in {0, ..., 2}{
     \blockred{\x}{\y}{}}}
\foreach \x in {6, ...,7}{
  \foreach \y in {0, ..., 2}{
     \blockblue{\x}{\y}{}}}
\foreach \x in {8, ...,10}{
  \foreach \y in {0, ...,2}{
     \blockred{\x}{\y}{}}}
\foreach \x in {11, ...,12}{
  \foreach \y in {0, ...,2}{
     \blockblue{\x}{\y}{}}}

\foreach \x in {0, ...,3}{\blockblue{\x}{3}{\ast}}
\foreach \x in {4, ...,5}{\blockred{\x}{3}{\ast}}
\foreach \x in {6, ...,7}{\blockblue{\x}{3}{\ast}}
\foreach \x in {8, ...,10}{\blockred{\x}{3}{\ast}}
\foreach \x in {11, ...,12}{\blockblue{\x}{3}{\ast}}

\foreach \x in {0, ...,3}{
  \foreach \y in {4, ...,6}{
     \blockred{\x}{\y}{}}}
\foreach \x in {4, ...,7}{
  \foreach \y in {4, ...,6}{
     \blockblue{\x}{\y}{}}}
\foreach \x in {8, ...,12}{
  \foreach \y in {4, ...,6}{
     \blockred{\x}{\y}{}}}
     
\foreach \x in {0, ...,3}{\blockred{\x}{7}{\ast}}
\foreach \x in {4, ...,7}{\blockblue{\x}{7}{\ast}}
\foreach \x in {8, ...,12}{\blockred{\x}{7}{\ast}}     

\foreach \x in {0, ...,7}{
  \foreach \y in {8, ...,9}{
     \blockred{\x}{\y}{}}}
\foreach \x in {8, ...,12}{
  \foreach \y in {8, ...,9}{
     \blockblue{\x}{\y}{}}}

\foreach \x in {0, ...,7}{\blockred{\x}{10}{\ast}}
\foreach \x in {8, ...,12}{\blockblue{\x}{10}{\ast}}

\foreach \x in {0, ...,12}{
  \foreach \y in {11, ...,12}{
     \blockred{\x}{\y}{}}}
\end{scope}
%%%%%%%%%%%%%%%%%%%%%%%%%%%%%%%%%%%%%%%%%%%%%%%%%%%%%%%%%%%%%
\begin{scope}
\draw (17.8,6.5) node[left] {$y=$};
\clip[draw,decorate,decoration={random steps, segment length=3pt, amplitude=1pt}] (19.2,0.2) rectangle (31.8,12.8);

\foreach \x in {19, ...,32}{\draw[dotted] (\x,-0) -- (\x,13);}
\foreach \y in {0, ...,13}{\draw[dotted] (19,\y) -- (32,\y);}

%\grid{-1}{13}{-1}{13}

\foreach \x in {0, ...,6}{
  \foreach \y in {0, ..., 3}{
     \blockblue{2*\x+19}{\y}{}}}
\foreach \x in {0, ...,5}{
  \foreach \y in {0, ..., 3}{
     \blockred{2*\x+20}{\y}{}}}

\foreach \x in {0, ...,6}{\blockblue{2*\x+19}{4}{\ast}}
\foreach \x in {0, ...,5}{\blockred{2*\x+20}{4}{\ast}}

\foreach \x in {0, ...,3}{
  \foreach \y in {5, ...,7}{
     \blockblue{4*\x+19}{\y}{}}}
\foreach \x in {1, ...,3}{
  \foreach \y in {5, ...,7}{     
     \blockblue{4*\x+18}{\y}{}}}
\foreach \x in {0, ...,2}{
  \foreach \y in {5,...,7}{
     \blockred{4*\x+20}{\y}{}
     \blockred{4*\x+21}{\y}{}}}

\foreach \x in {0, ...,3}{\blockblue{4*\x+19}{8}{\ast}}
\foreach \x in {1, ...,3}{\blockblue{4*\x+18}{8}{\ast}}
\foreach \x in {0, ...,2}{\blockred{4*\x+20}{8}{\ast}
     \blockred{4*\x+21}{8}{\ast}}

\foreach \x in {20,21,22,23,28,29,30,31}{\blockblue{\x}{9}{\ast}}
\foreach \x in {19,24,25,26,27}{\blockred{\x}{9}{\ast}}

\foreach \x in {19, ...,23}{
  \foreach \y in {10,...,12}{
     \blockred{\x}{\y}{}}}
\foreach \x in {24, ...,31}{
  \foreach \y in {10,...,12}{     
     \blockblue{\x}{\y}{}}}
\end{scope}
\end{tikzpicture}
\end{center}

\caption{$x$ and $y$ are two examples of configurations of $\T_{\texttt{Grid}}$. The configuration $y$ is an example where at least one line is periodic, thus all lines are periodic.}\label{figure.Grid}
\end{figure}
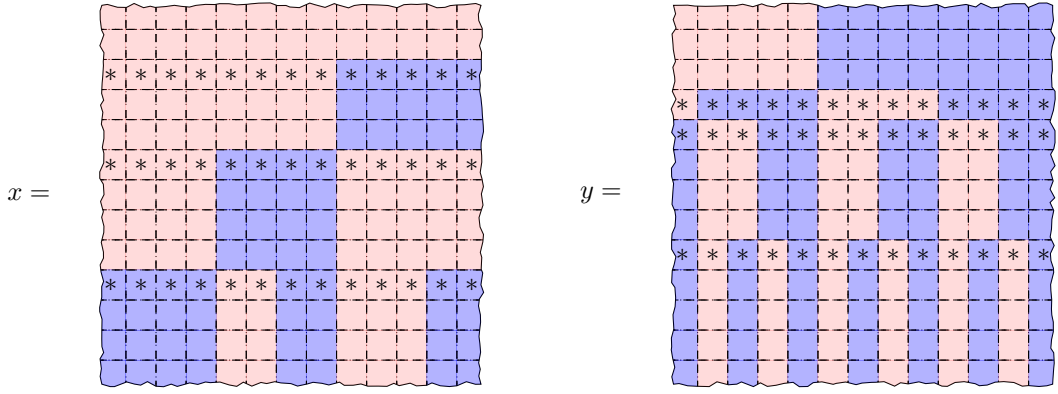

\item[\textbf{Layer 4:}] By definition the alphabet $\Gamma$ of $\TM$ contains $\A\cup\{\sharp,\$\}$. Denote $\A_{\TM}=\left((Q\times\Gamma)\cup\Gamma\right)^k$ where $k$ is the number of tapes and $\pi_{\TM,k}$ the projection on the $k^\textrm{th}$ coordinate of $\A_{\TM}$ which corresponds at the information on the $k^\textrm{th}$ tape. We recall that the forbidden words are enumerated on the first tape. The forth layer is a subshift of finite type $\T_{\TM}\subset \A_{\TM}^{\Z^2}$ given by the following set of forbidden patterns. 

\begin{itemize}
\item The behavior of the machine is given on each tape by $F_{\texttt{Comput}}=F_{\texttt{Comput}}^1\cup F_{\texttt{Comput}}^2\cup F_{\texttt{Comput}}^3\cup F_{\texttt{Comput}}^4$ where: 
\begin{itemize}
\item $F_{\texttt{Comput}}^1=\left\{\insidebis{\blocktm{0}{0}{\alpha}\blocktm{1}{0}{\beta}\blocktm{2}{0}{\gamma}\blocktm{0}{1}{a}\blocktm{1}{1}{b}\blocktm{2}{1}{c}} : \alpha,\beta,\gamma\in\Gamma',a,b,c\in\A_{\TM} \textrm{ such that }\inside{\blocktm{0}{0}{\alpha}\blocktm{1}{0}{\beta}\blocktm{2}{0}{\gamma}}\ne \inside{\blocktm{0}{0}{a}\blocktm{1}{0}{b}\blocktm{2}{0}{c}}\right\}$, this set implies that if the head does not appear, the tape does not change;
\item  the folowing set implies that if the head moves to the left, the transition is respected 
$$F_{\texttt{Comput}}^2=\left\{\begin{array}{l}\insidebis{\blocktm{0}{0}{\alpha}\blocktmhead{1}{0}{q_1}{\beta}\blocktm{2}{0}{\gamma}\blocktm{0}{1}{a}\blocktm{1}{1}{b}\blocktm{2}{1}{c}} : q_1,q_2\in Q, \alpha,\beta,\beta'\in\Gamma',\gamma,a,b,c\in\A_{\TM} \textrm{ such that } \\
\\
\delta(q_1,\beta)=(q_2,\beta',\leftarrow)\textrm{ and }\inside{\blocktm{0}{0}{a}\blocktm{1}{0}{b}\blocktm{2}{0}{c}}\ne \inside{\blocktmhead{0}{0}{q_2}{\alpha}\blocktm{1}{0}{\beta'}\blocktm{2}{0}{\gamma}}
 \end{array}\right\};$$
\item  the folowing set implies that if the head moves to the right, the transition is respected 
$$F_{\texttt{Comput}}^3=\left\{\begin{array}{l}\insidebis{\blocktm{0}{0}{\alpha}\blocktmhead{1}{0}{q_1}{\beta}\blocktm{2}{0}{\gamma}\blocktm{0}{1}{a}\blocktm{1}{1}{b}\blocktm{2}{1}{c}} : q_1,q_2\in Q, \beta,\beta',\gamma\in\Gamma',\alpha,a,b,c\in\A_{\TM} \textrm{ such that } \\
\\
\delta(q_1,\beta)=(q_2,\beta',\rightarrow)\textrm{ and }\inside{\blocktm{0}{0}{a}\blocktm{1}{0}{b}\blocktm{2}{0}{c}}\ne \inside{\blocktm{0}{0}{\alpha}\blocktm{1}{0}{\beta'}\blocktmhead{2}{0}{q_2}{\gamma}}
 \end{array}\right\};$$
 \item  the folowing set implies that if the head does not move, the transition is respected 
$$F_{\texttt{Comput}}^4=\left\{\begin{array}{l}\insidebis{\blocktm{0}{0}{\alpha}\blocktmhead{1}{0}{q_1}{\beta}\blocktm{2}{0}{\gamma}\blocktm{0}{1}{a}\blocktm{1}{1}{b}\blocktm{2}{1}{c}} : q_1,q_2\in Q, \beta,\beta'\in\Gamma',\alpha,\gamma,a,b,c\in\A_{\TM} \textrm{ such that } \\
\\
\delta(q_1,\beta)=(q_2,\beta',\cdot)\textrm{ and }\inside{\blocktm{0}{0}{a}\blocktm{1}{0}{b}\blocktm{2}{0}{c}}\ne \inside{\blocktm{0}{0}{\alpha}\blocktmhead{1}{0}{q_2}{\beta'}\blocktm{2}{0}{\gamma}}
 \end{array}\right\}.$$
\end{itemize}
Thus condition $F_{\texttt{Comput}}$ codes the rules to obtain the space-time diagram of the Turing Machine $\TM$. 

\item Sometime, we need to erase the two tapes and put the blank symbol, we use the following forbidden patterns:
$$F_{\texttt{Erase}}=\left\{\insidebis{\blocktm{0}{0}{\alpha}\blocktm{1}{0}{\beta}\blocktm{2}{0}{\gamma}\blocktm{0}{1}{\alpha'}\blocktm{1}{1}{\beta'}\blocktm{2}{1}{\gamma'}}:\alpha,\beta,\gamma,\alpha',\gamma'\in\A_{\TM}\textrm{ and }\beta'\in\A_{\TM}\setminus\{(\sharp)^k\}\right\}.$$
\end{itemize}

Each tape must verify the previous rules, so we consider the forbidden patterns $F_{\TM}=F_\texttt{Erase}\cap\bigcup_{i}\pi^{-1}_{\TM,i}(F_{\texttt{Comput}})$ and one defines
$$\T_{\TM}=\ft{F_{\TM}}{\A_{\TM}^{\Z^2}},$$
Thus the evolution of a cell following $\vect{e_2}$ can be given by the transition of the Turing machine $\TM$ or be erased, this will be precise by the condition $F_{\texttt{Extend}}$.

\item[\textbf{Layer 5:}] The fifth layer is the full-shift $\T_{\texttt{Compar}}=\{\inside{\block{0}{0}{}}\inside{\blockgreen{0}{0}{}}\}^{\Z^2}$.
\end{enumerate}

On the subshift of finite type $\T_{\texttt{Line}}\times\T_{\texttt{Config}}\times\T_{\texttt{Grid}}\times\T_{\TM}\times\T_{\texttt{Compar}}$ we define the factor maps $\pi_{\texttt{Line}}$, $\pi_{\texttt{Config}}$, $\pi_{\texttt{Grid}}$, $\pi_{\TM}$ and $\pi_{\texttt{Compar}}$ which correspond to the projection of each coordinate.

To obtain the subshift of finite type $\T_{\texttt{Final}}\subset\T_{\texttt{Line}}\times\T_{\texttt{Config}}\times\T_{\texttt{Grid}}\times\T_{\TM}\times\T_{\texttt{Compar}}$ we add to the subshift of finite type $\T_{\texttt{Line}}\times\T_{\texttt{Config}}\times\T_{\texttt{Grid}}\times\T_{\TM}\times\T_{\texttt{Compar}}$ the finite set of forbidden patterns $F_{\texttt{SyncroLine}}\cup F_{\texttt{Init}}\cup  F_{\texttt{Extend}}\cup F_{\texttt{Compar}}$ which codes the interaction between the different layers, that is to say $$\T_{\texttt{Final}}=\ft{F_{\texttt{SyncroLine}}\cup F_{\texttt{Init}}\cup  F_{\texttt{Extend}}\cup F_{\texttt{Compar}}}{\T_{\texttt{Line}}\times\T_{\texttt{Config}}\times\T_{\texttt{Grid}}\times\T_{\TM}\times\T_{\texttt{Compar}}}.$$ These local rules are:
\begin{enumerate}
\item[\textbf{Rules $F_{\texttt{SyncroLine}}$:}] The forbidden patterns $F_{\texttt{SyncroLine}}$ imply that for $x\in\T_{\texttt{Final}}$ if $\pi_{\texttt{Line}}(x)_{(i_1,i_2)}=\inside{\blockl{0}{0}}$ then $\pi_{\texttt{Grid}}(x)_{([i_1,i_1+1],i_2)}\notin \left\{\inside{\blockred{0}{0}{}\blockred{1}{0}{}},\inside{\blockred{0}{0}{\ast}\blockred{1}{0}{\ast}},\inside{\blockblue{0}{0}{}\blockblue{1}{0}{}},\inside{\blockblue{0}{0}{\ast}\blockblue{1}{0}{\ast}}\right\}$ for $(i_1,i_2)\in\Z^2$. Thanks to these rules, if a line $^{\infty}\inside{\blockl{0}{0}}^{\infty}$ appears in the layer $\T_{\texttt{Line}}$ of a configuration, then it is synchronized with a periodic point $^{\infty}\inside{\blockblue{0}{0}{}\blockred{1}{0}{}}^{\infty}$ or  $^{\infty}\inside{\blockblue{0}{0}{\ast}\blockred{1}{0}{\ast}}^{\infty}$ in the layer $\T_{\texttt{Grid}}$. 

\item[\textbf{Rules $F_{\texttt{Init}}$:}] The forbidden patterns $F_{\texttt{Init}}$ imply that for $x\in\T_{\texttt{Final}}$ if $\pi_{\texttt{Line}}(x)_{(i_1,i_2)}=\inside{\blockl{0}{0}}$ then $\pi_{\TM,i}(x)_{(i_1,i_2)}=\inside{\blocktmhead{0}{0}{q_0}{\sharp}}$ on each tape $i\in[1,k]$. Thank to these rules, if a line $^{\infty}\inside{\blockl{0}{0}}^{\infty}$ appears in the layer $\T_{\texttt{Line}}$ of a configuration, then this line is initialized in $\T_{\TM}$.

\item[\textbf{Rules $F_{\texttt{Extend}}$:}] The forbidden patterns $F_{\texttt{Extend}}$ imply that for $x\in\T_{\texttt{Final}}$ one has the two following things: 
\begin{itemize}
\item one has $\delta(\pi_{\TM}(x)_{(i_1,i_2)})=(q',\alpha',\rightarrow)$, $\pi_{\texttt{Grid}}(x)_{([i_1,i_1+1],i_2)}\in\{\inside{\blockred{0}{0}{}\blockblue{1}{0}{}},\inside{\blockblue{0}{0}{}\blockred{1}{0}{}},\inside{\blockred{0}{0}{\ast}\blockblue{1}{0}{\ast}},\inside{\blockblue{0}{0}{\ast}\blockred{1}{0}{\ast}}\}$ if and only if $\pi_{\texttt{Grid}}(x)_{(i_1,i_2)}=\inside{\blockred{0}{0}{\ast}}$ or $\inside{\blockblue{0}{0}{\ast}}$;
\item if $\pi_{\texttt{Grid}}(x)_{(i_1,i_2)}=\inside{\blockred{0}{0}{\ast}}$ then $\pi_{\TM}(x)_{(i_1,i_2+1)}=\inside{\blocktm{0}{0}{\sharp}}$;
\item if $\pi_{\texttt{Grid}}(x)_{(i_1,i_2)}\ne\inside{\blockred{0}{0}{\ast}}$ then $\pi_{\TM,i}(x)_{([i_1-1,i_1+1],[i_2,i_2+1])}\notin F_{\texttt{Comput}}$ (that is to say $F_{Erase}$ is not take in consideration) for all tape $i\in[1,k]$.
\end{itemize}
Thanks to these rules, if a computation needs more space, the space allowed by a Turing machine is doubled if the head is in a blue zone and disappear if it is in a red zone, and this action is done only if the Turing machine need more place. Moreover, the tape can be erased if and only if the corresponding cells in $\T_{\texttt{Grid}}$ are in the state $\inside{\blockred{0}{0}{\ast} }$, in other place the computation holds following the transition rule of the Turing machine given by $F_{\texttt{Comput}}$.  

\item[\textbf{Rules $F_{\texttt{Compar}}$:}] The forbidden patterns $F_{\texttt{Compar}}$ imply that for $x\in\T_{\texttt{Final}}$ one has the  following things: 
\begin{itemize}
\item if $\pi_{\TM,1}(x)_{(i_1,i_2)}=\inside{\blocktm{0}{0}{\$}}$ then $\pi_{\texttt{Compar}}(x)_{(i_1,i_2)}=\inside{\blockgreen{0}{0}{}}$;
\item if $\pi_{\TM,1}(x)_{(i_1,i_2)}=\inside{\blocktm{0}{0}{a}}=\pi_{\texttt{Config}}(x)_{(i_1,i_2)}$ with $a\in\A$ and $\pi_{\texttt{Compar}}(x)_{(i_1-1,i_2)}=\inside{\blockgreen{0}{0}{}}$ then $\pi_{\texttt{Compar}}(x)_{(i_1,i_2)}=\inside{\blockgreen{0}{0}{}}$;
\item if $\pi_{\texttt{Compar}}(x)_{(i_1-1,i_2)}=\inside{\blockgreen{0}{0}{}}$ then $\pi_{\TM,1}(x)_{(i_1,i_2)}\ne\inside{\blocktm{0}{0}{\$}}$.
\end{itemize}
Thanks to these rules, a forbidden pattern which appears in the enumeration obtained in $\T_{\TM}$ is compared with the corresponding pattern which appears in $\T_{\texttt{Config}}$. If the two patterns coincide then the configuration is forbidden in $\T_{\texttt{Final}}$. 
\end{enumerate}

We define the factor map $\pi_{\texttt{Final}}:\T_{\texttt{Final}}\to\A^{\Z^2}$ by
$$\textrm{for all }x\in\T_{\texttt{Final}}\textrm{ and }\vect{i}\in\Z^2\textrm{ then }\pi(x)_{\vect{i}}= \begin{cases} \pi_{\texttt{Conf}}(x)_{\vect{i}} &\textrm{ if } \pi_{\texttt{Line}}(x)_{\vect{i}}=\inside{\blockl{0}{0}} \\ \textrm{the first coordinate of } \pi_{\texttt{Line}}(x)_{\vect{i}} &\textrm{ if not}\end{cases}.$$

\Claim{$\sa{\vect{e_1}\Z}{\pi_{\texttt{Final}}(\T_{\texttt{Final})}}\supset\Sigma$.}
\bclaimprf
Let $x\in\Sigma$ we construct a configuration $y\in\T_{\texttt{Line}}\times\T_{\texttt{Config}}\times\T_{\texttt{Grid}}\times\T_{\TM}\times\T_{\texttt{Compar}}$ such that 
\begin{eqnarray*}
\sa{\vect{e_1}\Z}{\pi_{\texttt{Line}}(y)}=^{\infty}\inside{\blockl{0}{0}}^{\infty}&,\qquad &\sa{\vect{e_1}\Z}{\pi_{\texttt{Config}}(y)}=x,\\ 
\sa{\vect{e_1}\Z}{\pi_{\texttt{Grid}}(y)}=^{\infty}\inside{\blockblue{0}{0}{}\blockred{1}{0}{}}^{\infty}&\textrm{and}&\sa{\vect{e_1}\Z}{\pi_{\TM}(y)}=^{\infty}\inside{\blocktmhead{0}{0}{q_0}{\sharp}}^{\infty}.
\end{eqnarray*}
Since $x\in\Sigma$, $\TM$ does not enumerate sub-patterns of $x$. Thus it is possible to complete the configuration $y$ to obtain a configuration of $\T_{\texttt{Final}}$ such that $\sa{\vect{e_1}\Z}{\pi_{\texttt{Final}}(y)}=x$.
\eclaimprf

\Claim{$\sa{\vect{e_1}\Z}{\pi_{\texttt{Final}}(\T_{\texttt{Final})}}\subset\Sigma$.} 
\bclaimprf
Let $y\in\T_{\texttt{Final}}$. 

If $\pi_{\texttt{Line}}(y)_{(0,0)}\ne\inside{\blockl{0}{0}}$ then $\sa{\vect{e_1}\Z}{\pi_{\texttt{Final}}(y)}=^{\infty}w^{\infty}\in\Sigma$. 

If $\pi_{\texttt{Line}}(y)_{(0,0)}=\inside{\blockl{0}{0}}$ we consider $u$ a sub-pattern of $x=\sa{\vect{e_1}\Z}{\pi_{\texttt{Final}}(y)}$ of size $n$. Assume that $u\notin\Lang(\Sigma)$, then there exists a word $w\sqsubset x$ such that $u \sqsubset w$ or $w\sqsubset u$ which is enumerated by $\TM$ in time $t_F(n)=\Dtime^{\textrm{enu}}_{\TM}(n)$ and space $s_F(n)=\Dtime^{\textrm{enu}}_{\TM}(n)$. 

By construction of $\T_{\texttt{Final}}$, one has $\sa{\vect{e_1}\Z}{\pi_{\texttt{Grid}}(\s^{t_F(n)}(y))}=^{\infty}\left(\inside{\blockred{0}{0}{}}^{2^k}\inside{\blockblue{1}{0}{}}^{2^k}\right)^{\infty}$ where $k=\min\{k':s_F(n)<2^{k'}\}$, moreover there exists $i\in\Z$ such that $\sa{\vect{e_1}\Z}{\pi_{\TM}(y)}_{{([i+2^kj,i+|w|+1+2^kj],t_F(n))}}=\$w\$$ for all $j\in\Z$. Since $\TM$ waits $2^{k+1}$ steeps of time before to use more space to compute the next one and $\sa{\vect{e_1}\Z}{\s^{\vect{e_2}}(\pi_{\texttt{Config}}(z))}=\s\left(\sa{\vect{e_1}\Z}{\pi_{\texttt{Config}}(z)}\right)$, we conclude there exists $k'$ such that $t_F(n)\leq k'\leq t_F(n)+2^{1+min\{k: s_F(n)<2^k\}}$ and $j\in\Z$ such that
$$\sa{\vect{e_1}\Z}{\pi_{\TM}(y)}_{{([i+2^kj,i+|w|-1+2^kj],t_F(n))}}=w=\sa{\vect{e_1}\Z}{\pi_{\texttt{Config}}(y)}_{{([i+2^kj,i+|w|-1+2^kj],t_F(n))}}.$$
By the condition $F_{\texttt{Compar}}$ this is impossible. Thus $\sa{\vect{e_1}\Z}{\pi_{\texttt{Final}}(y)}\in\Sigma$.
\eclaimprf

Thus $\sa{\vect{e_1}\Z}{\pi_{\texttt{Final}}(\T_{\texttt{Final})}}=\Sigma$, moreover $\varphi_{F,\pi,2\to 1}(k)= t_F(k)+2^{1+\min\{n: s_F(k)<2^n\}}$ for all $k\in\N$. In particular $\varphi_{F,\pi,d+1\to d}\sim\max(\Dtime^{\textrm{enu}}_{\TM_F},\Dspace^{\textrm{enu}}_{\TM_F})=\Dtime^{\textrm{enu}}_{\TM_F}$.
\end{proof}

\subsection{Increase the dimension to increase the speed}

In general, the properties studied on subshifts of finite type exhibit a gap between dimension one and dimension two, and one assimilates that behaviors for the other dimension identical to those of dimension 2. In fact it appears differences in an algorithmic point of view, in particular we have the following result which shows that if $\varphi\in \cF^{\sofic}_{\Sigma,d+d'\to d}$ then $\varphi^{\frac{d'}{d''}}\in\cF^{\sofic}_{\Sigma,d+d''\to d}$. The idea is to wrap configurations of the $d+d'$-dimensional subshift on the higher dimensions following a Hilbert curve obtained thanks to a substitution (see Figure~\ref{figure: WrappedPattern}).

\begin{theorem}
Let $\Sigma\subset\A^{\Z^{d}}$ be a subshift which is $\varphi,d+d'$-realizable by projective subaction of a sofic then it is $\varphi^{\frac{d'}{d''}},d+d''$-realizable by projective subaction of a sofic for $d''\geq d'$.
\end{theorem}
\begin{proof}
Let $\Sigma\subset\A^{\Z}$ be an effective subshift. Consider $\T_F\subset\B^{\Z^{2}}$ a subshift of finite type with forbidden patterns $F$ and $\pi:\B\to\A$ be a factor map such that $\pi(\sa{\vect{e_1}\Z}{\T_F})=\Sigma$. Denote  $\varphi=\varphi_{F,\pi,2\to 1}$. One constructs $\T_{F'}\subset\B'^{\Z^{3}}$ a subshift of finite type with forbidden patterns $F'$ and $\pi':\B'\to\A$ a factor map such that $\sa{\vect{e_1}\Z}{\pi'(\T_{F'})}=\Sigma$ and $\varphi_{F',\pi',3\to 1}\in\Theta(\sqrt{\varphi})$. This prove the Theorem for $d=1$, $d'=1$ and $d''=2$.

\paragraph{Construction of a tangled grid}\  Consider the alphabet $\mathcal{C}$ formed by the three following tiles, their rotations and their symmetrized about to the axis, thus $\card(\mathcal{C})=3\times4\times2=24$.

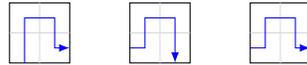
\begin{center}
\begin{tikzpicture}[scale=0.8]
 \hilbertb{0}{0}{0}{1}
 \hilbertl{2}{0}{0}{1}
 \hilbertlr{4}{0}{0}{1}
\end{tikzpicture}
\end{center}

Define the following substitution on the alphabet $\mathcal{C}$ (modulo rotations and symmetries):

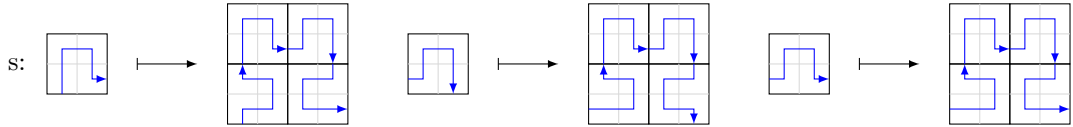
\begin{center}
\begin{tikzpicture}[scale=0.8]
\draw (-1,0) node {s:};
\begin{scope}[xshift=0cm]
\hilbertb{0}{0}{0}{1} 
\draw[->,>=latex] (1,0)--(2,0); 
\draw (1,0.1)--(1,-0.1);
\hilbertbtwo{3.5}{0}{0}{1}
\end{scope}

\begin{scope}[xshift=6cm]
\hilbertl{0}{0}{0}{1} 
\draw[->,>=latex] (1,0)--(2,0); 
\draw (1,0.1)--(1,-0.1);
\hilbertltwo{3.5}{0}{0}{1}
\end{scope}

\begin{scope}[xshift=12cm]
\hilbertlr{0}{0}{0}{1} 
\draw[->,>=latex] (1,0)--(2,0); 
\draw (1,0.1)--(1,-0.1);
\hilbertlrtwo{3.5}{0}{0}{1}
\end{scope}
\end{tikzpicture}
\end{center}

By iterating substitution $s$ on a letter $a\in\mathcal{C}$ (see Figure~\ref{figure:iterationS}), we construct for every $n\in\N$ the pattern $s^n(a)$ called the \define{super-tile of order $n$ and type $a$}. The substitutive subshift defined by \[\T_s=\left\{x\in\mathcal{C}^{\Z^2}: u\sqsubset x \textrm{ if there exists $n\in\N$ and $a\in\mathcal{C}$ which verifies }u\sqsubset s^n(a)\right\},\] is sofic according to Mozes' result~\cite{Mozes-1989}. That is to say there exist a subshift of finite type $\T_{F_s}\subset\mathcal{C}_s^{\Z^2}$ and a factor map $\pi_s:\mathcal{C}_s\to\mathcal{C}$ such that $\pi_s(\T_{F_s})=\T_{s}$. In the Mozes' construction, it is possible to consider a substitution $\overline{s}$ on $\mathcal{C}_s$ which factorizes on $s$ and verifies $\T_{F_s}=\T_{\overline{s}}$. Moreover in the Mozes' proof, forbidden patterns $F_s$ force every tile to be in a super-tile of order $1$, and then every super tile of order $n$ to be in a super tile of order $n+1$. Moreover, one just uses the fact that a pattern localy verifies the condition $F_s$ but not that the pattern is in a globally admissible configuration. Thus if $p\in\mathcal{C}_s^{[-k,k]^2}$ is a pattern which does not contain patterns of $F_s$, then the center letter $p_{\vect{0}}$ is in a super-tile of order $n$ such that $2^n\leq k<2^{n+1}$.

\figex{\begin{figure}[h!]
\begin{center}
\begin{tikzpicture}[scale=0.8]
\hilbertb{0}{0}{0}{1} 
\draw[->,>=latex] (1,0)--(2,0); 
\hilbertbtwo{3.5}{0}{0}{1}
\draw[->,>=latex] (5,0)--(6,0); 
\hilbertbthree{8.5}{0}{0}{1}
\draw[->,>=latex] (11,0)--(12,0); 

\begin{scope}[xshift=16.5cm]
\hilbertlrthree{-2}{-2}{-90}{-1}
\hilbertbthree{-2}{2}{0}{1}
\hilbertlthree{2}{2}{0}{1}
\hilbertlthree{2}{-2}{90}{-1}
\end{scope}
\end{tikzpicture}
\end{center}
\caption{Three iterations of $s$ on an element of $\mathcal{C}$.}\label{figure:iterationS}
\end{figure}
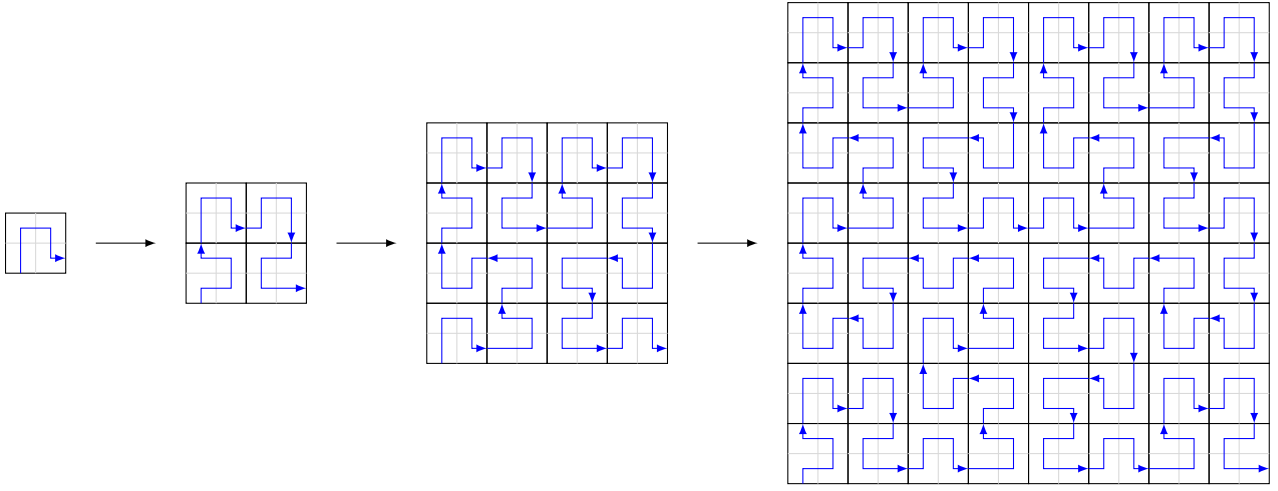}

\paragraph{Construction of a three-dimensional sofic subshift which realizes $\Sigma$.}\  For any configuration $x\in\T_s$, define the function $\textrm{Coor}_{x}:\Z\longrightarrow\Z^2$ such that for $i\in\Z$, $\textrm{Coor}_{x}(i)$ corresponds to the position $i$ starting from the cell $x_{\vect{0}}$ and following the arrows of the alphabet $\mathcal{C}$. 

Consider $\B'=\B\times\mathcal{C}_s$ and the factor maps $\pi_1:\B'\to\B$ and $\pi_2:\B'\to\mathcal{C}_s$ corresponding to the projection on each coordinates. Define the subshift of finite type $\T_{F'}\subset\B'^{\Z^3}$ such that the set of forbidden patterns $F'=F'_1\cup F'_2\cup F'_3$ contains the following patterns:
\begin{itemize}
\item $F'_1=\left\{\pi_2^{-1}(p) : p\in F_s\right\}$, these conditions ensure to have the tangled grid described above on the second layer on each slice $\{i\}\times\Z^2$ with $i\in\Z$;
\item $F'_2=\left\{p\in\B'^{\{(0,0,0),(1,0,0)\}} : \pi_2(p_{(0,0,0)})\ne \pi_2(p_{(1,0,0)})\right\}$, these conditions force each tangled grid to be similar by translation according to $\vect{e_1}$;
\item $F'_3=\{\ p'\in(\B\times\mathcal{C}_s)^{\ast} : \textrm{ there exist } p\in F \textrm{ and } x\in\T_s  \textrm{ such that } p_{(i,j)}=\pi_1(p')_{i,\textrm{Coor}_{x}(j)} \textrm{ and }x_{(i,j)}=\pi_s\circ\pi_2(p')_{i,\textrm{Coor}_{x}(j)} \textrm{ for all }(i,j)\in\supp(p)\}$, these conditions transfer the $2$-dimensional forbidden patterns of $F$ in $3$-dimensional forbidden patterns where the second coordinate is wrapped following the  tangled grid (see Figure~\ref{figure:WrappedPattern}).
\end{itemize}

By construction, if $z\in\T_{F'}$, there exists $y\in\T_F$ and $x\in\T_s$ such that for all $(i,j)\in\Z^2$, $z_{i,\textrm{Coor}_{x}(j)}=y_{(i,j)}$ and $\pi_s\circ \pi_2(\sa{\vect{e_2}\Z+\vect{e_3}\Z}{z})=x$. Reciprocally, for every $y\in\T_F$, there exists $z\in\T_{F'}$ and $x\in\T_s$ such that for all $(i,j)\in\Z^2$, $z_{i,\textrm{Coor}_{x}(j)}=y_{(i,j)}$ and $\pi_s\circ \pi_2(\sa{\vect{e_2}\Z+\vect{e_3}\Z}{z})=x$. Thus $\pi\circ \pi_1(\sa{\vect{e_1}\Z}{\T_{F'}})=\Sigma$.

\figex{
\begin{figure}[h!]
\begin{center}
\begin{tikzpicture}
\begin{scope}[xshift=-7cm,yshift=-4cm,scale=0.6]

\draw[->,>=latex] (0,5)--(1,5);
\draw[->,>=latex] (0,5)--(0,6);
\draw (0,6) node[above] {$\vect{e_1}$};
\draw  (1,5) node[right] {$\vect{e_2}$};

\foreach \x in {0,2,5}{
\fill[pattern=north west lines,opacity=0.7] (\x,0) rectangle (\x+1,3);}

\foreach \x in {1,3,4,6}{
\fill[gray!50,opacity=0.7] (\x,0) rectangle (\x+1,3);}

\foreach \x in {0,1,...,7}{
\foreach \y in {0,1,...,3}{
\draw(\x,\y) node{$\bullet$};
}}
\foreach \x in {0,1,...,7}{
\draw[dashed] (\x,-1)--(\x,4);}
\foreach \y in {0,1,...,3}{
\draw[dotted, very thick] (0,\y)--(7,\y);}
\foreach \x in {0,1,...,7}{
\draw[dashed] (\x,-1)--(\x,4);}

\draw(9,1.5) node {$\longmapsto$};
\draw (0,-1) node[below] {$\Sigma$};
\end{scope}
\begin{scope}[math3d,scale=0.7]
\begin{scope}[xshift=12cm]
\draw[->,>=latex] (0,0,0)--(0,0,1);
\draw (0,0,1) node[above] {$\vect{e_1}$};
\draw[->,>=latex] (0,0,0)--(1,0,0);
\draw (1,0,0) node[right] {$\vect{e_2}$};
\draw[->,>=latex] (0,0,0)--(0,1,0);
\draw (0,1,0) node[right] {$\vect{e_3}$};
\end{scope}
\foreach \i in{-6}{
\begin{scope}[yshift=\i cm]
\foreach \x in {0,1,...,8}{
\draw (\x,0,0)--(\x,8,0);
\draw (0,\x,0)--(8,\x,0);
}
\foreach \x in {0.5,1.5,...,7.5}{
\draw[gray!30] (\x,0,0)--(\x,8,0);
\draw[gray!30] (0,\x,0)--(8,\x,0);}

\foreach \x\y in {0/0,0/3,0/6,4/6}{
\draw[blue,->,>=latex](\x+0.25,\y,0)--(\x+0.25,\y+0.25,0)--(\x+0.75,\y+0.25,0)--(\x+0.75,\y+0.75,0)--(\x+0.25,\y+0.75,0)--(\x+0.25,\y+1,0);
}
\foreach \x\y in {0/1,2/2,0/4,0/7,2/7,4/7,6/7,6/1}{
\draw[blue,->,>=latex](\x+0.25,\y,0)--(\x+0.25,\y+0.75,0)--(\x+0.75,\y+0.75,0)--(\x+0.75,\y+0.25,0)--(\x+1,\y+0.25,0);
}
\foreach \x\y in {1/1,1/7,3/7,5/7,7/7,5/2,7/4,7/1}{
\draw[blue,->,>=latex](\x,\y+0.25,0)--(\x+0.25,\y+0.25,0)--(\x+0.25,\y+0.75,0)--(\x+0.75,\y+0.75,0)--(\x+0.75,\y,0);
}
\foreach \x\y in {1/0,4/0,4/2,7/0,2/4,6/4,1/6,5/6}{
\draw[blue,->,>=latex](\x+0.75,\y+1,0)--(\x+0.75,\y+0.75,0)--(\x+0.25,\y+0.75,0)--(\x+0.25,\y+0.25,0)--(\x+1,\y+0.25,0);
}
\foreach \x\y in {2/0,3/4,4/4,5/0}{
\draw[blue,->,>=latex](\x,\y+0.25,0)--(\x+0.25,\y+0.25,0)--(\x+0.25,\y+0.75,0)--(\x+0.75,\y+0.75,0)--(\x+0.75,\y+0.25,0)--(\x+1,\y+0.25,0);
}
\foreach \x\y in {3/0,1/4,2/6,5/4,3/2,6/6,6/0}{
\draw[blue,->,>=latex](\x,\y+0.25,0)--(\x+0.75,\y+0.25,0)--(\x+0.75,\y+0.75,0)--(\x+0.25,\y+0.75,0)--(\x+0.25,\y+1,0);
}
\foreach \x\y in {3/1,1/5,3/3}{
\draw[blue,->,>=latex](\x+0.25,\y,0)--(\x+0.25,\y+0.25,0)--(\x+0.75,\y+0.25,0)--(\x+0.75,\y+0.75,0)--(\x,\y+0.75,0);
}
\foreach \x\y in {2/1,0/2,0/5,4/5,6/2}{
\draw[blue,->,>=latex](\x+1,\y+0.75,0)--(\x+0.75,\y+0.75,0)--(\x+0.75,\y+0.25,0)--(\x+0.25,\y+0.25,0)--(\x+0.25,\y+1,0);
}
\foreach \x\y in {2/3,5/3}{
\draw[blue,->,>=latex](\x+1,\y+0.75,0)--(\x+0.75,\y+0.75,0)--(\x+0.75,\y+0.25,0)--(\x+0.25,\y+0.25,0)--(\x+0.25,\y+0.75,0)--(\x,\y+0.75,0);
}
\foreach \x\y in {1/3,2/5,6/5,4/3,4/1}{
\draw[blue,->,>=latex](\x+1,\y+0.75,0)--(\x+0.25,\y+0.75,0)--(\x+0.25,\y+0.25,0)--(\x+0.75,\y+0.25,0)--(\x+0.75,\y,0);
}
\foreach \x\y in {1/2,3/5,7/5,7/2,5/1}{
\draw[blue,->,>=latex](\x+0.75,\y+1,0)--(\x+0.75,\y+0.25,0)--(\x+0.25,\y+0.25,0)--(\x+0.25,\y+0.75,0)--(\x,\y+0.75,0);
}
\foreach \x\y in {3/6,7/6,7/3}{
\draw[blue,->,>=latex](\x+0.75,\y+1,0)--(\x+0.75,\y+0.75,0)--(\x+0.25,\y+0.75,0)--(\x+0.25,\y+0.25,0)--(\x+0.75,\y+0.25,0)--(\x+0.75,\y,0);
}
\foreach \x\y in {5/5,6/3}{
\draw[blue,->,>=latex](\x+0.25,\y,0)--(\x+0.25,\y+0.25,0)--(\x+0.75,\y+0.25,0)--(\x+0.75,\y+0.75,0)--(\x,\y+0.75,0);
}

\foreach \x\y in {0/0,0/3,0/6,4/6}{
\draw[blue,->,>=latex](\x+0.25,\y,0)--(\x+0.25,\y+0.25,0)--(\x+0.75,\y+0.25,0)--(\x+0.75,\y+0.75,0)--(\x+0.25,\y+0.75,0)--(\x+0.25,\y+1,0);
}
\end{scope}}
\fill[gray,opacity=0.4] (3.5,1.5,0)--(2.5,1.5,0)--(2.5,1.5,-6)--(3.5,1.5,-6)--cycle;

\fill[gray,pattern=north west lines,opacity=0.4] (0.5,0.5,0)--(0.5,1.5,0)--(0.5,1.5,-6)--(0.5,0.5,-6)--cycle;

\fill[gray,opacity=0.4](0.5,1.5,0)--(1.5,1.5,0)--(1.5,1.5,-6)--(0.5,1.5,-6)--cycle;

\fill[gray,pattern=north east lines,opacity=0.4](1.5,1.5,0)--(1.5,0.5,0)--(1.5,0.5,-6)--(1.5,1.5,-6)--cycle;

\fill[gray,opacity=0.4](1.5,0.5,0)--(2.5,0.5,0)--(2.5,0.5,-6)--(1.5,0.5,-6)--cycle;

\fill[gray,opacity=0.4] (2.5,0.5,0)--(3.5,0.5,0)--(3.5,0.5,-6)--(2.5,0.5,-6)--cycle;

\fill[gray,pattern=north west lines,opacity=0.4] (3.5,0.5,0)--(3.5,1.5,0)--(3.5,1.5,-6)--(3.5,0.5,-6)--cycle;

\foreach \a in {0.5,1.5,...,3.5}{
\foreach \b in {0.5,1.5}{
\draw[dashed] (\a,\b,2)--(\a,\b,-8);}}

\foreach \i in{0}{
\begin{scope}[yshift=\i cm]
\foreach \x in {0,1,...,8}{
\draw (\x,0,0)--(\x,8,0);
\draw (0,\x,0)--(8,\x,0);
}
\foreach \x in {0.5,1.5,...,7.5}{
\draw[gray!30] (\x,0,0)--(\x,8,0);
\draw[gray!30] (0,\x,0)--(8,\x,0);}

\foreach \x\y in {0/0,0/3,0/6,4/6}{
\draw[blue,->,>=latex](\x+0.25,\y,0)--(\x+0.25,\y+0.25,0)--(\x+0.75,\y+0.25,0)--(\x+0.75,\y+0.75,0)--(\x+0.25,\y+0.75,0)--(\x+0.25,\y+1,0);
}
\foreach \x\y in {0/1,2/2,0/4,0/7,2/7,4/7,6/7,6/1}{
\draw[blue,->,>=latex](\x+0.25,\y,0)--(\x+0.25,\y+0.75,0)--(\x+0.75,\y+0.75,0)--(\x+0.75,\y+0.25,0)--(\x+1,\y+0.25,0);
}
\foreach \x\y in {1/1,1/7,3/7,5/7,7/7,5/2,7/4,7/1}{
\draw[blue,->,>=latex](\x,\y+0.25,0)--(\x+0.25,\y+0.25,0)--(\x+0.25,\y+0.75,0)--(\x+0.75,\y+0.75,0)--(\x+0.75,\y,0);
}
\foreach \x\y in {1/0,4/0,4/2,7/0,2/4,6/4,1/6,5/6}{
\draw[blue,->,>=latex](\x+0.75,\y+1,0)--(\x+0.75,\y+0.75,0)--(\x+0.25,\y+0.75,0)--(\x+0.25,\y+0.25,0)--(\x+1,\y+0.25,0);
}
\foreach \x\y in {2/0,3/4,4/4,5/0}{
\draw[blue,->,>=latex](\x,\y+0.25,0)--(\x+0.25,\y+0.25,0)--(\x+0.25,\y+0.75,0)--(\x+0.75,\y+0.75,0)--(\x+0.75,\y+0.25,0)--(\x+1,\y+0.25,0);
}
\foreach \x\y in {3/0,1/4,2/6,5/4,3/2,6/6,6/0}{
\draw[blue,->,>=latex](\x,\y+0.25,0)--(\x+0.75,\y+0.25,0)--(\x+0.75,\y+0.75,0)--(\x+0.25,\y+0.75,0)--(\x+0.25,\y+1,0);
}
\foreach \x\y in {3/1,1/5,3/3}{
\draw[blue,->,>=latex](\x+0.25,\y,0)--(\x+0.25,\y+0.25,0)--(\x+0.75,\y+0.25,0)--(\x+0.75,\y+0.75,0)--(\x,\y+0.75,0);
}
\foreach \x\y in {2/1,0/2,0/5,4/5,6/2}{
\draw[blue,->,>=latex](\x+1,\y+0.75,0)--(\x+0.75,\y+0.75,0)--(\x+0.75,\y+0.25,0)--(\x+0.25,\y+0.25,0)--(\x+0.25,\y+1,0);
}
\foreach \x\y in {2/3,5/3}{
\draw[blue,->,>=latex](\x+1,\y+0.75,0)--(\x+0.75,\y+0.75,0)--(\x+0.75,\y+0.25,0)--(\x+0.25,\y+0.25,0)--(\x+0.25,\y+0.75,0)--(\x,\y+0.75,0);
}
\foreach \x\y in {1/3,2/5,6/5,4/3,4/1}{
\draw[blue,->,>=latex](\x+1,\y+0.75,0)--(\x+0.25,\y+0.75,0)--(\x+0.25,\y+0.25,0)--(\x+0.75,\y+0.25,0)--(\x+0.75,\y,0);
}
\foreach \x\y in {1/2,3/5,7/5,7/2,5/1}{
\draw[blue,->,>=latex](\x+0.75,\y+1,0)--(\x+0.75,\y+0.25,0)--(\x+0.25,\y+0.25,0)--(\x+0.25,\y+0.75,0)--(\x,\y+0.75,0);
}
\foreach \x\y in {3/6,7/6,7/3}{
\draw[blue,->,>=latex](\x+0.75,\y+1,0)--(\x+0.75,\y+0.75,0)--(\x+0.25,\y+0.75,0)--(\x+0.25,\y+0.25,0)--(\x+0.75,\y+0.25,0)--(\x+0.75,\y,0);
}
\foreach \x\y in {5/5,6/3}{
\draw[blue,->,>=latex](\x+0.25,\y,0)--(\x+0.25,\y+0.25,0)--(\x+0.75,\y+0.25,0)--(\x+0.75,\y+0.75,0)--(\x,\y+0.75,0);
}

\foreach \x\y in {0/0,0/3,0/6,4/6}{
\draw[blue,->,>=latex](\x+0.25,\y,0)--(\x+0.25,\y+0.25,0)--(\x+0.75,\y+0.25,0)--(\x+0.75,\y+0.75,0)--(\x+0.25,\y+0.75,0)--(\x+0.25,\y+1,0);
}
\end{scope}}
\foreach \z in {0,-2,-4,-6}{
\foreach \a in {0.5,1.5,...,3.5}{
\foreach \b in {0.5,1.5}{
\draw (\a,\b,\z) node{$\bullet$};
\draw[dotted,very thick] (0.5,0.5,\z)--(0.5,1.5,\z)--(1.5,1.5,\z)--(1.5,0.5,\z)--(2.5,0.5,\z)--(3.5,0.5,\z)--(3.5,1.5,\z)--(2.5,1.5,\z);
}}}
\draw (0.5,0.5,-8) node[below] {$\Sigma$};

\end{scope}
\draw(6,-6) node {Pattern of $\T_{F'}$};
\draw(-4.5,-6) node {Pattern of $\T_{F}$};

\end{tikzpicture}
\end{center}
\caption{Pattern of $\T_F$ wrapped following the tangled grid in a pattern of $\T_{F'}$.}\label{figure:WrappedPattern}
\end{figure}
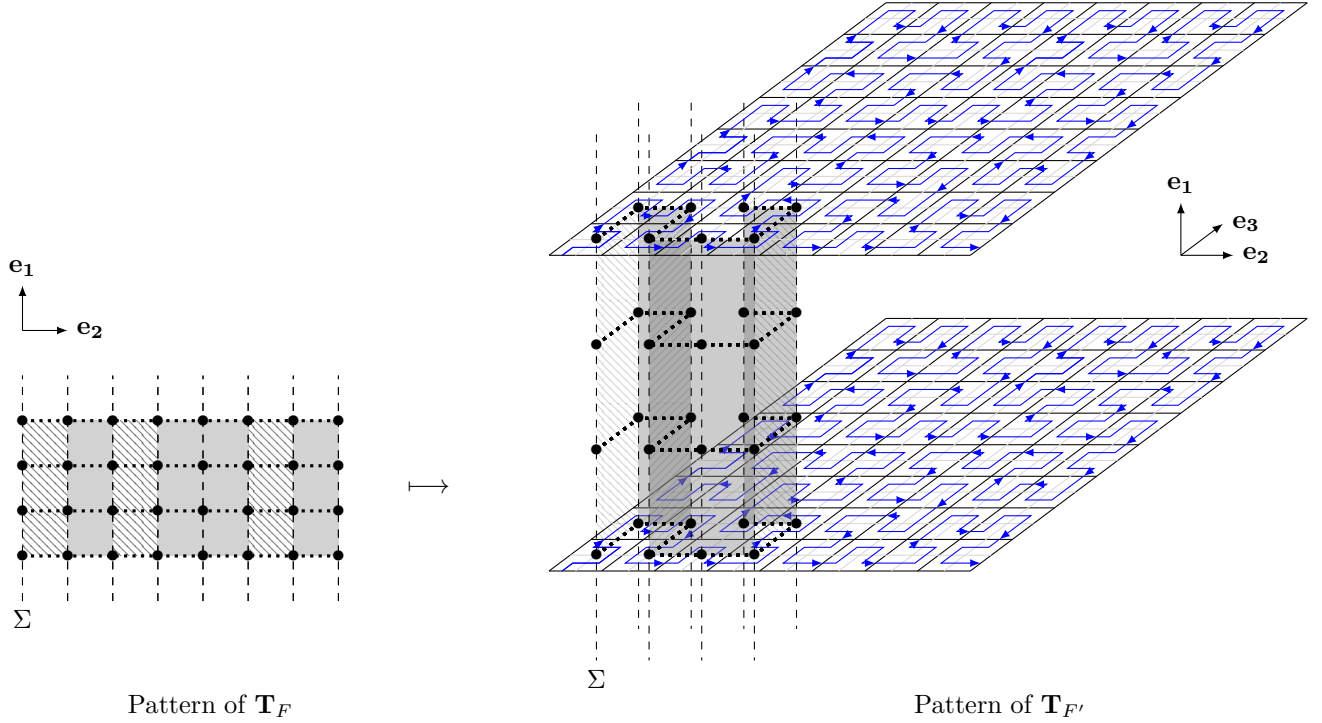}

\paragraph{$\pi'(\T_{F'})$ has the expected speed of convergence.}\  By definition of the speed of convergence, for any $u\in\A^k$, if $u\notin\Sigma$ then $u\notin\Lang\left(\pi'\left(\sa{\vect{e_1}\Z}{\T^{\varphi(k),2\to1}_{F}}\right)\right)$. 

Let $z\in\sa{\vect{e_1}\Z}{\T^{2\left\lceil\sqrt{\varphi(k)}\right\rceil,3\to1}_{F'}}$ and consider the pattern $p=\pi_s'(z)_{\{0\}\times\left[-2\left\lceil\sqrt{\varphi(k)}\right\rceil,2\left\lceil\sqrt{\varphi(k)}\right\rceil\right]^2}$. Since the condition $F_s$ is verified on $z_{\{0\}\times\left[-2\left\lceil\sqrt{\varphi(k)}\right\rceil,2\left\lceil\sqrt{\varphi(k)}\right\rceil\right]^2}$, we deduce that $z_{\vect{0}}$ is included in a super-tile of order $n=\left\lfloor\log_2\left(\left\lceil\sqrt{\varphi(k)}\right\rceil\right)\right\rfloor$, that is to say there exist $c\in\mathcal{C}$ and $\vect{i}\in\Z^3$ such that $s^n(c)=\pi_s\circ\pi_2(z)_{\vect{i}+\{0\}\times[1,2^n]^2}$ and $\vect{0}\in\vect{i}+\{0\}\times[1,2^n]^2$. Moreover for $\vect{j}\in\{(0,0,0),(0,2^n,0),(0,-2^n,0),(0,0,2^n),(0,0,-2^n)\}$, the paterns $\pi_s'(z)_{\vect{i}+\vect{j}+\{0\}\times[1,2^n]^2}$ are well assembled super-tiles of order $n$, that is to say there exists $x\in\T_s$ such that $\pi_s\circ\pi_2(z)_{\mathbb{U}}\sqsubset x$ where $\mathbb{U}=\vect{i}+\{(0,0,0),(0,2^n,0),(0,-2^n,0),(0,0,2^n),(0,0,-2^n)\}+\{0\}\times[1,2^n]^2$.

One deduces that $\pi_s\circ\pi_2(z)_0$ is in the center of a segment of blue line of amplitude $\left(2^{\left\lfloor\log_2\left(\left\lceil\sqrt{\varphi(k)}\right\rceil\right)\right\rfloor}\right)^2$. According to the conditions imposed by $F'_3$, there exists $y\in\sa{\vect{e_1}\Z}{\T^{\varphi(k),2\to1}_{F}}$ such that $y_{(i,j)}=z_{(i,\textrm{Coor}_x(j))}$ for all $i\in\Z\times[-\varphi(k),\varphi(k)]$. Thus $u\notin\Lang\left({\T^{2\left\lceil\sqrt{\varphi(k)}\right\rceil,3\to1}_{F'}}\right)$ that is to say $\varphi_{F',\pi',3\to1}\prec\sqrt{\varphi}$.  

In the same way the reverse holds and so $\varphi_{F',\pi',3\to1}\succ\sqrt{\varphi}$.

The same proof holds when we increase the dimension $d$ of $\Sigma$ since we do not use explicitly the direction $\vect{e_1}$. To increase the co-dimension $d''$ it is sufficient to generalize the $2\times2$ substitution in a $d''$-dimensional substitution which codes an Hilbert's curve covering a $d'$-dimensional vector space. Finally to increase the co-dimension $d'$ of the initial sofic which simulates $\Sigma$, one uses a substitution which codes an Hilbert's hyper surface of dimension $d'$ which cover a $d'$-dimensional vector space.
\end{proof}

%%%%%%%%%%%%%%%%%%%
%
%%%%%%%%%%%%%%%%%%%

\section{Lower bounds for the speed of convergence of a subshift}\label{section:LowerBound}

In this section we exhibit general principles to know whether the bound obtained by a given construction is sharp. 

\subsection{Combinatory lower bounds}

\begin{comment}
\begin{definition}
Let $\Sigma$ be a one dimensional subshift and let $u\in\A^{\ast}$, the \define{folder set} of $u$ is $\folder_{\Sigma}(u)=\{v\in\Lang(\Sigma) : uv\in\Lang(\Sigma)\}$. 
\end{definition}

\begin{remark}
If $u\notin\Lang(\Sigma)$ then $\folder(u)=\emptyset$. Moreover one has $\card(\{\folder_{\Sigma}(u) : u\in\A^k\})\leq \card(\A)^k$. 
\end{remark}

\begin{theorem}\label{theorem.folower}
Let $\Sigma\subset\A^{\Z}$ be a one dimensional subshift realized by projective subaction of the SFT $\T_F$. Then $\varphi_{F,d\to1}(k)\succ\left(\log(\card(\{\folder_{\Sigma}(u) : u\in\A^k\}))\right)^{\frac{1}{d}}$ 
\end{theorem}
\begin{proof}
Let $\Sigma\subset\A^{\Z}$ be a one dimensional subshift realized by projective subaction of the SFT $\T_F$ and let $r$ be the maximal size of elements of $F$. Consider $\varphi_{F,d\to1}(k)=\min\left\{ n\in\N : \aprox{k}{\Sigma} \supset \sa{\G}{\T_F^{n,d\to1}}\right\}$.

Assume there is $k$ such that $\varphi_{F,d\to1}(k)\leq$

\Claim{Let $\Sigma\subset\A^{\Z}$ be a subshift of finite type of order $r$. Then $\card(\{\folder_{\Sigma}(u) : u\in\A^k\})\leq\card(\A)^r$.}
\bclaimprf

\eclaimprf
\end{proof}
\end{comment}

\begin{definition}
Let $\Sigma$ be a one dimensional subshift and let $u\in\A^{\ast}$, the \define{folder set of word of size $k$} of $u$ is $\folder^k_{\Sigma}(u)=\{v\in\Lang_k(\Sigma) : uv\in\Lang(\Sigma)\}$. 
\end{definition}

\begin{remark}
If $u\notin\Lang(\Sigma)$ then $\folder^k_{\Sigma}(u)=\emptyset$. Moreover one has $\card(\{\folder^{k_2}_{\Sigma}(u) : u\in\A^{k_1}\})\leq \card(\A)^{k_1}$. 
\end{remark}

\begin{theorem}\label{theorem:Folower}
Let $\Sigma\subset\A^{\Z}$ be an one dimensional effective subshift and  $\varphi\in\cF^{\sofic}_{\Sigma,d}$. Then there exists a constant $M$ such that for all $k_1,k_2\in\N$ one has: 
\[M\varphi_{F,\pi,d\to1}(k_1+k_2)\geq\left(\log(\card(\{\folder^{k_2}_{\Sigma}(u) : u\in\A^{k_1}\}))\right)^{\frac{1}{d-1}}.\]
\end{theorem}
\begin{proof}
Let $\T_F\subset\B^{\Z^d}$ be a subshift of finite type of forbidden patterns $F$ and a morphism $\pi:\A^{\Z^d}\to\B^{\Z^d}$ such that $\Sigma=\sa{1}{\pi(\T_F)}$ and $\varphi=\varphi_{F,\pi,d\to1}$. For $u\in\Lang_{k_1}(\Sigma)$, one has
\[\folder^{k_2}_{\Sigma}(u)= \left\{\sa{1}{\pi(x)}_{[0,k_2-1]}\in\A^{k_2}:x\in \T^{\varphi(k_1+k_2),d\to1}_F \textrm{ such that  }\sa{1}{\pi(x)}_{[-k_1,-1]}=u\right\}.\]

Let $r$ such that every support of pattern of $F$ is included in $[0,r-1]^d$ modulo a translation. For $x\in\T^{\varphi(k_1+k_2),d\to1}_F$  such that  $\sa{1}{\pi(x)}_{[-k_1,-1]}=u\in\A^{k_1}$, the knowledge of $x_{[-r,-1]\times[-\varphi(k_1+k_2),\varphi(k_1+k_2)]^{d-1}}$ is suffisent to determine which set of $\{\folder^{k_2}_{\Sigma}(u'):u'\in\A^{k_1}\}$ is allowed to complete $u\in\A^{k_1}$ by a word $v\in\A^{k_2}$ such that $uv\in\Lang_{k_1+k_2}(\sa{1}{\pi(\T^{\varphi(k_1+k_2),d\to1}_F})=\Lang_{k_1+k_2}(\Sigma)$ (see Figure~\ref{figure:follower}). Thus
\[\card(\{\folder^{k_2}_{\Sigma}(u) : u\in\A^{k_1}\}))\leq \card\left(\left\{x_{[-r,-1]\times[-\varphi(k_1+k_2),\varphi(k_1+k_2)]^{d-1}}:x\in\T^{\varphi(k_1+k_2),d\to1}_F\right\}\right)\leq\B^{r(2\varphi(k_1+k_2))^{d-1}}.\]

The result follows by taking the logarithm.
\figex{
\begin{figure}[h!]
\begin{center}
\begin{tikzpicture}[math3dfig,scale=1.2]

\fill[gray!10!white] (-0.5,-3,3)--(-0.5,3,3)--(-0.5,3,-3)--(0,3,-3)--(0,-3,-3)--(0,-3,3)--cycle;
\draw[thick] (-0.5,-3,3)--(-0.5,3,3)--(-0.5,3,-3);
\draw[thick] (-0.5,-3,3)--(0,-3,3);
\draw[thick] (-0.5,3,3)--(0,3,3);
\draw[thick] (-0.5,3,-3)--(0,3,-3);
\draw[thick,fill=gray!10!white] (0,-3,-3)--(0,-3,3)--(0,3,3)--(0,3,-3)--cycle;
\foreach \x in {-2.9,-2.8,...,2.9}{
\draw[gray!30,very thin] (0,\x,-3)--(0,\x,3)--(-0.5,\x,3);
\draw[gray!30,very thin] (0,-3,\x)--(0,3,\x)--(-0.5,3,\x);
}
\foreach \x in {-0.4,-0.3,...,-0.1}{
\draw[gray!30,very thin] (\x,3,-3)--(\x,3,3);
\draw[gray!30,very thin] (\x,-3,3)--(\x,3,3);
}

\draw[->,>=latex] (0,0,0)--(0,0,3);
\draw (0,0,1.5) node[right] {$\varphi(k_1+k_2)$};

\draw[<->,>=latex] (0,3,-3.2)--(-0.5,3,-3.2);
\draw (-0.25,3,-3.2) node[below] {$r$};

\draw[<->,>=latex] (0,0,-.2)--(-3.8,0,-.2);
\draw (-1.6,0,-.2) node[below] {$k_1$};

\draw[<->,>=latex] (0,0,-.2)--(3,0,-.2);
\draw (1.5,0,-.2) node[below] {$k_2$};

\draw (0,0.1,0.1)--(0,0,0.1)--(5,0,0.1)--(5,0.1,0.1)--(0,0.1,0.1)--(0,0.1,0)--(5,0.1,0)--(5,0.1,0.1)--(5,0,0.1)--(5,0,0)--(5,0.1,0);
\draw[dotted] (0,0,0.1)--(-5,0,0.1)--(-5,0.1,0.1)--(0,0.1,0.1)--(0,0.1,0)--(-5,0.1,0);
\foreach \x in {0,0.1,...,5}{
\draw (\x,0.1,0)--(\x,0.1,0.1)--(\x,0,0.1);}
%\draw[line width=2pt] (0,0,0)--(5,0,0);
%\draw[line width=2pt,dashed] (-5,0,0)--(0,0,0);
\draw(4,0,0) node[below] {$\Sigma$};

\end{tikzpicture}
\end{center}
\caption{The knowledge of $x_{[-r,-1]\times[-\varphi(k_1+k_2),\varphi(k_1+k_2)]^{d-1}}$ suffices to determine which set of $\left\{\folder^{k_2}_{\Sigma}(u'):u'\in\A^{k_1}\right\}$ is allowed to complete $u\in\A^{k_1}$ by a word $v\in\A^{k_2}$ such that $uv\in\Lang_{k_1+k_2}(\sa{1}{\pi(\T_F^{\varphi(k_1+k_2),d\to1}})=\Lang_{k_1+k_2}(\Sigma)$.}\label{figure:follower}
\end{figure}
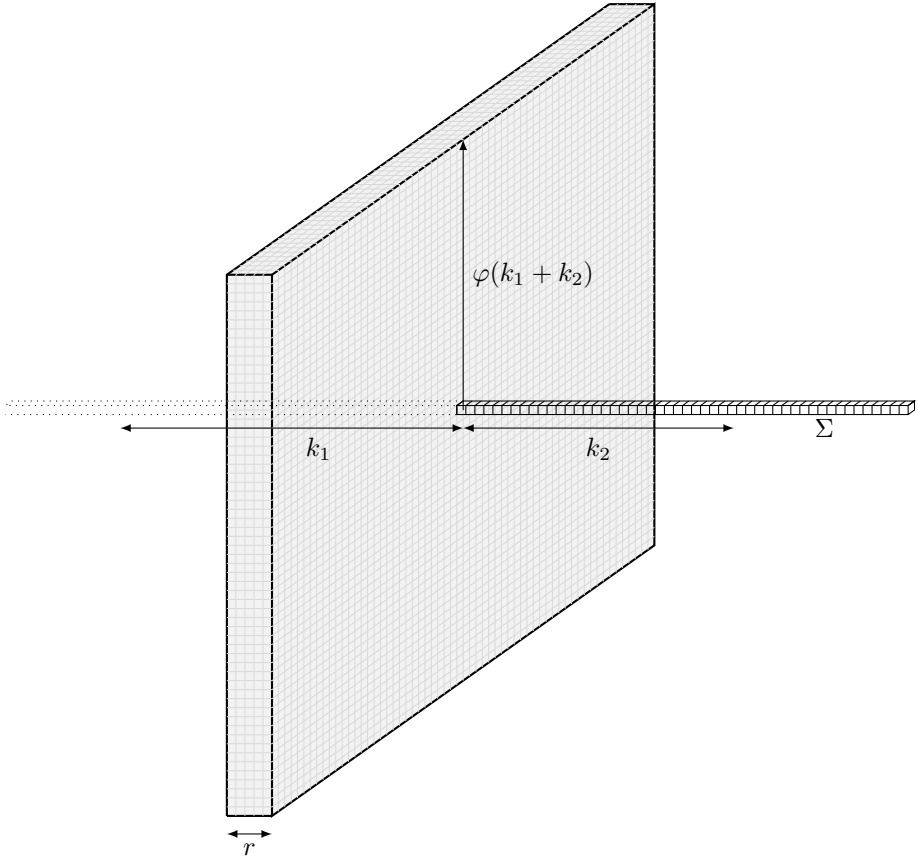}
\end{proof}

\subsection{Computational lower bounds}

In this subsection we exhibit lower bounds for the speed of convergence of an effective subshift relatively to the algorithmic complexity. 

\begin{theorem}\label{theorem.ComputationalBounds}
Let $\Sigma\subset\A^{\Z}$ be an one dimensional effective subshift and  $\varphi\in\cF^{\sofic}_{\Sigma,d}$. There exists a Turing machine $\TM$ which halts only on the entry $\Lang_{\textrm{rect}}(\Sigma)^c$ such that
\begin{itemize}
\item $\max(\log,(\varphi_{F,\pi,d\to 1})^{d-1})\succ\log\circ\Dtime_{\TM}$;
\item $(\varphi_{F,\pi,d\to 1})^{d-1}\succ\Dspace_{\TM}$.
\end{itemize}
\end{theorem}
\begin{remark}
Since $\Lang_{\textrm{rect}}(\Sigma)^c$ is not necessary recursive, $\Dtime_{\TM}$ and $\Dspace_{\TM}$ are not necessary computable.
\end{remark}
\begin{proof}
Let $\T_F\subset\B^{\Z^d}$ be a subshift of finite type of forbidden patterns $F$, where the maximal size patterns of $F$ is $r$, and a morphism $\pi:\A^{\Z^d}\to\B^{\Z^d}$ such that $\Sigma=\sa{\vect{e_1}}{\pi(\T_F)}$ and $\varphi=\varphi_{F,\pi,d\to1}$. Denote $\BB_n=\{k_2\vect{e_2}+\dots+k_d\vect{e_d}:(k_2,\dots,k_d)\in[-n,n]^{d-1}\}$ and $\T^{\varphi(k)}=\T_F^{\varphi(k),d\to1}$. One has  $\Lang_k\left(\sa{\Z\vect{e_1}}{\pi(\T^{\varphi(k)})}\right)=\Lang_k(\Sigma)$ and $\T^{\varphi(k)}\subset\left(\B^{\BB_{\varphi(k)}}\right)^{\Z}$ is a subshift of finite type of order $r$. This subshift can be represented by a graph where the vertices are $\left(\B^{\BB_{\varphi(k)}}\right)^r\cap\Lang(\T^{\varphi(k)})$ and there is an edge between $u_1\dots u_r\in\left(\B^{\BB_{\varphi(k)}}\right)^r$ and $v_1\dots v_r\in\left(\A^{\BB_{\varphi(k)}}\right)^r$ if $u_1\dots u_rv_r=u_1v_1\dots v_r\in\left(\B^{\BB_{\varphi(k)}}\right)^{r+1}\cap\Lang(\T^{\varphi(k)})$~\cite{Lind-Marcus-1995}. Thus this graph has at most $\card(\B)^{r(2\varphi(k))^{d-1}}$ vertices and can be viewed as an automaton which accepts words of $\Lang(\pi(\T^{\varphi(k)}))$ if we have eliminated components which cannot be prolonged infinitely (this takes a linear time in the size of the graph). 

Thus, to determine if $u\in\A^k$ is not in the language of $\Sigma$, it is sufficient to show that it is not in the language of $\sa{\vect{e_1}}{\pi(\T^m)}$ for some $m\in\N$. To do that we implement an algorithm which progressively, for each $m\in\N$,  explores the graph generated by $\T^m$ described previously, eliminates non-infinite component and search if $u$ is accepted with the corresponding automaton. There exists a constant $M>0$ such that at each steep, one knows if $u\in\Lang(\sa{\vect{e_1}}{\pi(\T^m)})$ in time less than $k\times\card(\A)^{r(2m)^{d-1}}$. Thus the corresponding Turing machine halts on $u\notin\Lang(\Sigma)$ in time \[\Dtime_{\TM}(k)\leq M\, k\sum_{m=1}^{\varphi(k)}\card(\A)^{r(2m)^{d-1}}\leq M\, \varphi(k)\card(\A)^{r(2\varphi(k))^{d-1}}\]
It follows that $\max(\log,\varphi^{d-1})\succ\log\circ \Dtime_{\TM}$.We deduce the first point of the theorem.

To prove the second point, the naive procedure to find a configuration of $\sa{\vect{e_1}}{\pi(\T^m)}$ which contains $u$ in the center is to start from an element of $\left(\B^{\BB_n}\right)^r\cap\Lang(\T^m)$ and complete it respecting the condition $F$ until to find again the same element on the left and on the right. To be sure to explore all the orbits it is possible to order them lexicographically. Thus, the algorithm just need to know the last orbit checked, this need $r(2m)^{d-1}$ space to know if $u\in\Lang(\sa{\vect{e_1}}{\pi(\T^m)})$. If $u\notin\Lang(\Sigma)$, the algorithm halts when we explore $\left(\B^{\BB_n}\right)^r\cap\Lang(\T^{\varphi(k)})$. So there exists $M>0$ such that $M(\varphi(k))^{d-1}\geq\Dspace_{\TM}(k)$. We recall that the word $u$ is written on an annex tape which is only used for the reading and which is not counted in $\Dspace_{\TM}$.
\end{proof}

\begin{remark}
These theorems do not generalize to dimension 2: Theorem~\ref{theorem:Folower} uses a characterization of one dimensional sofic subshifts with folder sets and Theorem~\ref{theorem.ComputationalBounds}  is blocked by the undecidability of emptiness of two-dimensional subshifts of finite type.  
\end{remark}

%%%%%%%%%%%%%%%%%%%
%
%%%%%%%%%%%%%%%%%%%

\section{Some classes of speed of convergence}\label{section:example}

In this section we present some examples of one-dimensional subshifts and give the sharp realization by SFT or sofic subshift. We do not give formal proof but just some elements to understand the different examples. Example~\ref{ex.constant} recalls a result of~\cite{Pavlov-Schraudner-2010}, adapted with the formalism of speed of convergence, one obtains a characterization of one-dimensional subshifts with a constant sharp realization. Proposition~\ref{loglog} specifies that there is no one-dimensional subshift for which the speed of the sharp realization is between constant and $n\mapsto\log(\log(n))$. Examples~\ref{ex.log},~\ref{ex.log2} and~\ref{ex.log3} exhibit subshifts with the speed of convergence of sharp realization $n\mapsto\log(n)$ and Example~\ref{ex.lin} exhibit one with sharp realization  $n\mapsto n$. Example~\ref{ex.dtime} exhibit subshifts where the speed of convergence of sharp realization is given by the space necessary by a Turing machine to compute a language. To finish, Example~\ref{ex.noncomput} gives a subshift where the set of speed of convergence does not contain any recursive function.

\begin{example}\label{ex.constant}
A subshift is constant-realizable by SFT if and only if it is a stable sofic. A subshift is constant-realizable by sofic if and only if it is a sofic. These results are detailed in~\cite{Pavlov-Schraudner-2010}.
\end{example}

\begin{proposition}\label{loglog}
If a one-dimensional subshift $\Sigma$ is $\varphi$-realizable by sofic with $\varphi\in o(\log(\log(n)))$ then $\Sigma$ is sofic, so it is constant-realizable by sofic.
\end{proposition}
\begin{proof}
Assume that a subshift $\Sigma$ is $\varphi$-realizable by sofic with $\varphi\in o(\log(\log(n)))$, then by Theorem~\ref{theorem.ComputationalBounds} one deduces that there exists a Turing machine $\TM$ which defines $\Lang_{\textrm{rect}}(\Sigma)^c$ in space $o(\log(\log(n)))$. Thus $\Lang_{\textrm{rect}}(\Sigma)^c$ is rational (see~\cite{Hopcroft-Ullman-1969}), that is to say $\Sigma$ is sofic. 
\end{proof}

Let $\Lang\subset \A^{\ast}$ be a language and the symbol $\$$ does not appear in $\A$. Define the subshift $\TT{\Lang}=\T_{F_{\Lang}}\subset \A'^{\Z}$ where $\A'=\A\cup\{\$\}$ and $F=\{\$u\$ : u\notin\Lang\}$. If $\Lang$ is effective then $\TT{\Lang}$ is an effective subshift. Moreover it contains a transitive sofic subshift of positive entropy, so it verifies the condition $\widetilde{\mathcal{H}}$ (see section~\ref{section:RealizationSFTvsSofic}), thus it is sufficient to study realization by projective subaction of sofic. 

\begin{example}\label{ex.log}
Consider $\Lang_= =\{a^nb^n:n\in\N\}$. Then for all $k\in\N$, $\TT{\Lang_=}\subset\{a,b,\$\}^{\Z}$ is sharp $\log,2$-realizable by subshift of finite type and sofic.
 
Theorem~\ref{theorem:Folower} gives an upper bound. For the lower bound, consider the subshift of finite type $\T_{F'}\subset\{a,b,\$,0_a,1_a,\emptyset_a,0_b,1_b,\emptyset_b\}^{\Z^2}$ where $F'$ are the forbidden patterns of shape $\U=\vbox to 8pt{\hbox to 16pt{\begin{tikzpicture}[scale=0.2]\foreach \x\y in {-1/0,0/0,1/0,0/-1}{\draw[thin,black!50] (\x,\y) rectangle (\x+1,\y+1);}\end{tikzpicture}}}$ which do not appear in the configuration represented in Figure~\ref{figure.ConfigEgal}. The factor $\pi$ maps $\$$ on $\$$, $\{0_a,1_a,\emptyset_a\}$ on $a$ and $\{0_b,1_b,\emptyset_b\}$ on $b$. The principe is to implement a counter and compare it at the frontier.

\begin{figure}[h!]
\begin{center}
\begin{tikzpicture}[scale=0.5]
\clip[draw,decorate,decoration={random steps, segment length=3pt, amplitude=1pt}] (0.2,0.2) rectangle (27.8,8.8);

\foreach \x in {0,1,...,27}{
\foreach \y in {0,1,...,8}{
\draw[dashed,black!50] (\x,0)--(\x,9);
\draw[dashed,black!50] (0,\y)--(28,\y);
\draw (\x+0.5,\y+0.5) node {$\$$};
}}
\foreach \x\y in {2,3,...,13}{
\fill [white] (\x+0.1,2.1) rectangle (\x+0.9,2.9);
\fill [opacity=0.2,blue] (\x,2) rectangle (\x+1,3);
\draw (\x+.5,2.5) node{$a$};}

\foreach \x in {14,15,...,25}{
\fill [white] (\x+0.1,2.1) rectangle (\x+0.9,2.9);
\fill [opacity=0.2,red] (\x,2) rectangle (\x+1,3);
\draw (\x+.5,2.5) node{$b$};
}

\foreach \x\y in {8/4,12/4,13/5}{
\fill [white] (\x+0.1,\y+0.1) rectangle (\x+0.9,\y+0.9);
\fill [opacity=0.2,blue] (\x,\y) rectangle (\x+1,\y+1);
\draw (\x+.5,\y+.5) node{$0_a$};
}

\foreach \x\y in {19/4,15/4,14/5}{
\fill [white] (\x+0.1,\y+0.1) rectangle (\x+0.9,\y+0.9);
\fill [opacity=0.2,red] (\x,\y) rectangle (\x+1,\y+1);
\draw (\x+.5,\y+.5) node{$0_b$};
}

\foreach \x\y in {4/3,6/3,7/4,8/3,10/3,11/4,12/3,12/5}{
\fill [white] (\x+0.1,\y+0.1) rectangle (\x+0.9,\y+0.9);
\fill [opacity=0.2,blue] (\x,\y) rectangle (\x+1,\y+1);
\draw (\x+.5,\y+.5) node{$\emptyset_a$};
}

\foreach \x\y in {23/3,21/3,20/4,19/3,17/3,16/4,15/3,15/5}{
\fill [white] (\x+0.1,\y+0.1) rectangle (\x+0.9,\y+0.9);
\fill [opacity=0.2,red] (\x,\y) rectangle (\x+1,\y+1);
\draw (\x+.5,\y+.5) node{$\emptyset_b$};
}

\foreach \x\y in {3/3,5/3,5/4,6/4,7/3,8/5,9/3,9/4,9/5,10/4,10/5,11/3,11/5,13/3,13/4,13/6}{
\fill [white] (\x+0.1,\y+0.1) rectangle (\x+0.9,\y+0.9);
\fill [opacity=0.2,blue] (\x,\y) rectangle (\x+1,\y+1);
\draw (\x+.5,\y+.5) node{$1_a$};
}

\foreach \x\y in {24/3,22/3,22/4,21/4,20/3,19/5,18/3,18/4,18/5,17/4,17/5,16/3,16/5,14/3,14/4,14/6}{
\fill [white] (\x+0.1,\y+0.1) rectangle (\x+0.9,\y+0.9);
\fill [opacity=0.2,red] (\x,\y) rectangle (\x+1,\y+1);
\draw (\x+.5,\y+.5) node{$1_b$};
}
\end{tikzpicture}
\end{center}
\caption{A configuration of $T_{F'}$.}\label{figure.ConfigEgal}
\end{figure}
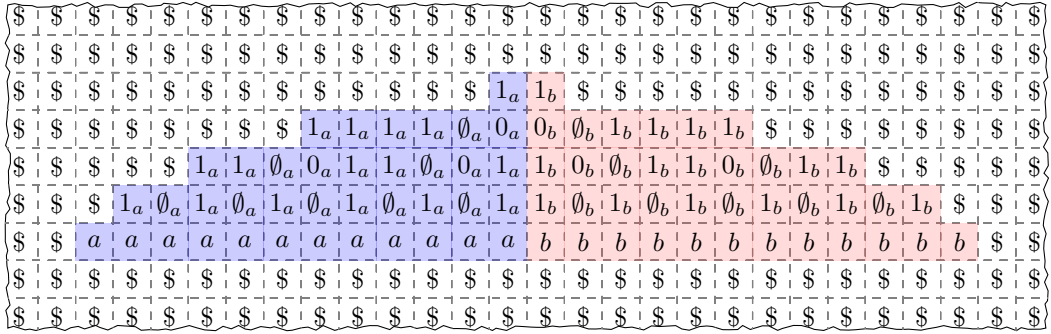

\end{example}

\begin{example}\label{ex.log2}
Consider $\Lang_2=\{a^nb^{n^2} : n\in\N \}$. Then $\TT{\Lang_2}$ is sharp $\log,2$-realizable by sofic.
\end{example}

\begin{example}\label{ex.log3}
Let $s$ be a substitution on $\A$ and consider $\Lang_s=\{s^n(a) : n\in\N \textrm{ and }a\in\A \}$. Then $\TT{\Lang_s}$ is sharp $\log,2$-realizable by sofic.
\end{example}

\begin{example}\label{ex.lin}
For a word $u\in\{0,1\}^{\ast}$, define $\overline{u}$ the miror of $u$. Consider $\Lang_{\textrm{palin}}=\{u\overline{u}:u\in\in\{0,1\}^{\ast}\}$, $\TT{\Lang_{\textrm{palin}}}\subset\{0,1,\$\}^{\Z}$ is sharp $\textrm{Id},2$-realizable by projective subaction of sofic where $\begin{array}{lrcl}\textrm{Id}:&\N&\to&\N\\& k&\mapsto& k\end{array}$.

Theorem~\ref{theorem:Folower} gives an upper bound. For the lower bound, consider the subshift of finite type $\T_{F'}\subset\{\$,0_l,1_l,0_r,1_r\}^{\Z^2}$ where $F'$ are the forbidden patterns of shape $\U=\vbox to 8pt{\hbox to 16pt{\begin{tikzpicture}[scale=0.2]\foreach \x\y in {0/0,1/0,1/1}{\draw[thin,black!50] (\x,\y) rectangle (\x+1,\y+1);}\end{tikzpicture}}}$ or $\vbox to 8pt{\hbox to 16pt{\begin{tikzpicture}[scale=0.2]\foreach \x\y in {0/0,0/1,1/0}{\draw[thin,black!50] (\x,\y) rectangle (\x+1,\y+1);}\end{tikzpicture}}}$ which do not appear in the configuration represented in Figure~\ref{figure.ConfigEgal}. The factor $\pi$ maps $\$$ on $\$$, $\{0_l,0_r\}$ on $0$ and $\{1_l,1_r\}$ on $1$. The principe is to compare the beginning and the end of a word written in $\{0,1\}^{\ast}$.

\begin{figure}[h!]
\begin{center}
\begin{tikzpicture}[scale=0.5]
\clip[draw,decorate,decoration={random steps, segment length=3pt, amplitude=1pt}] (0.2,0.2) rectangle (17.8,11.8);

\foreach \x in {0,1,...,18}{
\foreach \y in {0,1,...,11}{
\draw[dashed,black!50] (\x,0)--(\x,12);
\draw[dashed,black!50] (0,\y)--(18,\y);
\draw (\x+0.5,\y+0.5) node {$\$$};
}}

\foreach \x\y in {1/2,2/2,4/2,7/2,2/3,3/3,5/3,8/3,3/4,4/4,6/4,4/5,5/5,7/5,5/6,6/6,8/6,6/7,7/7,7/8,8/8,8/9}{
\fill [white] (\x+0.1,\y+0.1) rectangle (\x+0.9,\y+0.9);
\fill [opacity=0.2,blue] (\x,\y) rectangle (\x+1,\y+1);
\draw (\x+.5,\y+.5) node{$0_l$};
}

\foreach \x\y in {3/2,5/2,6/2,8/2,4/3,6/3,7/3,5/4,7/4,8/4,6/5,8/5,7/6,8/7}{
\fill [white] (\x+0.1,\y+0.1) rectangle (\x+0.9,\y+0.9);
\fill [opacity=0.2,blue] (\x,\y) rectangle (\x+1,\y+1);
\draw (\x+.5,\y+.5) node{$1_l$};
}

\foreach \x\y in {16/2,15/2,13/2,10/2,15/3,14/3,12/3,9/3,14/4,13/4,11/4,13/5,12/5,10/5,12/6,11/6,9/6,10/7,11/7,9/8,10/8,9/9}{
\fill [white] (\x+0.1,\y+0.1) rectangle (\x+0.9,\y+0.9);
\fill [opacity=0.2,red] (\x,\y) rectangle (\x+1,\y+1);
\draw (\x+.5,\y+.5) node{$0_r$};
}

\foreach \x\y in {14/2,12/2,11/2,9/2,13/3,11/3,10/3,12/4,10/4,9/4,11/5,9/5,10/6,9/7}{
\fill [white] (\x+0.1,\y+0.1) rectangle (\x+0.9,\y+0.9);
\fill [opacity=0.2,red] (\x,\y) rectangle (\x+1,\y+1);
\draw (\x+.5,\y+.5) node{$1_r$};
}

\end{tikzpicture}
\end{center}
\caption{A configuration of $T_{F'}$.}\label{figure.ConfigEgal}
\end{figure}
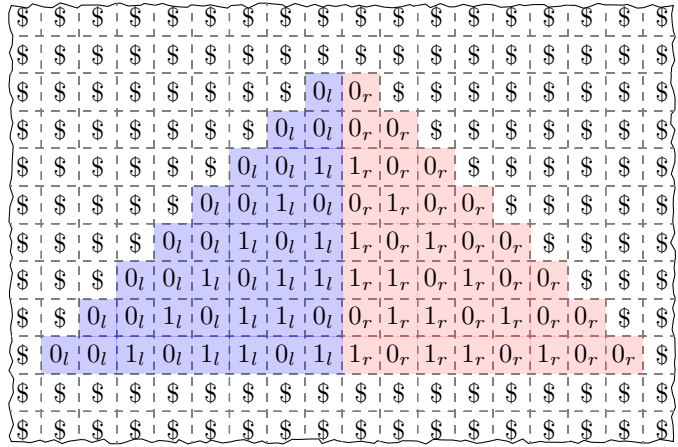
\end{example}

\begin{example}\label{ex.dtime}
If $\Lang$ is a language computable in time $\Dtime_{\TM}$ with a non-deterministic Turing machine $\TM$ then $\TT{\Lang}$ is $\Dtime_{\TM},2$-realizable.

If $\Lang$ is a language computable in time $\Dspace_{\TM}$ with a Turing machine $\TM$ then $\TT{\Lang}$ is sharp $\Dspace_{\TM},2$-realizable.
\end{example}

\begin{example}\label{ex.noncomput}
Consider an enumeration of Turing machine and \[F=\{01^n0: n\textrm{ such that the Turing machine of number $n$ halts} \}.\] 
Since $F$ is recursively enumerable, the subshift $\Sigma=\T_F$ is an effective subshift, moreover it satisfies the property $\widetilde{\mathcal{H}}$. Thus for all $\varphi\in\cF^{\sofic}_{\Sigma,2}$, $\varphi$ is larger than all recursive function. Otherwise it could be possible to decide if the Turing machine of number $n$ halts.
\end{example}

For the subshift with rational set of forbidden patterns, which corresponds to sofic subshift,  there exists a characterization with the optimal speed of convergence. There is the only characterization thanks to speed of convergence of class of effective subshift known, but it seems interesting to explore the links between dynamical properties of effective subshift and their optimal speed of convergence. 

%\section{Links between dynamical properties}

%\bibliographystyle{alpha}
%\bibliography{biblio}

\end{document}